\documentclass[a4paper, 11pt]{article}
 
\usepackage[top=2cm, bottom=2cm, left=2cm, right=2cm]{geometry}
\usepackage{graphicx, array, amsmath, amsthm, mathrsfs, framed, marvosym, amssymb, mathtools, verbatim, tikz}

\usepackage[all]{xy}
\tikzset{node distance=2cm, auto}
\usepackage{enumerate}
\usepackage{color}

\allowdisplaybreaks

\allowdisplaybreaks

\input xy
\xyoption{matrix}\xyoption{arrow}
\def\edge{\ar@{-}}
\def\dedge{\ar@{.}}

\newtheorem{thm}{Theorem}[section]
\newtheorem{pro}[thm]{Proposition}
\newtheorem{lem}[thm]{Lemma}
\newtheorem{cor}[thm]{Corollary}

\theoremstyle{definition}%ensures content of definitions, remarks, etc are displayed in roman text

\newtheorem{rem}[thm]{Remark}

\newtheorem{hypo}[thm]{Hypothesis}
\newtheorem{notn}[thm]{Notation}

\newcommand{\zbar}{Z}

\usepackage{version}% use this to include or exclude text between \begin{full}---\end{full} 

\includeversion{full} 
\title{Poisson derivations of a semiclassical limit of  a family of quantum second Weyl algebras} 

\author{S Launois\footnote{The first-named author's research was partly supported by EPSRC grant EP/R009279/1.}~~and~ I Oppong\footnote{The second-named author's research was partly supported by EPSRC grant EP/W522454/1.}}

\begin{document}
\maketitle
\begin{abstract}
In \cite{lo}, we studied deformations $A_{\alpha,\beta}$ of the second Weyl algebra and  computed their derivations. In the present paper, we identify the semiclassical limits $\mathcal{A}_{\alpha,\beta}$ of these deformations and compute their Poisson derivations. Our results show that the first Hochschild cohomology group HH$^1(A_{\alpha,\beta})$ is isomorphic to the first Poisson cohomology group HP$^1( \mathcal{A}_{\alpha,\beta}).$
\end{abstract}

%%%%%%%%%%%%%%%%%%%%%%%%%%%%%%%%%%%%%%%%%%%%%%%%%%%%%%%%%%%%%%
%%%%%%%%%%%%%%%%%%%%%%%%%%%%%%%%%%%%%%%%%%%%%%%%%%%%%%%%%%%%%%
%%%%%%%%%%%%%%%%%%%%%%%%%%%%%%%%%%%%%%%%%%%%%%%%%%%%%%%%%%%%%%
\section{Introduction} 
Throughout this study, $\mathbb{K}$ will denote  a field with characteristic zero. 

Given a non-commutative algebra $A$ and its semiclassical limit $\mathcal{A},$ an intriguing question has always been ``Do the properties of $A$ always reflect the (Poisson) properties of $\mathcal{A}$?" For example, given the centre, automorphisms, endomorphisms, derivations, and prime ideals of $A$, can one successfully  predict/conjecture the Poisson centre, Poisson automorphisms, Poisson endomorphisms, Poisson derivations and Poisson prime ideals of $\mathcal{A}$? Suppose for instance that the prime spectrum of $A$ reflects the Poisson prime spectrum of $\mathcal{A}.$ 
A follow-on question will be whether they are homeomorphic/isomorphic? 
The answers to these questions, in some specific cases and for some specific algebras, are affirmative.  For example, Goodearl \cite{gscl} has conjectured that the prime and  primitive spectra of the quantized coordinate rings are respectively homeomorphic to the Poisson prime and Poisson primitive spectra of their corresponding semiclassical limits when the base field is algebraically closed and of characteristic zero. This conjecture has been verified for the following quantized coordinate rings: $\mathcal{O}_q(\mathbb{K}^n)$ (see \cite[Theorem 4.1]{kenL}), $\mathcal{O}_q(SL_2(\mathbb{K}))$ (see \cite[Example 9.7]{gscl}), $\mathcal{O}_q(SL_3(\mathbb{K}))$ (see \cite[Theorem 5.21 \& Corollary 5.22]{fryer}) and $\mathcal{O}_{q}(GL_2)$ (see \cite[Corollary 5.23]{fryer}). Goodearl further established that the prime and primitive spectra of the  enveloping algebra $U(\mathfrak{g})$ are respectively homeomorphic  to the prime and primitive spectra of its semiclassical limit (see \cite[Theorem 8.11, Example 2.6]{gscl}). From \cite{cho}, we also have that  the Poisson endomorphisms of the Poisson quantum generalized Weyl algebra are precisely the Poisson analogue of the endomorphisms of the quantum generalised Weyl algebra. Belov-Kanel and Kontsevich \cite{bkk} have  also conjectured that the group of automorphisms of an $n^{\text{th}}$-Weyl algebra $A_n(\mathbb{K})$ is isomorphic to the group of Poisson automorphisms of the corresponding Poisson Weyl algebra in characteristic zero. 

In \cite{lo}, we studied a family of simple quotients $$A_{\alpha,\beta}:=U_q^+(G_2)/\langle \Omega_1-\alpha, \Omega_2-\beta\rangle \ \ \ \ (\alpha,\beta)\in \mathbb{K}^2\setminus \{(0,0)\})$$ of the positive part of the quantized enveloping algebra $U_q^+(G_2)$,  and concluded that the algebra $A_{\alpha,\beta}$ is a $q$-deformation of a quadratic extension of the second Weyl algebra $A_2(\mathbb{K}).$ Since $A_{\alpha,\beta}$ deforms approximately to $A_2(\mathbb{K}),$ it is considered as a \textit{quantum second Weyl algebra.}  Our goal here is to study a semiclassical limit $\mathcal{A}_{\alpha,\beta}$ of $A_{\alpha,\beta},$ and compare its Lie algebra of Poisson derivations to the Lie algebra of derivations of $A_{\alpha,\beta}$ studied in \cite{lo}. Another property that is also worth investigating is the Belov-Kanel and Kontsevich  conjecture [ibid]. Thus, it is natural to ask if the automorphism group of $A_{\alpha,\beta}$ is isomorphic to the Poisson automorphism group of $\mathcal{A}_{\alpha,\beta}$. We will return to this problem in the near future after we have successfully studied the automorphism group of $A_{\alpha,\beta}.$ In the present case and as already mentioned, we only focus on studying the Lie algebra of Poisson derivations of $\mathcal{A}_{\alpha,\beta}$,  and comparing them to their non-commutative counterparts in \cite{lo}. 

In the noncommutative world, the knowledge of the derivations of twisted group algebras, studied by Osborn and Passman \cite{op}, has helped in  studying the derivations of other non-commutative algebras such as the quantum second Weyl algebra (see \cite{lo}), quantum matrices (see \cite{sltl}), generalized Weyl algebras (see \cite{ak}) and some specific examples of quantum enveloping algebras (see \cite{ss}, \cite{xt}, and \cite{zt}). In view of this, we also study the Poisson derivations of the Poisson analogue of the twisted group algebras---called Poisson group algebras---and apply the results to study the Poisson derivations of a semiclassical limit $\mathcal{A}_{\alpha,\beta}$ of $A_{\alpha,\beta}$. The rest of the paper is organised as follows. 

In Section \ref{s2}, we recall some basics  on Poisson algebras and semiclassical limit. We then proceed to study the Poisson derivations of the Poisson group algebras. Similarly to their non-commutative counterparts in \cite{op}, every Poisson derivation of a Poisson group algebra is the sum of an inner Poisson derivation and a central/scalar Poisson derivation.

In Section \ref{sec2},  we study a semiclassical limit $\mathcal{A}$ of the quantum algebra $U_q^+(G_2)$ and establish that $\mathcal{A}$ is a Poisson polynomial $\mathbb{K}$-algebra generated by six indeterminates $X_1,\ldots, X_6.$ Since $\mathcal{A}=\mathbb{K}[X_1,\ldots, X_6]$ satisfies the conditions in \cite[Hypothesis 1.7]{sc}, we can apply the Poisson deleting derivations algorithm \cite{sc} to study its Poisson centre  $\mathcal{A}$:  a polynomial ring $\mathbb{K}[\Omega_1,\Omega_2]$ in two variables. In Section \ref{sec3}, we study some Poisson $\mathcal{H}$-prime ideals of 
$\mathcal{A}$ using Goodearl's $\mathcal{H}$-stratification theory \cite{gd}, and  proceed to study some family $(\langle \Omega_1-\alpha, \Omega_2-\alpha\rangle)_{(\alpha,\beta)\in \mathbb{K}^2\setminus \{(0,0)\})}$ of maximal and primitive Poisson prime ideals of $\mathcal{A}$.
Consequently, we study their corresponding Poisson simple quotients  $$\mathcal{A}_{\alpha,\beta}:=\mathbb{K}[X_1,\ldots, X_6]/\langle \Omega_1-\alpha, \Omega_2-\beta\rangle,$$  and conclude  that the Poisson algebra $\mathcal{A}_{\alpha,\beta}$ is a semiclassical limit of the quantum second Weyl algebra $A_{\alpha,\beta}.$ Having a complete description of the semiclassical limit $\mathcal{A}_{\alpha,\beta}$ of $A_{\alpha,\beta},$ we proceed to study its Poisson derivations in the final section of this paper, by following procedures similar to its non-commutative counterpart $A_{\alpha,\beta}$ (see \cite[\S 5]{lo}). That is, we successively embed $\mathcal{A}_{\alpha,\beta}$ into a suitable Poisson torus $\mathcal{R}_3$ via localization as  follows:
$$\mathcal{A}_{\alpha,\beta}=\mathcal{R}_7\subset \mathcal{R}_6=\mathcal{R}_7\Sigma_6^{-1}\subset \mathcal{R}_5=\mathcal{R}_6\Sigma_5^{-1}\subset \mathcal{R}_4=\mathcal{R}_5\Sigma_4^{-1}
 \subset \mathcal{R}_3.$$
 These embeddings and localization allow us to extend every Poisson derivation of $\mathcal{A}_{\alpha,\beta}$ successively and uniquely to a Poisson derivation of each of the Poisson algebras $\mathcal{R}_i$ through to the Poisson torus 
 $\mathcal{R}_3.$   Since a Poisson torus is an example of a Poisson group algebra, we have that every Poisson derivation of $\mathcal{R}_3$ is the sum of an inner Poisson derivation and a central/scalar Poisson derivation. 
Given the Poisson derivations of 
 $\mathcal{R}_3,$ we backwardly and successively pull the Poisson derivations of $\mathcal{R}_3$ to $\mathcal{A}_{\alpha,\beta}.$ This gives us a complete description of the Poisson derivations of $\mathcal{A}_{\alpha,\beta}.$  Similarly to their non-commutative counterparts in \cite{lo}, every Poisson derivation of $\mathcal{A}_{\alpha,\beta}$ is an inner Poisson  derivations provided $\alpha\beta\neq 0,$ and the sum of inner Poisson and scalar Poisson derivations whenever $\alpha$ or $\beta$ is zero. More precisely, the first Poisson cohomology group HP$^1( \mathcal{A}_{\alpha,\beta})$ is a one-dimensional vector space in the case where $\alpha$ or $\beta$ is zero. 
 
%%%%%%%%%%%%%%%%%%%%%%%%%%%%%%%%%%%%%%%%%%%%%%%%%%%%%%%%%%%%
%%%%%%%%%%%%%%%%%%%%%%%%%%%%%%%%%%%%%%%%%%%%%%%%%%%%%%%%%%%%
%%%%%%%%%%%%%%%%%%%%%%%%%%%%%%%%%%%%%%%%%%%%%%%%%%%%%%%%%%%% 
\section{Poisson derivations of Poisson group algebras}
\label{s2}
This section begins with a reminder about Poisson algebras and semiclassical limit. We will then proceed to introduce Poisson group algebras, and consequently, study their derivations. 
\subsection{Poisson algebras}
A \textit{Poisson algebra} $\mathcal{A}$ is a commutative algebra over  $\mathbb{K}$ endowed with a skew-symmetric $\mathbb{K}$-bilinear map $\{-,-\}:\mathcal{A}\times \mathcal{A}\longrightarrow \mathcal{A}$ which satisfies the Leibniz rule (i.e., $\{x,yz\}=\{x,y\}z+y\{x,z\};$  $x,y,z\in\mathcal{A}$)  and Jacobi identity (i.e., $\{x,\{y,z\}\}+\{y,\{z,x\}\}+\{z,\{x,y\}\}=0;$   $x,y,z\in\mathcal{A}$).

The map $\{-,-\}$ is called the \textit{Poisson bracket.}
 From \cite[Prop. 1.7]{fl}, every Poisson bracket of $\mathcal{A}$ extends uniquely to the localizations of $\mathcal{A}.$  A \textit{Poisson ideal}  of $\mathcal{A}$ is any ideal $I$ such that $\{\mathcal{A},I\}\subseteq I.$ 
Given a Poisson ideal $I$ of $\mathcal{A},$ it is well known that the quotient algebra $\mathcal{A}/I$ is a Poisson algebra with an induced Poisson bracket defined as $\{\bar{x},\bar{y}\}=\overline{\{x,y\}},$ where $\bar{x}:=x+I$ and $\bar{y}:=y+I,$ note that $x,y\in \mathcal{A}.$  
Finally, the subalgebra 
$Z_P(\mathcal{A}):=\{a\in\mathcal{A}\mid \{a,x\}=0, \ \forall x\in\mathcal{A}\}$ is called the \textit{ Poisson centre } of  $\mathcal{A}.$    
\begin{rem}
\label{ev29}
 If $\mathcal{A}$ is a Poisson algebra and $\{x_1,\ldots, x_n\}$ is a generating set for $\mathcal{A}$ (as an algebra), then 
 \begin{itemize}
\item[(1)] it is always enough to define a Poisson bracket $\{-,-\}$ on $\mathcal{A}$ by defining it on only the generating set.
\item[(2)]for all 
$f,g\in \mathcal{A},$ we have that
$\{f,g\}=\displaystyle\sum_{i,j=1}^n\{x_i,x_j\}\frac{\partial f}{\partial x_i}\frac{\partial g}{\partial x_j}$ (see \cite[Example 2.2(a)]{gscl}).
\end{itemize}
\end{rem}
\subsection{Semiclassical limit} 
Given a non-commutative algebra, one can move from the `Non-commutative World' to the `Poisson World' through a process called \textit{semiclassical limit}, and reverse this process through \textit{quantization.} This transformation (semiclassical limit) and its reverse transformation (quantization) have been widely studied (for example, see \cite[\S\S 1.1.3]{dumas}, \cite[Chapter III.5]{bg}, and \cite[\S 2]{gkls}). In line with the presentation in \cite[\S\S 1.1.3]{dumas}, we present the following overview of semiclassical limit. Let $R$ be a commutative principal ideal domain containing the field $\mathbb{K}$ and $hR$ be a maximal ideal of $R$ for a fix $h\in R.$  Let $A$ be an algebra which is not necessarily commutative torsion-free $R$-algebra such that the quotient  $\mathcal{A}:=A/hA$ is a commutative algebra.  For $u,v\in A;$ we have that $\bar{u}:=u+hA$ and $\bar{v}:=v+hA$ are their respective canonical images  in $\mathcal{A}$.  Since $\bar{u}\bar{v}=\bar{v}\bar{u},$ we have that $[u,v]:=uv-vu\in hA.$ There exists a unique element $\gamma(u,v)$ of $A$ such that $[u,v]=h\gamma(u,v).$ It follows that $$\{\bar{u},\bar{v}\}:=\gamma(u,v)+hA=\frac{[u,v]}{h}+hA$$ defines a Poisson bracket on $\mathcal{A}$  (see 
  \cite[\S\S 1.1.3]{dumas} for further details). We say that $A$ is  a \textit{quantization} of $\mathcal{A},$ and $\mathcal{A}$ is a \textit{semiclassical limit} of $A.$ Fix $\lambda\in \mathbb{K}.$ The algebra $\mathcal{A}_{\lambda}:=A/(h-\lambda)A$ is a \textit{deformation} of the Poisson algebra $\mathcal{A}=\mathcal{A}_0$  if the central element $h-\lambda$ is not invertible in $A.$ We refer the interested reader to \cite[\S 2]{gkls} for some known examples of semiclassical limit of some families of quantum algebras. 

\subsection{Introduction to Poisson group algebras}
In \cite[\S 1\&{2}]{op}, Osborn and Passman studied the derivations of twisted group algebras.  In line with their results, we also study the Poisson derivations of Poisson group algebras. The results in this section will be crucial in the final section of this paper where we study the Poisson derivations of a semiclassical limit of the quantum second Weyl algebra $A_{\alpha,\beta}.$

Let  $G$ represent a finitely generated abelian group and  $\lambda:G\times G\longrightarrow \mathbb{K}$ be a map such that $\lambda(y,x)=-\lambda(x,y)$ and $\lambda(x,yz)=\lambda(x,y)+\lambda(x,z).$     We define a \textit{Poisson group algebra} $\mathbb{K}_P^\lambda[G]$ as a commutative $\mathbb{K}$-algebra which has a copy $\overline{G}:=\{\bar{g}\mid g\in G\}$ of $G$ as a basis and  satisfies the Poisson bracket via
$\{\bar{x},\bar{y}\}=\lambda(x,y)\bar{x}\bar{y}=\lambda(x,y)\overline{xy}$ for all $x,y\in G$  (note that $\bar{x}\bar{y}=\overline{xy}$).  Observe that $\lambda(x,y) =0$ if and only if  $\{\bar{x},\bar{y}\}=0$.

As an example, take the group algebra 
$\mathbb{K}[\mathbb{Z}^2]$   generated by  $x^{\pm 1}, y^{\pm 1}$  over $\mathbb{K},$ with basis $\overline{\mathbb{Z}^2}:=\{x^iy^j\mid (i,j)\in \mathbb{Z}^2\}.$  A Poisson structure can be defined on $\mathbb{K}[\mathbb{Z}^2]$ via 
$\{x^iy^j,x^ky^l\}=\lambda((i,j),(k,l))x^{i+k}y^{j+l}$, where 
$\lambda((i,j),(k,l)):=il-jk$, to obtain a rank $2$ Poisson torus $\mathbb{K}_P^{\lambda}[\mathbb{Z}^2].$  
In general,  $\mathbb{K}_P^\lambda[\mathbb{Z}^n]$ is a Poisson torus of rank $n$ over the field $\mathbb{K}$ for some  $\lambda:\mathbb{Z}^n\times\mathbb{Z}^n\rightarrow \mathbb{K},$ where $\mathbb{Z}^n$ is the usual additive group.

Let $\gamma\in \mathbb{K}_P^\lambda[G].$  One can write $\gamma$ as $\gamma=\sum_{g}c_g\bar{g}$ with 
$g\in G$ and $c_g\in \mathbb{K}.$ Note that $c_g=0$ for almost all $g$.  The set supp$(\gamma):=\{g\in G\mid
 c_g\neq 0$ in $\gamma\}$  is called the \textit{suppor}t of $\gamma.$
Furthermore, the set 
$C:=\{g\in G\mid \{\bar{g},\bar{x}\}=0$ for all $ x\in G\}$ and 
$\Delta(\bar{x}):=\{g\in G\mid \{\bar{g},\bar{x}\}=0\}$  are both subgroups of $G.$ The following remark establishes a relationship between these two subgroups.
\begin{rem}
Let   $\{g_1,\ldots,g_n\}$ be a generating set for the group $G.$  Then,  $C=\bigcap_{i=1}^n\Delta(g_i).$
\end{rem}
 
\begin{lem}
The Poisson centre $Z_P(\mathbb{K}_P^\lambda[G])$ of $\mathbb{K}_P^\lambda[G]$ is $\mathbb{K}_P^\lambda[C]$. 
\end{lem}
\begin{proof}
Clearly, $\mathbb{K}_P^\lambda[C]\subseteq Z_P(\mathbb{K}_P^\lambda[G]).$  For the reverse inclusion, take $\gamma=\sum_{g}c_g\bar{g}\in Z_P(\mathbb{K}_P^\lambda[G]).$  It follows that $0=\{\gamma,\bar{x}\}=\sum_{g}c_g\{\bar{g},\bar{x}\}=\sum_{g}c_g\lambda(g,x)\bar{g}\bar{x}$ for any $x\in G.$ Consequently, $\lambda(g,x)=0$ for all $g\in$ supp$(\gamma).$ This implies that
supp$(\gamma)\subseteq C,$ hence $\gamma\in \mathbb{K}_P^\lambda[C].$
\end{proof}
\begin{rem}
Let $e$  be the identity element of $G$. One can easily observe that $ Z_P(\mathbb{K}_P^\lambda[G])=\mathbb{K}$ if and only if $C=\{e\}.$  
\end{rem}

\subsection{Central and inner Poisson derivations} 
\subsubsection{Central Poisson derivations.}
Let $\theta:(G,\cdot)\longrightarrow (\mathbb{K}_P^\lambda[C],+)$ be  a group homomorphism.  That is, $\theta(xy)=\theta(x)+\theta(y)$ for all $x,y\in G.$ Define a $\mathbb{K}$-linear operator 
$\mathcal{D}:=\mathcal{D}_{\theta}$ by $$\mathcal{D}(\bar{x})=\theta(x)\bar{x}$$ for all $x\in G.$ 
\begin{lem}
\label{dd}
 $\mathcal{D}$ is a Poisson derivation of $\mathbb{K}_P^\lambda[G]$.
\end{lem}
\begin{proof}
We need to show that
$\mathcal{D}(\bar{x}\bar{y})=\mathcal{D}(\bar{x})\bar{y}+\bar{x}\mathcal{D}(\bar{y})\  \text{ and}  \ \mathcal{D}(\{\bar{x},\bar{y}\})=\{\mathcal{D}(\bar{x}),\bar{y}\}+\{\bar{x},\mathcal{D}(\bar{y})\}$ for all $x,y\in G.$ Now, $\mathcal{D}(\bar{x}\bar{y})=\theta(xy)\overline{xy}=\theta(xy)\bar{x}\bar{y}=\theta(x)\bar{x}\bar{y
}+\theta(y)\bar{x}
\bar{y}=\mathcal{D}(\bar{x})\bar{y}+\bar{x}\mathcal{D}(\bar{y}).$ 

Secondly, 
$
\mathcal{D}(\{\bar{x},\bar{y}\})=\lambda{(x,y)}\mathcal{D}(\bar{x}\bar{y})=
\theta(xy)\lambda{(x,y)}\bar{x}\bar{y}=\theta(xy)\{\bar{x},\bar{y}\}
=[\theta(x)+\theta(y)]\{\bar{x},\bar{y}\}
=\theta(x)\{\bar{x},\bar{y}\}+\theta(y)\{\bar{x},\bar{y}\}=\{\theta(x)\bar{x},\bar{y}\}+\{\bar{x},\theta(y)\bar{y}\}=\{\mathcal{D}(\bar{x}),\bar{y}\}+
\{\bar{x},\mathcal{D}(\bar{y})\}
$
(note that $\{\theta(x),\bar{y}\}=\{\bar{x},\theta(y)\}=0$, since $\theta(x)$ and $\theta(y)$ are Poisson central elements). 
\end{proof}
Similarly to \cite{op}, we will call $\mathcal{D}$ in Lemma \ref{dd} as a \textit{central Poisson derivation} when $Z_P(\mathbb{K}_P^\lambda[G])$ strictly contains $\mathbb{K},$ and a \textit{scalar Poisson derivation} when $Z_P(\mathbb{K}_P^\lambda[G])=\mathbb{K}.$ 

Observe that $\mathcal{D}(\bar{x})=\theta(x)\bar{x}\in \mathbb{K}_P^\lambda[Cx]$ for all $x\in G.$ 

\subsubsection{Inner Poisson derivations.}
Let $\gamma=\sum_{g}c_g\bar{g}\in \mathbb{K}_P^\lambda[G],$ where $ c_g\in \mathbb{K},$ and 
ham$_{\gamma}:=\{\gamma,-\}.$ It is well known that  ham$_{\gamma}:\mathbb{K}_P^\lambda[G]\longrightarrow \mathbb{K}_P^\lambda[G]$ is a derivation  called the \textit{hamiltonian derivation associated to $\gamma$.} 
Moreover, ham$_{\gamma}(\bar{x})=\{\gamma, \bar{x}\}=\sum_{g}\lambda{(g,x)}c_g\bar{g}\bar{x}\in \mathbb{K}_P^\lambda[Gx]$  for all $x\in G.$ As usual, we can assume that  $C \ \cap$ supp$(\gamma)=\emptyset.$
Therefore, ham$_{\gamma}( \bar{x}) \in \mathbb{K}_P^\lambda[(G\setminus C)x]$ for all $x\in G.$ 

We call  the hamiltonian derivation ham$_{\gamma}$ an \textit{ inner  Poisson derivation}. \\

We can now state our main result in this section in the theorem below.
\begin{thm}
\label{pic1}
Every Poisson derivation of the Poisson group algebra $\mathbb{K}_P^\lambda[G]$ is uniquely the sum of an inner Poisson  derivation and a central Poisson  derivation.
\end{thm}
\begin{proof} 
Let $\mathcal{D}$ be a Poisson derivation of $\mathbb{K}_P^\lambda[G].$ Then, for $x\in G,$ we have that $\mathcal{D}(\bar{x})\in \mathbb{K}_P^\lambda[G].$ Hence, $\mathcal{D}(\bar{x})=\sum_{h\in G}b_h(x)\bar{h}=\sum_{h\in G}b_h(x)\bar{h}\bar{x}^{-1}\bar{x}.$ Now, the map $G\rightarrow G$ with 
$h\mapsto hx^{-1}$ is bijective, and so $$\mathcal{D}(\bar{x})=\sum_{g}a_g(x)\bar{g}\bar{x},$$ where $g:=hx^{-1}$ and $a_g(x):=b_{gx}(x).$ Note that $a_g:G\longrightarrow \mathbb{K}$ and $ a_g(x)=0$ for almost all $x\in G.$

 Since $\mathcal{D}$ is a Poisson derivation, we have that
$\mathcal{D}(\bar{x}\bar{y})=\mathcal{D}(\bar{x})\bar{y}+\bar{x}\mathcal{D}(\bar{y})$ for all $x,y\in G.$ As a result,
$$\sum_{g}a_g(xy)\bar{g}\bar{x}\bar{y}=\sum_{g}a_g(x)\bar{g}\bar{x}\bar{y}
+\sum_{g}a_g(y)\bar{g}\bar{x}\bar{y}=\sum_{g}[a_g(x)+a_g(y)]\bar{g}\bar{x}\bar{y}.$$ Identifying the coefficients in the above equality reveals that $$a_g(xy)=a_g(x)+a_g(y).$$

Secondly, $\mathcal{D}(\{\bar{x},\bar{y}\})=\{\mathcal{D}(\bar{x}),\bar{y}\}+\{\bar{x},\mathcal{D}(\bar{y})\}.$  Now,
\begin{align}
\mathcal{D}(\{\bar{x},\bar{y}\})=\lambda{(x,y)}\mathcal{D}(\bar{x}\bar{y})=\sum_{g}\lambda{(x,y)}a_g(xy)\bar{g}\bar{x}
\bar{y}.\label{e27}
\end{align}
 On the other hand,
\begin{align}
\{\mathcal{D}(\bar{x}),\bar{y}\}+\{\bar{x},\mathcal{D}(\bar{y})\}&=\sum_{g}a_g(x)\{\bar{g}\bar{x},\bar{y}\}+\sum_{g}a_g(y)\{\bar{x},\bar{g}\bar{y}\}\nonumber\\
&=\sum_{g}a_g(x)\{\overline{gx},\bar{y}\}+\sum_{g}a_g(y)\{\bar{x},\overline{gy}\}\nonumber\\
&=\sum_{g}[a_g(x)\lambda(gx,y)+a_g(y)\lambda(x,gy)]\bar{g}\bar{x}\bar{y}\nonumber\\
&=\sum_{g} a_g(x)[\lambda(g,y)+\lambda(x,y)]+a_g(y)[\lambda(x,g)+\lambda(x,y)]\bar{g}\bar{x}\bar{y}\nonumber\\
&=\sum_{g}[ \lambda(x,y)a_g(xy) + a_g(x)\lambda(g,y)-a_g(y)\lambda(g,x)]\bar{g}\bar{x}\bar{y}. \label{e28}
\end{align}
Since $\mathcal{D}(\{\bar{x},\bar{y}\}=\{\mathcal{D}(\bar{x}),\bar{y}\}+\{\bar{x},\mathcal{D}(\bar{y})\},$ comparing \eqref{e27} to \eqref{e28} reveals that
\begin{align*}
\lambda{(x,y)}a_g(xy)&= \lambda(x,y)a_g(xy) + a_g(x)\lambda(g,y)-a_g(y)\lambda(g,x).
\end{align*}
 This implies that
 \begin{equation}
 \label{sl}
 a_g(x)\lambda(g,y) = a_g(y)\lambda(g,x).
 \end{equation}
Suppose that $g\in C$. It follows that $\lambda(g,y)=\lambda(g,x)=0$ for all $x,y\in G.$ Since $a_g(xy)=a_g(x)+a_g(y),$ the map $\theta: (G,\cdot)\longrightarrow (\mathbb{K}_P^\lambda[C],+)$ given by
$\theta(x)=\sum_{g\in C}a_g(x)\bar{g}$ is a group homomorphism. Hence, $\theta$ defines a central Poisson derivation $\mathcal{D}_{\theta}$ of $\mathbb{K}_P^\lambda[G],$ where
\begin{align}
\label{pgae1}
\mathcal{D}_{\theta}(\bar{x})=\sum_{g\in C}a_g(x)\bar{g}\bar{x}.
\end{align} 
Now, let $g\not\in C.$ There exists $y\in G$ such that 
$\lambda(g,y)\neq 0.$ Fix $y$ and define 
$$c_g:=\frac{a_g(y)}{\lambda(g,y)}.$$
Take any arbitrary element $x\in G$. It follows that 
$$c_g\lambda(g,x)=\frac{a_g(y)\lambda(g,x)}{\lambda(g,y)}.$$
From \eqref{sl}, we have that 
$$c_g\lambda(g,x)=\frac{a_g(y)\lambda(g,x)}{\lambda(g,y)}=
\frac{a_g(x)\lambda(g,y)}{\lambda(g,y)}=a_g(x),$$ for all 
$x\in G. $

Define $\gamma\in \mathbb{K}_P^\lambda[G]$ as $\gamma:=\sum_{g\not\in C}c_g \bar{g}.$ Then,
\begin{align}
\label{pgae2}
\text{ham}_{\gamma}(\bar{x})=\{\gamma, \bar{x}\}=\sum_{g\not\in C}c_g\lambda(g,x)\bar{g}\bar{x}=
\sum_{g\not\in C}a_g(x)\bar{g}\bar{x}.
\end{align}
From \eqref{pgae1} and \eqref{pgae2}, 
one can conclude that every Poisson derivation $\mathcal{D}$ of $\mathbb{K}_P^\lambda[G]$ can be written as $\mathcal{D}=\mathcal{D}_\theta+\text{ham}_{\gamma}.$ This decomposition of $\mathcal{D}$ into an inner Poisson  derivation ($\text{ham}_{\gamma}$) and a central Poisson derivation ($\mathcal{D}_\theta$) is actually unique, because $\mathbb{K}_P^\lambda[Gx]$ can be decomposed as $\mathbb{K}_P^\lambda[Gx]=\mathbb{K}_P^\lambda[Cx]\oplus \mathbb{K}_P^\lambda[(G\setminus C)x].$ Now, every central Poisson derivation maps $\bar{x}$ to an element of the subspace $\mathbb{K}_P^\lambda[Cx],$ and every inner Poisson derivation maps $\bar{x}$ to an element of the subspace $\mathbb{K}_P^\lambda[(G\setminus C)x].$
\end{proof}
\begin{cor}
\label{pic2}
Suppose that $C=\{e\}$ (equivalently, $Z_P(\mathbb{K}_P^\lambda[G])=\mathbb{K}$). Then, every Poisson derivation of $\mathbb{K}_P^\lambda[G]$ is uniquely the sum of an inner and a scalar Poisson derivation.
\end{cor}

%%%%%%%%%%%%%%%%%%%%%%%%%%%%%%%%%%%%%%%%%%%%%%%%%%%%%%%%%%%%
%%%%%%%%%%%%%%%%%%%%%%%%%%%%%%%%%%%%%%%%%%%%%%%%%%%%%%%%%%%%
%%%%%%%%%%%%%%%%%%%%%%%%%%%%%%%%%%%%%%%%%%%%%%%%%%%%%%%%%%%% 

\section{Poisson prime spectrum and Poisson deleting derivations algorithm of a semiclassical limit of  $U_q^+(G_2)$}
\label{sec2}
This section aims to study a semiclassical limit of the positive part of the quantized enveloping algebra of type $G_2,$ $U_q^+(G_2)$. Given the semiclassical limit of $U_q^+(G_2)$, we will study its Poisson prime spectrum using Goodearl's $\mathcal{H}$-stratification theory \cite{gd}, and its Poisson deleting derivations algorithm--introduced by Launois and Lecoutre \cite{sc}. Given the data of the Poisson deleting derivations algorithm, we will study the Poisson centre of the semiclassical limit.  
\subsection{Semiclassical limit of the algebra $U_q^+(G_2)$}
\label{sub1}
Recall from \cite[\S 2]{lo} that $U_q^+(G_2)$ satisfies the following relations:
\begin{align*}
E_2E_1&=q^{-3}E_1E_2& E_3E_1&=q^{-1}E_1E_3-(q+q^{-1}+q^{-3})E_2\\
E_3E_2&=q^{-3}E_2E_3&E_4E_1&=E_1E_4+(1-q^2)E_3^2\\
E_4E_2&=q^{-3}E_2E_4-\frac{q^4-2q^2+1}{q^4+q^2+1}E_3^3&E_4E_3&=q^{-3}E_3E_4\\
E_5E_1&=qE_1E_5-(1+q^2)E_3&E_5E_2&=E_2E_5+(1-q^2)E_3^2\\
E_5E_3&=q^{-1}E_3E_5-(q+q^{-1}+q^{-3})E_4&E_5E_4&=q^{-3}E_4E_5\\
E_6E_1&=q^3E_1E_6-q^3E_5&E_6E_2&=q^3E_2E_6+(q^4+q^2-1)E_4+\\
 E_6E_3&=E_3E_6+(1-q^2)E_5^2  &&\qquad \ (q^2-q^4)E_3E_5\\
E_6E_4&=q^{-3}E_4E_6-\frac{q^4-2q^2+1}{q^4+q^2+1}E_5^3& E_6E_5&=q^{-3}E_5E_6.
\end{align*}
Set
$U_i:= (q-1)E_i$
for $i=1,3,4,5,$  and $U_i:=f(q)(q-1)E_i$ for $i=2,6;$ where $f(q)=q^4+q^2+1.$ Then, $U_q^+(G_2)$ is now generated by $U_1, \ldots, U_6$ subject to the relations: 
\begin{align*}
U_2U_1&=q^{-3}U_1U_2& \qquad U_3U_2&=q^{-3}U_2U_3 \\
U_3U_1&=q^{-1}U_1U_3-q^{-3}(q-1)U_2&\qquad U_4U_1&=U_1U_4+(1-q^2)U_3^2\\
U_4U_2&=q^{-3}U_2U_4-(q+1)^2(q-1)U_3^3&\qquad U_4U_3&=q^{-3}U_3U_4\\
U_5U_1&=qU_1U_5-(1+q^2)(q-1)U_3&\qquad U_5U_2&=U_2U_5+f(q) (1-q^2)U_3^2\\
U_5U_3&=q^{-1}U_3U_5-f(q)(q^{-2}-q^{-3})U_4&\qquad U_5U_4&=q^{-3}U_4U_5\\
U_6U_1&=q^3U_1U_6-f(q)(q^4-q^3)U_5&\qquad U_6U_2&=q^3U_2U_6+f(q)^2(q^2-q^4)U_3U_5+\\
U_6U_3&=U_3U_6+f(q)(1-q^2)U_5^2&\qquad & \qquad f(q)^2 (q^4+q^2-1)(q-1)U_4  \\
U_6U_4&=q^{-3}U_4U_6-(q+1)^2(q-1)U_5^3&\qquad U_6U_5&=q^{-3}U_5U_6.
\end{align*}
We now find a `new' presentation for  $U_q^+(G_2)$ that allows us to introduce a quantisation of $U_q^+(G_2).$
 Let $\widehat{U_q^+(G_2)} $ be a $\mathbb{K}[z^{\pm 1}]$-algebra generated by 
$\widehat{U}_1,\ldots,\widehat{U}_6$ subject to the relations:
\begin{align*}
\widehat{U}_2\widehat{U}_1&=z^{-3}\widehat{U}_1\widehat{U}_2&\qquad \widehat{U}_3\widehat{U}_2&
=z^{-3}\widehat{U}_2\widehat{U}_3 \\
\widehat{U}_3\widehat{U}_1&=z^{-1}\widehat{U}_1\widehat{U}_3-z^{-3}(z-1)\widehat{U}_2&\qquad \widehat{U}_4\widehat{U}_1&=\widehat{U}_1\widehat{U}_4+(1-z^2)\widehat{U}_3^2\\
\widehat{U}_4\widehat{U}_2&=z^{-3}\widehat{U}_2\widehat{U}_4-(z+1)^2(z-1)\widehat{U}_3^3
&\qquad \widehat{U}_4\widehat{U}_3&=z^{-3}\widehat{U}_3\widehat{U}_4\\
\widehat{U}_5\widehat{U}_1&=z\widehat{U}_1\widehat{U}_5-(1+z^2)(z-1)\widehat{U}_3&\qquad \widehat{U}_5\widehat{U}_2&=\widehat{U}_2\widehat{U}_5+f(z)(1-z^2)\widehat{U}_3^2\\
\widehat{U}_5\widehat{U}_3&=z^{-1}\widehat{U}_3\widehat{U}_5-f(z)(z^{-2}-z^{-3})\widehat{U}_4&\qquad \widehat{U}_5\widehat{U}_4&=z^{-3}\widehat{U}_4\widehat{U}_5\\
\widehat{U}_6\widehat{U}_1&=z^3\widehat{U}_1\widehat{U}_6-f(z)(z^4-z^3)\widehat{U}_5&\qquad \widehat{U}_6\widehat{U}_2&=z^3\widehat{U}_2\widehat{U}_6
+ f(z)^2(z^2-z^4)\widehat{U}_3\widehat{U}_5+\\
\widehat{U}_6\widehat{U}_3&=\widehat{U}_3\widehat{U}_6+f(z)(1-z^2)\widehat{U}_5^2& \qquad &\qquad f(z)^2(z^4+z^2-1)(z-1)\widehat{U}_4 \\
\widehat{U}_6\widehat{U}_4&=z^{-3}\widehat{U}_4\widehat{U}_6-
(z+1)^2(z-1)\widehat{U}_5^3& \qquad \widehat{U}_6\widehat{U}_5&=z^{-3}\widehat{U}_5\widehat{U}_6,
\end{align*}
where $f(z)=z^4+z^2+1.$
Fix $\lambda\in \mathbb{K}^*.$ Observe that the element $z-\lambda$ is central and not invertible in $\widehat{U_q^+(G_2)},$ hence we set
$\mathcal{A}_\lambda:=\widehat{U_q^+(G_2)}/(z-\lambda)\widehat{U_q^+(G_2)}.$
Now, $\mathcal{A}_q$ is the non-commutative algebra $U_q^+(G_2)$ and 
$\mathcal{A}_1=\mathbb{K}[X_1,\ldots,X_6]$
with $X_i:=\widehat{U}_i+(z-1)\widehat{U_q^+(G_2)}$ is a Poisson algebra with the Poisson bracket defined as follows:
\begin{align*}
\{X_2, X_1\}&=-3X_1X_2& \{X_3,X_1\}&=-X_1X_3-X_2 \\
\{X_3,X_2\}&=-3X_2X_3& \{X_4, X_1\}&=-2X_3^2 \\
\{X_4, X_2\}&=-3X_2X_4-{4}X_3^3&\{X_4, X_3\}&=-3X_3X_4\\
\{X_5, X_1\}&=X_1X_5-2X_3& \{X_5,X_2\}&=-6X_3^2\\
\{X_5,X_3\}&=-X_3X_5-3X_4&\{X_5,X_4\}&=-3X_4X_5\\
 \{X_6,X_1\}&=3X_1X_6-3X_5&\{X_6, X_2\}&=3X_2X_6+9X_4-18X_3X_5\\
 \{X_6, X_3\}&=-6X_5^2&\{X_6, X_4\}&=-3X_4X_6-4X_5^3\\
\{X_6, X_5\}&=-3X_5X_6.
\end{align*}
  Therefore, $\mathcal{A}_1$ is a semiclassical limit of the non-commutative algebra 
$\widehat{U_q^+(G_2)},$
 and  $\mathcal{A}_q$ is a deformation of the Poisson algebra $\mathcal{A}_1.$ 
For simplicity, we set $$\mathcal{A}:=\mathcal{A}_1=\mathbb{K}[X_1,\ldots,X_6]$$ in the rest of this paper.  
One can write the Poisson algebra $\mathcal{A}$ as an iterated Poisson-Ore extension over $\mathbb{K}$ (see \cite[Theorem 1.1]{oh} for the definition of iterated Poisson-Ore extension) as follows:
\begin{equation}
\label{eqs1}
\mathcal{A}=\mathbb{K}[X_1][X_2;\sigma_2]_P[X_3;\sigma_3,\delta_3]_P[X_4;\sigma_4,\delta_4]_P[X_5;\sigma_5,\delta_5]_P[X_6;\sigma_6,\delta_6]_P;
\end{equation}
 where, $\sigma_i$ and $\delta_i$ are respectively the Poisson derivations and Poisson $\sigma_i$-derivations of
$$
\mathbb{K}[X_1][X_2;\sigma_2]_P[X_3;\sigma_3,\delta_3]_P\ldots [X_{i-1};\sigma_{i-1},\delta_{i-1}]_P
$$
($2\leq i\leq 6$ and $\delta_2=0$) defined as follows:
\begin{align*}
\sigma_2(X_1)&=-3X_1& \sigma_3(X_1)&=-X_1&\sigma_3(X_2)&=-3X_2& \sigma_4(X_1)&=0\\
\sigma_4(X_2)&=-3X_2& \sigma_4(X_3)&=-3X_3&\sigma_5(X_1)&=X_1&\sigma_5(X_2)&=0\\
\sigma_5(X_3)&=-X_3&\sigma_5(X_4)&=-3X_4& \sigma_6(X_1)&=3X_1&\sigma_6(X_2)&=3X_2\\
\sigma_6(X_3)&=0&\sigma_6(X_4)&=-3X_4& \sigma_6(X_5)&=-3X_5,
\end{align*}
and 
\begin{align*}
\delta_3(X_1)&=-X_2&\delta_3(X_2)&=0&\delta_4(X_1)&=-2X_3^2&\delta_4(X_2)&=-{4} X_3^3\\
\delta_4(X_3)&=0&\delta_5(X_1)&=-2X_3&\delta_5(X_2)&=-6X_3^2&\delta_5(X_3)&=-3 X_4\\
\delta_5(X_4)&=0&\delta_6(X_1)&=-3X_5&\delta_6(X_2)&=9X_4-18X_3X_5&\delta_6(X_3)&=-6X_5^2\\
\delta_6(X_4)&=-{4}X_5^3&\delta_6(X_5)&=0.
\end{align*}
\begin{rem} \cite[Theorem I.1.13]{bg} (Hilbert's Basis Theorem)
Since  $\mathbb{K}$ is a noetherian domain,
the Poisson algebra $\mathcal{A}=\mathbb{K}[X_1][X_2;\sigma_2]_P[X_3;\sigma_3,\delta_3]_P[X_4;\sigma_4,\delta_4]_P[X_5;\sigma_5,\delta_5]_P[X_6;\sigma_6,\delta_6]_P$ is a noetherian domain. 
\end{rem}
\subsection{Poisson prime spectrum of the semiclassical limit of the algebra $U_q^+(G_2)$}
We study the Poisson prime spectrum of the Poisson algebra $\mathcal{A}=\mathbb{K}[X_1,\ldots,X_6]$ in this subsection. Let $P$ be a proper Poisson ideal of a Poisson algebra $\mathcal{A}$ and $I_1$, $I_2$  be Poisson ideals of $\mathcal{A}$ such that $P\supseteq I_1I_2.$ The ideal  $P$ is called a \textit{Poisson prime ideal} provided $P\supseteq I_1$ or $P\supseteq I_2.$ Since the Poisson algebra $\mathcal{A}$ is a noetherian domain, every Poisson ideal which is also a prime ideal is a Poisson prime ideal and vice versa (see \cite[Lemma 1.1]{gd}).
 The set of all the Poisson prime ideals of  $\mathcal{A}$ is called the    \textit{Poisson prime spectrum} of $\mathcal{A},$ denoted by P.Spec$(\mathcal{A}).$  The largest Poisson prime ideal contained in a given maximal ideal of $\mathcal{A}$ is called a \textit{Poisson primitive ideal.}  The set of all these Poisson primitive ideals is  called the \textit{Poisson primitive spectrum} of $\mathcal{A}$, denoted by P.Prim$(\mathcal{A}).$ 

Let  $\mathcal{H}$ be an algebraic torus acting rationally  on $\mathcal{A}$ by  Poisson automorphism (see \cite[\S 2]{gd} for the definition of a rational torus action). A Poisson prime ideal $P$ is $\mathcal{H}$-\textit{invariant} if $h\cdot P=P$ for all 
$h\in\mathcal{H}.$ Moreover, $J:=\bigcap_{h\in\mathcal{H}}h\cdot P$ is the largest Poisson  $\mathcal{H}$-invariant prime ideal contained in $P.$

The set
$$\text{P.Spec}_J(\mathcal{A}):=\{P\in\text{P.Spec}(\mathcal{A})\mid (P:\mathcal{H})=J\}$$
is called the \textit{$J$-stratum} of P.Spec$(\mathcal{A}).$ 
 The $\mathcal{H}$-strata
 $\text{P.Spec}_J(\mathcal{A})$ partitioned $\text{P.Spec}(\mathcal{A})$ into a disjoint union of strata. Hence,
 $$\text{P.Spec}(\mathcal{A})=\underset{J\in\mathcal{H}\text{-P.Spec}(\mathcal{A})}{\bigsqcup}\text{P.Spec}_J(A),$$ where $\mathcal{H}\text{-P.Spec}(\mathcal{A})$ is the collection of all the Poisson $\mathcal{H}$-invariant prime ideals of $\mathcal{A}.$ This partition is called an $\mathcal{H}$-\textit{stratification} of $\text{P.Spec}(\mathcal{A}).$ In a similar manner, an $\mathcal{H}$-stratification of P.Prim$(\mathcal{A})$ is obtained as follows:  
$$\text{P.Prim}(\mathcal{A})=\underset{J\in\mathcal{H}\text{-P.Prim}(\mathcal{A})}{\bigsqcup}\text{P.Prim}_J(\mathcal{A}),$$
where $\text{P.Prim}_J(\mathcal{A})=\text{P.Spec}_J(\mathcal{A})\cap \text{P.Prim}(\mathcal{A})$. 
% and  $\mathcal{H}\text{-P.Prim}(\mathcal{A})$ is the collection of all the Poisson $\mathcal{H}$-invariant primitive ideals of $\mathcal{A}.$

The iterated-Poisson algebra $\mathcal{A}$ has only finitely many Poisson $\mathcal{H}$-primes by \cite[Theorem 1.7]{gkls}, and so we deduce from  \cite[Theorem 4.3]{gd} the following result.

\begin{pro} 
\label{ppf1}
Let $P\in$ P.Spec$_J(\mathcal{A}),$ P is a Poisson primitive ideal of $\mathcal{A}$ if and only if $P$ is maximal in   P.Spec$_J(\mathcal{A}).$
\end{pro}

\subsection{Poisson deleting derivations algorithm of the semiclassical limit of $U_q^+(G_2)$} 
\label{w}
\label{pdda}
In \cite{sc}, the authors studied a Poisson version, called the \textit{Poisson deleting derivations algorithm} (PDDA for short), of the well-known Cauchon's deleting derivations algorithm (see \cite{ca}).  In this subsection, we  study PDDA  of $\mathcal{A}=\mathbb{K}[X_1,\ldots,X_6].$ 
From \eqref{eqs1}, we have that $\mathcal{A}=\mathbb{K}[X_1][X_2;\sigma_2,\delta_2]_p\ldots [X_6;\sigma_6,\delta_6]_P.$ 
One can  easily verify that  $\mathcal{A}$ satisfies the conditions in the hypothesis below. 
\begin{hypo}
\label{hyp}
\begin{itemize}
\item[(H1)] For all $1\leq j<i\leq 6,$ there exists 
$\mu_{ij}\in\mathbb{K}$ with $\mu_{ji}:=-\mu_{ij},$ such that $\sigma_i(X_j)=\mu_{ij}X_j.$
\item[(H2)] The derivations $\delta_i$ are all locally nilpotent and $\delta_i\alpha-\alpha\delta_i=\eta_i\delta_i$ for some non-zero scalar $\eta_i,$ for all $2\leq i\leq 6.$
\end{itemize} 
\end{hypo}
As a result, the PDDA can be used to study the Poisson prime spectrum of $\mathcal{A}$ (see \cite[Hypothesis 1.7]{sc}). Let $1\leq i\leq 6$ and $2\leq j\leq 7.$ Using the relation
\begin{equation*}
  X_{i,j}:=\begin{cases}
    X_{i,j+1} & \text{if $i\geq j$}\\
    \displaystyle\sum_{k=0}^{+\infty}\frac{1}{\eta_j^kk!}
    \delta_j^k(X_{i,j+1})X_{j,j+1}^{-k} & \text{if $i<j$},
  \end{cases}
\end{equation*}
(note that since $\delta_j$ is locally nilpotent, the  summation is finite), one can construct a family $(X_{1,j},\ldots,X_{6,j})$ of elements of the division ring of fractions Fract$(\mathcal{A})$ of  $\mathcal{A}$ as follows: 
\begin{align*}
X_{1,6}&=X_1-\frac{1}{2}X_5X_6^{-1}\\
X_{2,6}&=X_2+\frac{3}{2}X_4X_6^{-1}-3X_3X_5X_6^{-1}+X_5^3X_6^{-2}\\
X_{3,6}&=X_3-X_5^2X_6^{-1}\\
X_{4,6}&=X_4-\frac{2}{3}X_5^3X_6^{-1}\\
X_{1,5}&=X_{1,6}-X_{3,6}X_{5,6}^{-1}+\frac{3}{4}X_{4,6}X_{5,6}^{-2}\\
X_{2,5}&=X_{2,6}-3X_{3,6}^2X_{5,6}^{-1}+\frac{9}{2}X_{3,6}
X_{4,6}X_{5,6}^{-2}-\frac{9}{4}X_{4,6}^2X_{5,6}^{-3}\\
X_{3,5}&=X_{3,6}-\frac{3}{2}X_{4,6}X_{5,6}^{-1}\\
X_{1,4}&=X_{1,5}-\frac{1}{3}X_{3,5}^2X_{4,5}^{-1}\\
X_{2,4}&=X_{2,5}-\frac{2}{3}X_{3,5}^3X_{4,5}^{-1}\\
X_{1,3}&=X_{1,4}-\frac{1}{2}X_{2,4}X_{3,4}^{-1}\\
T_1&:=X_{1,2}=X_{1,3}\\
T_2&:=X_{2,2}=X_{2,3}=X_{2,4}\\
T_3&:=X_{3,2}=X_{3,3}=X_{3,4}=X_{3,5}\\
T_4&:=X_{4,2}=X_{4,3}=X_{4,4}=X_{4,5}=X_{4,6}\\
T_5&:=X_{5,2}=X_{5,3}=X_{5,4}=X_{5,5}=X_{5,6}=X_5\\
T_6&:=X_{6,2}=X_{6,3}=X_{6,4}=X_{6,5}=X_{6,6}=X_6.
\end{align*}
For each $2\leq j\leq 7,$ the algebra $\mathcal{A}^{(j)}$ represents the subalgebra of Fract$(\mathcal{A})$  generated by all the  $X_{i,j}.$
That is, $\mathcal{A}^{(j)}=\mathbb{K}[ X_{1,j},\ldots, X_{6,j}].$ Note that  $\mathcal{A}^{(7)}=\mathcal{A}.$  It follows from \cite[Prop. 1.11]{sc} that
\begin{align*}
\mathcal{A}^{(j)}&\cong\mathbb{K}[X_1][X_2;\sigma_2,\delta_2]_P\cdots[X_{j-1};\sigma_{j-1},\delta_{j-1}]_P[X_j;\tau_j]_P\cdots[X_7;\tau_7]_P,
\end{align*}
 by an isomorphism that maps $X_{i,j}$ to $X_i,$ and $\tau_j,\ldots, \tau_7$ denote the Poisson derivations defined by $\tau_l(X_i)=\mu_{li}X_i$ for all $1\leq i< l\leq 7.$ With a slight abuse of notation,  one can identify $\tau_j,\ldots,\tau_7$ with $\sigma_j,\ldots, \sigma_7$ respectively. 
\begin{notn}
$\overline{\mathcal{A}}:=\mathcal{A}^{(2)}=\mathbb{K}[T_1,\ldots, T_6].$
\end{notn}
One can easily check that $\overline{\mathcal{A}}$ is a Poisson affine space associated to the skew-symmetric matrix
\begin{align*}
M:=&\begin{bmatrix}
0&3&1&0&-1&-3\\
-3&0&3&3&0&-3\\
-1&-3&0&3&1&0\\
0&-3&-3&0&3&3\\
1&0&-1&-3&0&3\\
3&3&0&-3&-3&0
\end{bmatrix}.
%\label{c2ea}
\end{align*}
That is, $\overline{\mathcal{A}}$ satisfies the relation $\{T_i,T_j\}=\mu_{ij}T_jT_i$ for all 
$1\leq i,j\leq 6,$ where $\mu_{ij}$ are the entries of $ M.$

\subsection{Canonical embedding}
\label{sub4}
The set $\Sigma:=\{X_{j,j+1}^n\mid n\in\mathbb{N}\}=\{X_{j,j}^n\mid n\in\mathbb{N}\}$ is a multiplicative system of regular elements of $\mathcal{A}^{(j)}$ and $\mathcal{A}^{(j+1)}.$ Moreover,  $\mathcal{A}^{(j)}\Sigma_j^{-1}=\mathcal{A}^{(j+1)}\Sigma_j^{-1}$ \cite[Prop. 1.11]{sc}.
One can use the PDDA to relate 
P.Spec$(\mathcal{A})$ to P.Spec$(\overline{\mathcal{A}})$  by constructing an embedding $\psi_j:$ P.Spec$(\mathcal{A}^{(j+1)})\hookrightarrow$ P.Spec$(\mathcal{A}^{(j)})$ defined by 
\begin{equation*}
  \psi_j(P):=\begin{cases}
    P\Sigma_j^{-1}\cap \mathcal{A}^{(j)} & \text{if $X_{j,j+1}=T_j\not\in P,$}\\
    g_j^{-1}(P/\langle X_{j,j+1}\rangle) & \text{if $X_{j,j+1}\in P,$}
  \end{cases}
\end{equation*}
for each $2\leq j\leq 6$ (see \cite[Lemma 2.3]{sc}).  The map
 $g_j$ is a surjective homomorphism $$g_j:\mathcal{A}^{(j)}\rightarrow \mathcal{A}^{(j+1)}/\langle X_{j,j+1}\rangle$$ defined by
 $$g_j(X_{j,j}):=X_{j,j+1}+\langle X_{j,j+1}\rangle$$
(further details can be found in \cite[\S 2]{sc}). 
From \cite[\S 2.1]{sc}, there exists an increasing homeomorphism from the topological space 
$$\{P\in \text{P.Spec}(\mathcal{A}^{(j+1)})\mid X_{j,j+1}\not\in P\}$$
onto the topological space
$$\{Q\in \text{P.Spec}(\mathcal{A}^{(j)})\mid X_{j,j}\not\in P\}$$
whose inverse is also an increasing homeomorphism (the topology being the Zariski topology). The map $\psi_j$ is injective but not necessarily bijective. However, $\psi_j$ induces a bijection  between $\{ P\in \text{P.Spec}(\mathcal{A}^{(j+1)})\mid P\cap \Sigma_j=\emptyset \}$ and $\{Q\in\text{P.Spec}(\mathcal{A}^{(j)})\mid Q\cap \Sigma_j=\emptyset\}$ \cite[Lemma 2.1]{sc}. 
The so-called \textit{canonical embedding} $$\psi:=\psi_2\circ \cdots \circ \psi_6=: \text{P.Spec}(\mathcal{A})\hookrightarrow \text{P.Spec}(\overline{\mathcal{A}})$$ is obtained by composing all the $\psi_j.$  This canonical embedding $\psi$ helps to construct a partition of P.Spec$(\mathcal{A})$ into a disjoint union of strata known as the \textit{canonical partition} via the Cauchon diagrams.

The torus invariant Poisson prime spectrum has generally been described in \cite[\S 2.2]{sc} via Cauchon's diagrams. Based on that, we can describe the set $\mathcal{H}\text{-P.Spec}(\overline{\mathcal{A}})$ as follows.
For any subset $C$ of $\{1,\dots,6\}$, let $K_C$ denote the Poisson $\mathcal{H}$-invariant prime ideal of 
 $\overline{\mathcal{A}}$  generated by the $T_i$ with $i\in C$. We deduce from \cite[\S 2.2]{sc} that
 \[
K_C= \langle T_i\mid i\in C\rangle ,
\]
and  
$$\mathcal{H}\text{-P.Spec}(\overline{\mathcal{A}}) = \{K_C ~|~ C \subseteq \{1,\dots,6\} \},$$
so that 
$$\psi (\mathcal{H}\text{-P.Spec}(\mathcal{A})) \subseteq \{K_C ~|~ C \subseteq \{1,\dots,6\} \}.$$

\subsection{Poisson centre of $\mathcal{A}$ }
\label{sub2}
The monomials $\Omega_1:=T_1T_3T_5$ and $\Omega_2:=T_2T_4T_6$ are Poisson central elements of the Poisson affine space $\overline{\mathcal{A}},$ since $\{\Omega_i,T_j\}=0$ for all $i=1,2,$ and $1\leq j\leq 6.$
We now want to successively pull $\Omega_1$ and $\Omega_2$ from  $\overline{\mathcal{A}}$ into  ${\mathcal{A}}$ using the data of the PDDA of ${\mathcal{A}}$.   Through a direct computation, one can confirm that
\begin{align*}
\Omega_1&:=T_1T_3T_5\\
&=X_{1,3}X_{3,3}X_{5,3}\\
&=X_{1,4}X_{3,4}X_{5,4}-\frac{1}{2}X_{2,4}X_{5,4}\\
&=X_{1,5}X_{3,5}X_{5,5}-\frac{1}{2}X_{2,5}X_{5,5}\\
&=X_{1,6}X_{3,6}X_{5,6}-\frac{3}{2}X_{1,6}X_{4,6}-\frac{1}{2}X_{2,6}X_{5,6}+\frac{1}{2}X_{3,6}^2\\
&=X_1X_3X_5-\frac{3}{2}X_1X_4-\frac{1}{2}X_2X_5+\frac{1}{2}X_3^2, 
\end{align*}
and
\begin{align*}
\Omega_2&:=T_2T_4T_6\\&=X_{2,4}X_{4,4}X_{6,4}\\
&=X_{2,5}X_{4,5}X_{6,5}-\frac{2}{3}X_{3,5}^3X_{6,5}\\
&=X_{2,6}X_{4,6}X_{6,6}-\frac{2}{3}X_{3,6}^3X_{6,6}\\
&=X_2X_4X_6-\frac{2}{3}X_3^3X_6-\frac{2}{3}X_2X_5^3+2X_3^2X_5^2-3
X_3X_4X_5+\frac{3}{2}X_4^2.
\end{align*}
Since $\Omega_i\in \mathcal{A}^{(7)}\subseteq \mathcal{A}^{(j)}\subseteq 
\overline{\mathcal{A}},$ we have that $\Omega_1$ and $\Omega_2$ are also Poisson central elements of $\mathcal{A}^{(j)}$ for each $2\leq j\leq 7.$  We proceed to show that  the Poisson centre of $\mathcal{A}$ is a polynomial ring in two variables: $\Omega_1$ and $\Omega_2.$ That is,  $Z_P(\mathcal{A})=\mathbb{K}[\Omega_1,\Omega_2].$ The following discussions   will lead us to the proof.

The set $S_j:=\{\lambda T_j^{i_j}T_{j+1}^{i_{j+1}}\ldots T_6^{i_6}\mid i_j,\ldots,i_6\in \mathbb{N}\}$  is a multiplicative system of non-zero divisors of $\mathcal{A}^{(j)}$ for each $2\leq j\leq 6.$ One can therefore localize $\mathcal{A}^{(j)}$ at $S_j$ as follows:
$$\mathfrak{R}_j:=\mathcal{A}^{(j)}S_j^{-1}.$$
Note that the set $\Sigma_j:=\{T_j^n\mid n\in \mathbb{N}\}$ is also a multiplicative set in both $\mathcal{A}^{(j)}$ and $\mathcal{A}^{(j+1)}$ for each $2\leq j\leq 6.$ It follows from \cite[Prop. 1.11]{sc} that 
$$\mathcal{A}^{(j)}\Sigma_j^{-1}=\mathcal{A}^{(j+1)}\Sigma_j^{-1}.$$
One can easily verify that
\begin{align*}
\mathfrak{R}_j=\mathfrak{R}_{j+1}\Sigma_j^{-1}, \ \ \text{for all} \ \ 2\leq j\leq 6.
\end{align*}
Moreover, the localization $$\mathfrak{R}_1:=\mathfrak{R}_2[T_1^{-1}]$$ also holds in $\mathfrak{R}_2,$ since $T_1$ generates a multiplicative system in $\mathfrak{R}_2$.  In fact, $\mathfrak{R}_1$ is the Poisson torus associated to the Poisson affine space $\overline{\mathcal{A}}.$ As a result, $\mathfrak{R}_1=\mathbb{\mathbb{K}}[T_1^{\pm 1},\ldots, T_6^{\pm 1}],$ where  $\{T_i,T_j\}={\mu_{ij}}T_jT_i$ for all $1\leq i,j\leq 6$.  
The PDDA helps to construct the following embeddings:
\begin{align}
\label{embp}
\mathcal{A}:=\mathfrak{R}_7&\subset \mathfrak{R}_6=\mathfrak{R}_7\Sigma_6^{-1}\subset \mathfrak{R}_5=\mathfrak{R}_6\Sigma_5^{-1}\subset \mathfrak{R}_4=\mathfrak{R}_5\Sigma_4^{-1}\nonumber\\
&\subset \mathfrak{R}_3=\mathfrak{R}_4\Sigma_3^{-1}\subset \mathfrak{R}_2=\mathfrak{R}_3\Sigma_2^{-1}\subset \mathfrak{R}_1.
\end{align}
Note that the family $(X_{1,j}^{k_1}\ldots X_{6,j}^{k_6}),$ where $k_i\in \mathbb{N}$ if $i<j$ and $k_i\in \mathbb{Z}$ otherwise is a PBW-basis of 
$\mathfrak{R}_j$ for all $2\leq j\leq 7.$ In addition, the family $( T_1^{k_1}\ldots T_6^{k_6})_{k_1,\ldots, k_6\in \mathbb{Z}}$ is a basis of $\mathfrak{R}_1.$
\begin{lem}
\label{pl11}
\label{pl1}
\begin{itemize}
\item[1.] $Z_P(\mathfrak{R}_1)=\mathbb{K}[\Omega_1^{\pm 1},\Omega_2^{\pm 1}].$ 
\item[2.]  $Z_P(\overline{\mathcal{A}})=\mathbb{K}[\Omega_1,\Omega_2].$
\item[3.] $Z_P(\mathfrak{R}_3)=\mathbb{K}[\Omega_1,\Omega_2].$
\item[4.] $Z_P(\mathcal{A})=\mathbb{K}[\Omega_1,\Omega_2].$
\end{itemize}
\end{lem}
\begin{proof}
1.  Obviously, $\mathbb{K}[\Omega_1^{\pm 1},\Omega_2^{\pm 1}]\subseteq Z_P(\mathfrak{R}_1).$ For the reverse inclusion, let $y\in Z_P(\mathfrak{R}_1).$ Then, $y$ can be written in terms of the basis of $\mathfrak{R}_1$ as  $y=\sum_{(i,\ldots,n)\in \mathbb{Z}^6}a_{(i,\ldots,n)} T_1^iT_2^{j}T_3^{k}T_4^{l}T_5^{m}T_6^{n}.$ One can verify that $\{y,T_1\}=(-3j-k+m+3n) yT_1.$ Since $y\in Z_P(\mathfrak{R}_1),$ it follows that
 $-3j-k+m+3n=0.$   
Following the same pattern for $T_2, T_3, T_4, T_5$ and $T_6$, one can confirm that
$3i-3k-3l+3n=0, \ $
 $i+3j-3l-m=0, \ $
 $3j+3k-3m-3n=0,\ $
 $-i+k+3l-3n=0,$ and
 $-3i-3j+3l+3m=0.$
 Solving this system of six equations will reveal that
 $i=k=m$ and $j=l=n.$  One can therefore write 
 \begin{equation*}
 y=\displaystyle \sum_{(i,j)\in \mathbb{Z}^2}a_{(i,j)} T_1^iT_2^{j}T_3^{i}T_4^{j}T_5^{i}T_6^{j}
=\sum_{(i,j)\in \mathbb{Z}^2}q^{\bullet}a_{(i,j)} T_1^iT_3^{i}T_5^{i}T_2^{j}T_4^{j}T_6^{j}=
\sum_{(i,j)\in \mathbb{Z}^2}q^{\bullet}a_{(i,j)}\Omega_1^i\Omega_2^j
\in \mathbb{K}[\Omega_1^{\pm 1},\Omega_2^{\pm 1}].
 \end{equation*}

2. Similar argument as in (1) will prove the result.

3. Observe that $\mathbb{K}[\Omega_1,\Omega_2]\subseteq Z_P(\overline{\mathcal{A}})\subseteq Z_P(\mathfrak{R}_3)=\mathbb{K}[\Omega_1,\Omega_2].$ Hence, $Z_P(\mathfrak{R}_3)=\mathbb{K}[\Omega_1,\Omega_2].$

4. Since $\mathfrak{R}_i$ is a localization of $\mathfrak{R}_{i+1},$ it follows that $Z_P(\mathfrak{R}_{i+1})\subseteq Z_P(\mathfrak{R}_i).$ From \eqref{embp}, we have that  $Z_P(\mathcal{A})\subseteq Z_P(\mathfrak{R}_3).$ Observe that
 $\mathbb{K}[\Omega_1,\Omega_2]\subseteq Z_P(\mathcal{A})\subseteq Z_P(\mathfrak{R}_3)=\mathbb{K}[\Omega_1,\Omega_2].$ Hence, $Z_P(\mathcal{A})=\mathbb{K}[\Omega_1,\Omega_2].$
\end{proof}

\section{Maximal and primitive Poisson ideals, and simple quotients of the semiclassical limit of $U_q^+(G_2)$ }
\label{sec3}
\label{sec4}
This section studies the  Poisson $\mathcal{H}$-invariant prime ideals of  $\mathcal{A}=\mathbb{K}[X_1,\ldots,X_6]$ of height one. Obviously, $\langle 0\rangle$ is the only height zero Poisson $\mathcal{H}$-invariant prime ideal of  $\mathcal{A}.$ Now, the Poisson $\mathcal{H}$-invariant prime ideals of at most height one will enable us to study a family of primitive and maximal Poisson ideals of $\mathcal{A},$ and consequently study a family of simple quotients of $\mathcal{A}$. 
\subsection{Height one Poisson $\mathcal{H}$-invariant prime ideals of $\mathcal{A}$}
In this subsection, we study the Poisson $\mathcal{H}$-invariant prime ideals of  $\mathcal{A}=\mathbb{K}[X_1,\ldots,X_6]$ with height one.  
We will begin by showing that  $\langle\Omega_1\rangle$ and $\langle\Omega_2\rangle$
are Poisson prime ideals.
Note that
$\langle \Theta\rangle_R$ will denote an ideal generated by the element $\Theta$ in any the ring $R.$ Where no doubt arises, we will simply write $\langle \Theta\rangle.$

Recall from Subsection \ref{sub4} that there exists a bijection  between $\{ P\in \text{P.Spec}(\mathcal{A}^{(j+1)})\mid P\cap \Sigma_j=\emptyset \}$ and $\{Q\in\text{P.Spec}(\mathcal{A}^{(j)})\mid Q\cap \Sigma_j=\emptyset\}$.  Observe that $\langle T_1\rangle$ and $ \langle T_2\rangle$ are both elements of  P.Spec$(\overline{\mathcal{A}}).$ 
The following result  shows that $\langle T_1\rangle\in$ Im$(\psi)$, and that  $\langle\Omega_1\rangle$ is the Poisson prime ideal of $\mathcal{A}$ such that $\psi (\langle\Omega_1\rangle)=\langle T_1\rangle.$

\begin{lem}
\label{plem1}
$\langle\Omega_1\rangle$ is the Poisson prime ideal of $\mathcal{A}$ such that $\psi (\langle\Omega_1\rangle)=\langle T_1\rangle$.
\end{lem}
\begin{proof}
We will prove this result in several steps by showing that:
\begin{itemize}
\item[1.] $\langle T_1\rangle\in$  P.Spec$(\mathcal{A}^{(3)}).$
\item[2.] $\langle X_{1,4}T_3-\frac{1}{2} T_2\rangle\in$  P.Spec$(\mathcal{A}^{(4)})$ [note that $\langle T_1\rangle[T_3^{-1}]\cap \mathcal{A}^{(4)}= \langle X_{1,4}T_3-\frac{1}{2} T_2\rangle$].
\item[3.] $\langle X_{1,5}T_3-\frac{1}{2} X_{2,5} \rangle\in$ P.Spec$(\mathcal{A}^{(5)})$ [note that $\langle  X_{1,4}T_3-\frac{1}{2} T_2\rangle[T_4^{-1}]\cap \mathcal{A}^{(5)}=$\newline $\langle X_{1,5}T_3-\frac{1}{2} X_{2,5}\rangle$].
\item[4.] $ \langle X_{1,5}T_3-\frac{1}{2} X_{2,5} \rangle [T_5^{-1}]\cap \mathcal{A}^{(6)}=\langle\Omega_1\rangle_{\mathcal{A}^{(6)}},$ hence $\langle \Omega_{1}\rangle_{\mathcal{A}^{(6)}}\in$ P.Spec$(\mathcal{A}^{(6)}).$
\item[5.] $\langle\Omega_1\rangle_{\mathcal{A}^{(6)}} [T_6^{-1}]\cap \mathcal{A}=\langle\Omega_1\rangle_{\mathcal{A}}$, hence $\langle \Omega_1\rangle_{\mathcal{A}}\in$ P.Spec$(\mathcal{A}).$
\end{itemize}
We proceed to prove the above claims.

1. One can easily verify that 
${\mathcal{A}^{(3)}}/\langle T_1\rangle$ is isomorphic to a Poisson affine space of rank 5.  Hence, $\langle T_1\rangle$ is a Poisson prime ideal in  $\mathcal{A}^{(3)}.$

2. Set $I:= \langle X_{1,4}T_3-\frac{1}{2} T_2\rangle.$  One can verify that $\{X_{i,4},I\}\subseteq I$ for all $i=1,\ldots,6.$ Therefore, $I$ is a Poisson ideal in $\mathcal{A}^{(4)}$.
In addition, ${\mathcal{A}^{(4)}}/I$ is isomorphic to a polynomial ring in five variables which is a domain, hence $I$ is a prime ideal. Since $I$ is both Poisson and prime ideal, it is a Poisson prime ideal in $\mathcal{A}^{(4)}.$ 

3. Similar to (2).

4. Observe that $\Omega_1':=X_{1,5}T_3-\frac{1}{2} X_{2,5}=\Omega_1T_5^{-1}$ in $\mathcal{A}^{(6)}[T_5^{-1}].$  Since $\langle\Omega_1'\rangle\in$ P.Spec$(\mathcal{A}^{(5)}),$ we want to show that  $ \langle\Omega_1'\rangle[T_5^{-1}]\cap \mathcal{A}^{(6)}=\langle\Omega_1\rangle_{\mathcal{A}^{(6)}}.$ Observe that $\langle\Omega_1\rangle_{\mathcal{A}^{(6)}}\subseteq \langle\Omega_1'\rangle[T_5^{-1}]\cap \mathcal{A}^{(6)}.$ We establish the reverse inclusion. Let $y\in\langle\Omega_1'\rangle[T_5^{-1}]\cap \mathcal{A}^{(6)}.$
 Then,
$y\in \langle\Omega_1'\rangle[T_5^{-1}].$ There exists $i\in\mathbb{N}$ such that $yT_5^i\in \langle\Omega_1'\rangle.$ Hence, $yT_5^i=\Omega_1'v,$ for some $v\in \mathcal{A}^{(5)}.$ Furthermore, since $\mathcal{A}^{(5)}[T_5^{-1}]=\mathcal{A}^{(6)}[T_5^{-1}],$ there exists $j\in\mathbb{N}$ such that 
$vT_5^j=v'$ for some $v'\in \mathcal{A}^{(6)}.$ It follows from $yT_5^i=\Omega_1'v$  that $yT_5^{i+j}=\Omega_1'vT_5^j=\Omega_1'v'.$ Hence, $yT_5^{\delta}=\Omega_1'T_5v'=\Omega_1v',$ where
$\delta=i+j+1$ (note that $\Omega_1'T_5=\Omega_1$ in $\mathcal{A}^{(6)}$). Let $S=\{s\in\mathbb{N}\mid \ \exists  v'\in \mathcal{A}^{(6)}: y T_5^{s}=\Omega_{1}v'\}.$ Since $\delta\in S,$ we have that  $S\neq \emptyset.$  Let $s=s_0$ be the minimum element of $S$ such that $y T_5^{s_0}=\Omega_{1}v'$  for some $v'\in \mathcal{A}^{(6)}.$
We want to show that $s_0=0.$ Suppose that  $s_0>0.$ Since $T_5$ is irreducible, $y T_5^{s_0}=\Omega_{1}v'$ implies that $T_5$ is a factor of $\Omega_1$ or $v'.$ Clearly, $T_5$ is not a factor of
$\Omega_1,$ hence, it must be a factor of $v'.$
Now $\mathcal{A}^{(6)}$ can be viewed as a free
$\mathbb{K}[ X_{1,6},X_{2,6},X_{3,6},T_4,T_6]-$module with basis $\left( T_5^{\xi}\right) _{\xi\in\mathbb{N}}.$ One can therefore write $v'=\sum_{\xi=1}^n\beta_\xi T_5^{\xi},$ where $\beta_\xi\in \mathbb{K}[ X_{1,6},X_{2,6},X_{3,6},T_4,T_6].$
Returning to $yT_5^{s_0}=\Omega_{1}v',$ we have that $y T_5^{s_0}=\Omega_{1}\sum_{\xi=1}^n\beta_\xi T_5^{\xi}.$ This implies that $y T_5^{s_0-1}=\Omega_{1}v^{\prime\prime},$
where $v^{\prime\prime}=\sum_{\xi=1}^n\beta_\xi T_5^{\xi-1}\in \mathcal{A}^{(6)}.$ Consequently,
$s_0-1\in S,$ a contradiction! Therefore, $s_0=0$ and $y=\Omega_{1}v'\in\langle\Omega_1\rangle_{\mathcal{A}^{(6)}}.$ Hence, $\langle\Omega_1'\rangle[T_5^{-1}]\cap \mathcal{A}^{(6)}\subseteq \langle\Omega_1\rangle_{\mathcal{A}^{(6)}}$ as expected. 

5. The proof is similar to (4). 
\end{proof}

Following similar procedures, one can also prove that $\langle T_2\rangle\in$ Im$(\psi),$ and that  $\langle\Omega_2\rangle$ is the Poisson prime ideal of $\mathcal{A}$ such that $\psi (\langle\Omega_2\rangle)=\langle T_2\rangle$ (the interested reader can check out the details of the proof in \cite{io}). We state only the result in the following lemma.
\begin{lem}
\label{plem2}
  $\langle\Omega_2\rangle$ is the Poisson prime ideal of $\mathcal{A}$ such that $\psi (\langle\Omega_2\rangle)=\langle T_2\rangle$.
\end{lem}
Observe that $\langle T_1,T_2\rangle$ and $\langle T_2,T_3\rangle$ are  Poisson prime ideals of $\mathcal{\overline{A}}.$
In the next lemma, we will show that $\langle T_1,T_2\rangle, \langle T_2,T_3\rangle\in \psi(\text{P.Spec}(\mathcal{A}))$. 
\begin{lem}
\label{ht2}
$\langle T_1,T_2\rangle, \langle T_2,T_3\rangle\in \psi(\text{P.Spec}(\mathcal{A})).$
\end{lem}
\begin{proof}
 Let $J^{(j)}\in$ P.Spec$(\mathcal{A}^{(j)})$ for all $2\leq j\leq 6.$ We already know that $J^{(j+1)}=J^{(j)}[T_{j}^{-1}]\cap \mathcal{A}^{(j+1)}\in$ P.Spec$(\mathcal{A}^{(1+j)})$ provided $T_{j}\not\in J^{(j)}.$ Note that
 $J^{(7)}=J$ and $\mathcal{A}^{(7)}=\mathcal{A}.$
 
 We begin by showing that $\langle T_1,T_2\rangle\in \psi(\text{P.Spec}(\mathcal{A})).$\newline  Set $J_{1,2}^{(3)}:=\langle T_1,T_2\rangle \in$ P.Spec$(\mathcal{A}^{(3)})$. Observe that $T_3\not\in J_{1,2}^{(3)}.$ Therefore, $J_{1,2}^{(3)}[T_3^{-1
}]\cap \mathcal{A}^{(4)}=J_{1,2}^{(4)}.$ Suppose that $T_4\in J_{1,2}^{(4)}.$ Then, since
 $J_{1,2}^{(3)}[T_3^{-1}]=J_{1,2}^{(4)}[T_3^{-1}],$ we have that $T_4\in J_{1,2}^{(3)}[T_3^{-1}]\cap \mathcal{A}^{(4)}=J_{1,2}^{(4)}[T_3^{-1}]\cap \mathcal{A}^{(3)}=J_{1,2}^{(3)},$ a contradiction! Therefore, $T_4\not\in J_{1,2}^{(4)}.$ Hence, $J_{1,2}^{(4)}[T_4^{-1}]\cap \mathcal{A}^{(5)}=J_{1,2}^{(5)}.$ Suppose that $T_5\in J_{1,2}^{(5)}.$ Then, 
 $T_5\in J_{1,2}^{(4)}[T_4^{-1}]\cap \mathcal{A}^{(5)}=J_{1,2}^{(5)}[T_4^{-1}]\cap \mathcal{A}^{(4)}=J_{1,2}^{(4)},$ a contradiction! Therefore, $T_5\not\in J_{1,2}^{(5)}.$
Hence, $J_{1,2}^{(5)}[T_5^{-1}]\cap \mathcal{A}^{(6)}=J_{1,2}^{(6)}.$ Similarly, one can show that $T_6\not\in J_{1,2}^{(6)}.$ Hence, $J_{1,2}^{(6)}[T_6^{-1}]\cap \mathcal{A}=J_{1,2}.$ Therefore, there exists 
$J_{1,2}\in \text{P.Spec}(\mathcal{A})$ such that $\psi (J_{1,2})=\langle T_1,T_2\rangle.$ Similar arguments will also show that there exists $J_{2,3}\in \text{P.Spec}(\mathcal{A})$ such that
$\psi(J_{2,3})=\langle T_2,T_3\rangle.$   
\end{proof}

Recall that $\Omega_1=T_1T_3T_5$ and $\Omega_2=T_2T_4T_6$ in $\overline{\mathcal{A}}.$ Observe that
$\Omega_1,\Omega_2$ are both elements of $\langle T_1,T_2\rangle$ and $\langle T_2,T_3\rangle.$
From Lemma \ref{ht2}, we know that  $J_{1,2}$ and 
$J_{2,3}$ are the elements of P.Spec$(\mathcal{A})$ such that 
$\psi(J_{1,2})=\langle T_1,T_2\rangle$ and $\psi(J_{2,3})=\langle T_2,T_3\rangle.$ In the next lemma, we show that $J_{1,2}$ and $J_{2,3}$  contain  $\Omega_1$ and $\Omega_2.$ 
\begin{lem}
\label{lcl}
$\Omega_1$ and $\Omega_2$ are elements of $J_{1,2}$ and $J_{2,3}.$ 
\end{lem}
\begin{proof}
Recall that $\Omega_1$ and $\Omega_2$ are central elements of $\mathcal{A}^{(j)}$ for all $2\leq j\leq 7.$
Given the set-up in the proof of Lemma \ref{ht2}, we know that 
$\Omega_1, \Omega_2\in$ $J_{1,2}^{(3)}=\langle T_1,T_2\rangle.$  By an induction on $j$ (where $j=3,4,5,6$), one can confirm that  
$\Omega_1, \Omega_2\in J_{1,2}^{(j)}[T_j^{-1
}]\cap \mathcal{A}^{(j+1)}=J_{1,2}^{(j+1)}$. Therefore; $\Omega_1, \Omega_2\in J_{1,2}.$ 
Similarly, one can easily show  that $\Omega_1,\Omega_2\in J_{2,3}^{(4)}=
\langle T_2, T_3\rangle.$ 
\end{proof}
We now want to find the height one Poisson $\mathcal{H}$-invariant prime ideals of $\mathcal{A},$ and show that the height $2$ Poisson $\mathcal{H}$-invariant prime ideals of $\mathcal{A}$ contain those of height one. 

Let the torus $\mathcal{H}:=(\mathbb{K}^*)^2$ acts by Poisson automorphisms on $\mathcal{A}$ via:
\begin{align*}
h\cdot X_1&=\alpha_1X_1& h\cdot X_2&=\alpha_1^3\alpha_6X_2& h\cdot X_3&=\alpha_1^2\alpha_6X_3\\
h\cdot X_4&=\alpha_1^3\alpha_6^2X_4&h\cdot X_5&=\alpha_1\alpha_6X_5&h\cdot X_6&=\alpha_6X_6,
\end{align*}
for all $h:=(\alpha_1,\alpha_6)\in \mathcal{H}.$ This $\mathcal{H}$-action is rational. Furthermore,  $\Omega_1$ and $\Omega_2$ are Poisson $\mathcal{H}$-eigenvectors. Hence, $\langle\Omega_1\rangle$ and $\langle\Omega_2\rangle$ are Poisson $\mathcal{H}$-invariant prime ideals.

Set $\theta_1:=$ id$_{\text{P.Spec}(\overline{\mathcal{A}})}.$ For all $2\leq j\leq 6,$ define 
$\theta_j:=\theta_{j-1}\circ \psi_j.$ Then, $\theta_j:$ P.Spec$(\mathcal{A}^{(j+1)})\hookrightarrow$ P.Spec$(\overline{\mathcal{A}}).$ The map $\theta_j$ is injective \cite[\S 2.3]{sc}. Let $\langle X_{j,j+1}\rangle_P$ denote the smallest Poisson ideal in $\mathcal{A}^{(1+j)}$ containing $X_{j,j+1}.$ It follows from \cite[Lemma 2.2]{sc} that there is a surjective Poisson algebra homomorphism $g_j:\mathcal{A}^{(j)}\longrightarrow \mathcal{A}^{(j+1)}/\langle X_{j,j+1}\rangle_P$ given by
$g_j(X_{i,j})=\overline{X_{i,j+1}}:=X_{i,j+1}+\langle X_{j,j+1}\rangle_P$ for all $1\leq i\leq 6.$ Denote the kernel of $g_j$ by ker$(g_j)$ and the image of $\psi$ by Im$(\psi)$. We have the following lemma.  
\begin{lem}\cite[Proposition 2.8]{sc}
\label{pl3}
Let $P\in$ P.Spec$(\overline{\mathcal{A}}).$ The following are equivalent:
\begin{itemize}
\item[•] $P\in$ Im$(\psi),$
\item[•] for all $2\leq j\leq 6,$ we have that $P\in$ Im$(\theta_{j-1})$ and either
$X_{j,j}=X_{j,j+1}\not\in\theta_{j-1}^{-1}(P)$ or ker$(g_j)\subseteq \theta_{j-1}^{-1}(P).$
\end{itemize}
\end{lem}

Note that the map $\psi$ induces a canonical embedding from $\mathcal{H}$-P.Spec$(\mathcal{A})$ to $\mathcal{H}$-P.Spec$(\overline{\mathcal{A}}).$ Observe that $\{\langle T_i\rangle\mid i=1,\ldots, 6 \}$ is the set of only height one Poisson $\mathcal{H}$-invariant prime ideals we have in $\overline{\mathcal{A}}.$ Since 
 $\psi$ preserves the height of a prime ideal, if $\psi^{-1}(\langle T_i\rangle)\in$  
P.Spec$(\mathcal{A}),$ then it is a height one Poisson $\mathcal{H}$-invariant prime ideal in $\mathcal{A}$ for all $1\leq i\leq 6.$ For example, we already know that $\psi^{-1}(\langle T_1\rangle)=\langle\Omega_1\rangle$ and $\psi^{-1}(\langle T_2\rangle)=\langle\Omega_2\rangle$. Therefore, $\langle\Omega_1\rangle$ and $\langle\Omega_2\rangle$ are height one Poisson $\mathcal{H}$-invariant prime ideals in $\mathcal{A}.$ We will show in the next lemma that $\langle\Omega_1\rangle$ and $\langle\Omega_2\rangle$ are the only height one Poisson $\mathcal{H}$-invariant prime ideals  in $\mathcal{A}.$
\begin{lem} \label{chr} For each $j\in \{3,4,5,6\},$
$\langle T_j\rangle\not\in\psi(\text{P.Spec}(\mathcal{A})).$ 
\end{lem}
\begin{proof}
Suppose that there exists $j\in \{3,4,5,6\}$ such that $\langle T_j\rangle\in\psi(\text{P.Spec}(\mathcal{A})).$ Then, there exists $P\in$ P.Spec$(\mathcal{A})$ such that $\psi(P)=\langle T_j\rangle_{\mathcal{A}^{(j)}},$ where 
$\psi:=\psi_j\circ \cdots\circ \psi_6.$ Set 
$P^{(j)}:=\langle T_j\rangle_{\mathcal{A}^{(j)}}.$  Since $T_j\in P^{(j)},$ it follows from Lemma \ref{pl3} that ker$(g_j)\subseteq P^{(j)}.$ The rest follows in cases.
\begin{itemize}
\item[•] When $j=3,$ then ker$(g_3)\subseteq P^{(3)}=\langle T_3\rangle_{\mathcal{A}^{(3)}}.$ One can easily deduce from the Poisson bracket of $\mathcal{A}$ that  $\{X_{3,4},X_{1,4}\}=-X_{1,4}X_{3,4}-X_{2,4}.$
This implies that $X_{2,4}=-\{X_{3,4},X_{1,4}\}-X_{1,4}X_{3,4}\in \langle X_{3,4}\rangle_{P}=\langle T_{3}\rangle_{P}.$  It follows that $g_3(X_{2,3})=X_{2,4}+\langle T_{3}\rangle_{P}=\overline{0}.$ Hence, $X_{2,3}=T_2\in$ ker$(g_3)\subseteq \langle T_3\rangle_{\mathcal{A}^{(3)}},$ a contradiction! 
Consequently, $\langle T_3\rangle\not\in\psi(\text{P.Spec}(\mathcal{A})).$
\item[•]When $j=4,$ then ker$(g_4)\subseteq P^{(4)}=\langle T_4\rangle_{\mathcal{A}^{(4)}}.$ From the Poisson bracket of $\mathcal{A},$ we have that 
$X_{3,5}^2=-\frac{1}{2}\{X_{4,5}, X_{1,5}\}\in\langle X_{4,5}\rangle_P=\langle T_4\rangle_P.$ Since $g_4$ is a homomorphism, it follows that $g_4(X_{3,4}^2)=(g_4(X_{3,4}))^2=X_{3,5}^2+\langle T_4\rangle_P=\overline{0}.$ 
Therefore, $X_{3,4}^2=T_3^2\in$ ker$(g_4)\subseteq \langle T_4\rangle_{\mathcal{A}^{(4)}},$ a contradiction! Hence, $\langle T_4\rangle\not\in\psi(\text{P.Spec}(\mathcal{A})).$
\item[•] When $j=5$, then ker$(g_5)\subseteq P^{(5)}=\langle T_5\rangle_{\mathcal{A}^{(5)}}.$ One can deduce from the Poisson bracket of $\mathcal{A}$ that  $X_{4,6}=-\frac{1}{3}X_{3,6}X_{5,6}-\frac{1}{3}\{X_{5,6}, X_{3,6}\}\in \langle X_{5,6}\rangle_P=\langle T_5\rangle_P.$ It follows that
$g_5(X_{4,5})=X_{4,6}+\langle T_5\rangle_P=\overline{0}.$ Hence, 
$X_{4,6}=T_4\in$ ker$(g_5)\subseteq \langle T_5\rangle_{\mathcal{A}^{(5)}},$ a contradiction! Therefore, $\langle T_5\rangle\not\in\psi(\text{P.Spec}(\mathcal{A})).$
\item[•] Finally, when $j=6,$ then ker$(g_6)\subseteq P^{(6)}=\langle T_6\rangle_{\mathcal{A}^{(6)}}.$ Similarly, one can verify that  $X_5=X_1X_6-\frac{1}{3}\{X_6,X_1\}\in \langle X_6\rangle_P=\langle T_6\rangle_P.$ Now,
$g_6(X_{5,6})=X_5+\langle T_6\rangle_P=\overline{0}.$ Therefore,
$X_{5,6}=T_5\in$ ker$(g_6)\subseteq \langle T_6\rangle_{\mathcal{A}^{(6)}},$ a contradiction! Hence, $\langle T_6\rangle\not\in\psi(\text{P.Spec}(\mathcal{A})).$
\end{itemize} 
In conclusion, $\langle T_j\rangle\not\in\psi(\text{P.Spec}(\mathcal{A}))$ for all $j\in \{3,4,5,6\}$  as expected. 
\end{proof}
The corollary below is deduced from  the proof of Lemma \ref{chr}. 

\begin{cor}
\label{ccc}
Let  $P\in\psi($P.Spec$(\mathcal{A})).$ If $T_j\in P,$ then  $T_{j-1}, \ldots, T_2\in P$ for all $3\leq j\leq 6.$ 
\end{cor}
Recall from Lemma \ref{ht2} that there exist $J_{1,2}$ and $J_{2,3}$ of 
$\text{P.Spec}(\mathcal{A})$ such that 
$\psi(J_{1,2})=\langle T_1, T_2\rangle$ and $\psi(J_{2,3})=\langle T_2, T_3\rangle.$
As a result of Corollary \ref{ccc}, the Poisson ideals $\langle T_1, T_2\rangle$ and $ \langle T_2, T_3\rangle$ are the only height two Poisson $\mathcal{H}$-invariant prime ideals  of $\psi (\text{P.Spec}(\mathcal{A})).$
 Since  $\psi$ preserves Poisson $\mathcal{H}$-invariant prime ideals and the height of a Poisson prime ideal, this  implies that $J_{1,2}$ and $J_{2,3}$ are the only height two Poisson $\mathcal{H}$-invariant prime ideals of $\mathcal{A}.$ It follows from Lemma \ref{lcl} that the height two Poisson 
 $\mathcal{H}$-invariant prime ideals of  $\mathcal{A}$ contain 
  $\Omega_1$ and $\Omega_2.$ 
\begin{rem}
\label{pr1}
Since the height two Poisson 
$\mathcal{H}$-invariant prime ideals of  $\mathcal{A}$ contain 
$\Omega_1$ and $\Omega_2,$  every non-zero Poisson $\mathcal{H}$-invariant prime ideal of $\mathcal{A}$ will contain either $\Omega_1$ or $\Omega_2.$ 
Note that those Poisson $\mathcal{H}$-invariant primes of at least height 2 will contain both $\Omega_1$ and $\Omega_2.$
\end{rem}

\subsection{Some maximal and primitive Poisson ideals of $\mathcal{A}$} 
We begin this subsection by finding the $\mathcal{H}$-strata corresponding to $\langle 0\rangle,$ $\langle\Omega_1\rangle$ and  $\langle\Omega_2\rangle$.    Note that all Poisson ideals in $\mathcal{A}$ shall be written as $\langle \Theta\rangle,$ where $\Theta\subset \mathcal{A}.$ However, if we want to refer to a Poisson ideal in any other Poisson algebra, say $R$, then that Poisson ideal shall be written as $\langle \Theta\rangle_R,$ where in this case, $\Theta\subset R.$

\begin{pro}
\label{pp0}
Assume that $\mathbb{K}$ is algebraically closed. Let $\mathcal{P}$ be the set of those monic irreducible polynomials 
$P(\Omega_1,\Omega_2)\in \mathbb{K}[\Omega_1,\Omega_2]$
 with $P(\Omega_1,\Omega_2)\neq \Omega_1$ and $P(\Omega_1,\Omega_2)\neq \Omega_2$.
 Then P.Spec$_{\langle 0\rangle}(\mathcal{A})=
\{\langle 0\rangle\}\cup \{\langle P(\Omega_1,\Omega_2)\rangle\mid P(\Omega_1,\Omega_2)\in \mathcal{P}\}\cup \{\langle \Omega_1-\alpha,\Omega_2-\beta\rangle\mid\alpha,\beta\in \mathbb{K}^*\}.$  
\end{pro}
\begin{proof}
We claim that P.Spec$_{\langle 0\rangle}(\mathcal{A})\subseteq \{Q \in \text{P.Spec}(\mathcal{A})\mid\Omega_1,\Omega_2\not\in Q\}.$ To establish this claim, let us assume that this is not the case. Suppose that there exists $ Q\in$ P.Spec$_{\langle 0\rangle}(\mathcal{A})$ such that $\Omega_1$ or $\Omega_2$ belongs to $ Q;$ then the product $\Omega_1\Omega_2$ which is Poisson $\mathcal{H}$-eigenvector belongs to $Q.$ Consequently, $\Omega_1\Omega_2\in \bigcap_{h\in \mathcal{H}}h\cdot Q=\langle 0\rangle,$ a contradiction! Hence, P.Spec$_{\langle 0\rangle}(\mathcal{A})\subseteq \{Q \in \text{P.Spec}(\mathcal{A})\mid\Omega_1,\Omega_2\not\in Q\}.$ Conversely, suppose that $Q\in$ P.Spec$(\mathcal{A})$  such that $\Omega_1,\Omega_2\not\in Q,$ then $\bigcap_{h\in \mathcal{H}}h\cdot Q$ is a Poisson $\mathcal{H}$-invariant prime ideal of $\mathcal{A},$ which contains neither $\Omega_1$ nor $\Omega_2.$ The only possibility for $\bigcap_{h\in \mathcal{H}}h\cdot Q$ is the zero ideal since every non-zero Poisson $\mathcal{H}$-invariant prime ideal of $\mathcal{A}$ contains $\Omega_1$ or $ \Omega_2$ (Remark \ref{pr1}). Thus, $\bigcap_{h\in \mathcal{H}}h\cdot Q=\langle 0\rangle.$ Hence, $Q\in$ P.Spec$_{\langle 0\rangle}(\mathcal{A}).$ Therefore, $\{Q \in \text{P.Spec}(\mathcal{A})\mid\Omega_1,\Omega_2\not\in Q\}\subseteq$ P.Spec$_{\langle 0\rangle}(\mathcal{A})$, and so we have proved that P.Spec$_{\langle 0\rangle}(\mathcal{A})= \{Q \in \text{P.Spec}(\mathcal{A})\mid\Omega_1,\Omega_2\not\in Q\}.$

Since $\Omega_1,\Omega_2\in Z_P(\mathcal{A}),$ we have that the set $\{\Omega_1^i\Omega_2^j\mid i,j\in\mathbb{N}\}$ is a multiplicative set in $\mathcal{A}.$ We can now localize $\mathcal{A}$ as $R:=\mathcal{A}[\Omega_1^{-1},\Omega_2^{-1}].$ Given
$Q\in$ P.Spec$_{\langle 0\rangle}(\mathcal{A}),$ the map
 $\phi:Q\mapsto
Q[\Omega_1^{-1},\Omega_2^{-1}]$ is an increasing bijection from 
P.Spec$_{\langle 0\rangle}(\mathcal{A})$ onto P.Spec($R$).

Let us verify that $R$ is Poisson $\mathcal{H}$-simple before we describe P.Spec($R$). Now, $\phi$ still induces a bijection between the set of those Poisson $\mathcal{H}$-invariant prime ideals of P.Spec$_{\langle 0\rangle}(\mathcal{A})$ and the set of Poisson $\mathcal{H}$-invariant prime ideals of P.Spec$(R)$.
It is already known that the set of Poisson $\mathcal{H}$-invariant prime ideals of $\mathcal{A}$ that contains neither $\Omega_1$ nor $\Omega_2$ consists only of the zero ideal $\{\langle 0\rangle\}$ (Remark \ref{pr1}). This implies that $\langle 0\rangle_R$ is the only Poisson $\mathcal{H}$-invariant prime ideal of $R.$ Every Poisson $\mathcal{H}$-invariant proper ideal of $R$ is contained in a 
Poisson $\mathcal{H}$-invariant prime ideal of $R$. Therefore, $\langle 0\rangle_R$ is the only unique Poisson $\mathcal{H}$-invariant proper ideal of $R.$ This confirms that $R$ is Poisson $\mathcal{H}$-simple.
It follows from \cite[Theorem 4.2]{gd} that the extension and contraction maps provide mutually inverse bijections between P.Spec($R$) and Spec($Z_P(R)$). From Lemma \ref{pl1}, $Z_P(\mathcal{A})=\mathbb{K}[\Omega_1,\Omega_2],$ and so $Z_P(R)=\mathbb{K}[\Omega_1^{\pm 1},\Omega_2^{\pm 1}].$ Since $\mathbb{K}$ is algebraically closed, we have that 
Spec$(Z_P(R))=\{\langle 0\rangle_{Z_P(R)}\}\cup \{\langle P(\Omega_1,\Omega_2)\rangle_{Z_P(R)}\mid P(\Omega_1,\Omega_2)\in \mathcal{P}\}\cup \{\langle \Omega_1-\alpha,\Omega_2-\beta\rangle_{Z_P(R)}\mid\alpha,\beta\in \mathbb{K}^*\}.$ One can now recover P.Spec($R$) from Spec($Z_P(R)$) as follows:
P.Spec($R$)$=\{\langle 0\rangle_{R}\}\cup \{\langle P(\Omega_1,\Omega_2)\rangle_{R}\mid P(\Omega_1,\Omega_2)\in \mathcal{P}\}\cup \{\langle \Omega_1-\alpha,\Omega_2-\beta\rangle_R\mid\alpha,\beta\in \mathbb{K}^*\}.$  It follows that
P.Spec$_{\langle 0\rangle}(\mathcal{A})=\{\langle 0\rangle_{R}\cap \mathcal{A}\}\cup \{\langle P(\Omega_1,\Omega_2)\rangle_{R}\cap \mathcal{A}\mid P(\Omega_1,\Omega_2)\in \mathcal{P}\}\cup \{\langle \Omega_1-\alpha,\Omega_2-\beta\rangle_R\cap \mathcal{A}\mid \alpha,\beta\in \mathbb{K}^*\}.$

Undoubtedly, $\langle 0\rangle_{R}\cap \mathcal{A}=\langle 0 \rangle.$ We now have to show that $\langle P(\Omega_1,\Omega_2)\rangle_{R}\cap \mathcal{A}=\langle P(\Omega_1,\Omega_2)\rangle,$ $\forall P(\Omega_1,\Omega_2)\in \mathcal{P},$ and $\langle \Omega_1-\alpha,\Omega_2-\beta\rangle_R\cap \mathcal{A}=\langle \Omega_1-\alpha,\Omega_2-\beta\rangle,$ $\forall \alpha,\beta\in \mathbb{K}^*,$ in order to complete the proof.

Fix $P(\Omega_1,\Omega_2)\in \mathcal{P}.$ Clearly, $\langle P(\Omega_1,\Omega_2) \rangle \subseteq \langle P(\Omega_1,\Omega_2)\rangle_{R}\cap \mathcal{A}.$ To show the reverse inclusion,
let $y\in \langle P(\Omega_1,\Omega_2)\rangle_{R}\cap \mathcal{A}$. Since $y\in \langle P(\Omega_1,\Omega_2) \rangle_R,$ this implies that $y=dP(\Omega_1,\Omega_2)$ for some $ d\in R.$  Also, 
$d\in R$ implies that there exists $ i,j\in \mathbb{N}$ such that $d=a\Omega_1^{-i}\Omega_2^{-j},$ where $a\in \mathcal{A}.$ Therefore, 
$y=a\Omega_1^{-i}\Omega_2^{-j}P(\Omega_1,\Omega_2), $ which implies that $y\Omega_1^i\Omega_2^j=aP(\Omega_1,\Omega_2).$ Choose  $(i,j)\in \mathbb{N}^2$ minimal (in the lexicographic order on $\mathbb{N}^2$) such that the equality holds. Without loss of generality, let us suppose that $i>0,$ then 
$aP(\Omega_1,\Omega_2)\in \langle\Omega_1\rangle.$ Since $\langle\Omega_1\rangle$ is a prime ideal, this implies that $a\in \langle\Omega_1\rangle$ or $P(\Omega_1,\Omega_2)\in \langle\Omega_1\rangle.$ Since $P(\Omega_1,\Omega_2)\in \mathcal{P},$ this implies that  $P(\Omega_1,\Omega_2)\not\in \langle\Omega_1\rangle.$ Hence, $a\in \langle \Omega_1\rangle,$ which implies that $a=t\Omega_1$ for some $t\in \mathcal{A}.$ Returning to $y\Omega_1^i\Omega_2^j=aP(\Omega_1,\Omega_2),$ we have that $y\Omega_1^i\Omega_2^j=t\Omega_1 P(\Omega_1,\Omega_2).$ Finally, $y\Omega_1^{i-1}\Omega_2^j=tP(\Omega_1,\Omega_2).$ This clearly contradicts the minimality of $(i,j),$ hence $(i,j)=(0,0).$ As a result, 
$y=aP(\Omega_1,\Omega_2)=\langle P(\Omega_1,\Omega_2)\rangle.$ Consequently,
$\langle P(\Omega_1,\Omega_2)\rangle_{R}\cap \mathcal{A}=\langle P(\Omega_1,\Omega_2)\rangle$ for all $P(\Omega_1,\Omega_2)\in \mathcal{P}$ as desired.

Similarly, we show that  $\langle \Omega_1-\alpha,\Omega_2-\beta\rangle_R\cap \mathcal{A}=\langle \Omega_1-\alpha,\Omega_2-\beta\rangle;$ $\forall \alpha,\beta\in \mathbb{K}^*.$ Fix $ \alpha,\beta\in \mathbb{K}^*.$ Obviously, 
$\langle \Omega_1-\alpha,\Omega_2-\beta\rangle\subseteq\langle \Omega_1-\alpha,\Omega_2-\beta\rangle_R\cap \mathcal{A}.$ We establish the reverse inclusion. Let $y\in \langle \Omega_1-\alpha,\Omega_2-\beta\rangle_R \cap \mathcal{A}.$ Since $y\in \langle \Omega_1-\alpha,\Omega_2-\beta\rangle_R,$ we have that $y=m_0(\Omega_1-\alpha)+n_0(\Omega_2-\beta),$ where $m_0,n_0\in R.$  Also, $m_0,n_0\in R$ implies that there exists 
 $i,j\in\mathbb{N}$ such that 
$m_0=m\Omega_1^{-i}\Omega_2^{-j}$ and $n_0=n\Omega_1^{-i}\Omega_2^{-j}$ for some $m,n\in \mathcal{A}.$ Therefore, $y=m\Omega_1^{-i}\Omega_2^{-j}(\Omega_1-\alpha)+n\Omega_1^{-i}\Omega_2^{-j}(\Omega_2-\beta)$, which implies that $ y\Omega_1^{i}\Omega_2^{j}=m(\Omega_1-\alpha)+n(\Omega_2-\beta).$ Choose  $(i,j)\in \mathbb{N}^2$ minimal (in the lexicographic order on $\mathbb{N}^2$) such that the equality holds. Without loss of generality, suppose that 
$i>0$ and let $f:\mathcal{A}\longrightarrow {\mathcal{A}}/{\langle \Omega_2-\beta\rangle}$ be a canonical surjection.
We have that $f(y)f(\Omega_1)^if(\Omega_2)^j=f(m)f(\Omega_1-\alpha).$ It follows that
$f(m)f(\Omega_1-\alpha)\in \langle f(\Omega_1)\rangle.$ 
Note that $f(\Omega_1-\alpha)\not\in \langle f(\Omega_1)\rangle.$  This implies that $f(m)\in \langle f(\Omega_1)\rangle.$  Therefore, 
$\exists \lambda\in \mathcal{A}$ such that $f(m)=f(\lambda)f(\Omega_1).$
Consequently, $f(y)f(\Omega_1)^if(\Omega_2)^j=f(\lambda)f(\Omega_1)f(\Omega_1-\alpha).$  Since
$f(\Omega_1)\neq 0,$ this implies that $f(y)f(\Omega_1)^{i-1}f(\Omega_2)^j=f(\lambda)f(\Omega_1-\alpha).$ Therefore, $y\Omega_1^{i-1}\Omega_2^j=\lambda(\Omega_1-\alpha)+\lambda^\prime(\Omega_2-\beta),$ for some $\lambda^\prime\in \mathcal{A}$.  This contradicts the minimality of $(i,j).$ Hence, $(i,j)=(0,0)$ and so $y=m(\Omega_1-\alpha)+n(\Omega_2-\beta)\in \langle \Omega_1-\alpha,\Omega_2-\beta\rangle.$ In conclusion, $\langle \Omega_1-\alpha,\Omega_2-\beta\rangle_R\cap \mathcal{A}=\langle \Omega_1-\alpha,\Omega_2-\beta\rangle,$ $\forall \alpha,\beta\in \mathbb{K}^*.$  
\end{proof}
One can also proove the following results in a similar manner. The details of the proof can be deduced from \cite[\S 2.4]{io}. 
\begin{pro}
\label{p1p1}
\label{pp2}
Assume that $\mathbb{K}$ is algebraically closed.
\begin{itemize}
\item[1.] P.Spec$_{\langle\Omega_1\rangle}(\mathcal{A})=
\{\langle\Omega_1\rangle\}\cup \{\langle\Omega_1,\Omega_2-\beta\rangle\mid \beta\in \mathbb{K}^*\}.$
\item[2.]  P.Spec$_{\langle\Omega_2\rangle}(\mathcal{A})=
\{\langle\Omega_2\rangle\}\cup \{\langle\Omega_1-\alpha,\Omega_2\rangle\mid\alpha\in \mathbb{K}^*\}.$
\end{itemize}
\end{pro}

\begin{cor}
\label{corPoissonprimitive}
The Poisson ideal $\langle \Omega_1-\alpha,\Omega_2-\beta\rangle$ is Poisson primitive in $\mathcal{A}$ for each $(\alpha,\beta)\in \mathbb{K}^2\setminus \{(0,0)\}.$ 
\end{cor}
\begin{proof}
Since the Poisson ideal $\langle \Omega_1-\alpha,\Omega_2-\beta\rangle$ is maximal in its respective strata for each $(\alpha,\beta)\in \mathbb{K}^2\setminus \{(0,0)\},$ it is also  Poisson primitive (see Prop. \ref{ppf1}).
\end{proof}

While we assumed that $\mathbb{K}$ is algebraically closed in the above propositions and corollary, Corollary \ref{corPoissonprimitive} is also true when we drop this assumption (with the same proof). 

\begin{pro}
\label{c3p2}
Let $(\alpha, \beta)\in \mathbb{K}^2\setminus \{(0,0)\}.$ The Poisson prime ideal
$\langle \Omega_1-\alpha,\Omega_2-\beta\rangle$
is maximal  in $\mathcal{A}.$
\end{pro}
\begin{proof}
Suppose that there exists a maximal Poisson ideal $I$ of $\mathcal{A}$ such that $\langle \Omega_1-\alpha,\Omega_2-\beta\rangle\varsubsetneq I\varsubsetneq \mathcal{A}.$ Let $J$ be the Poisson $\mathcal{H}$-invariant prime ideal in $\mathcal{A}$ such that $I\in$ P.Spec$_J(\mathcal{A}).$ By Propositions \ref{pp0} and \ref{pp2}, $J$ cannot be $\langle 0\rangle,$  
$\langle\Omega_1\rangle$ or $\langle\Omega_2\rangle,$ since either of these will lead to a contradiction.  Every non-zero Poisson $\mathcal{H}$-invariant prime ideal contains  only $\Omega_1$ or  only  $\Omega_2$  or both (Remark \ref{pr1}). Since $J\neq \langle \Omega_1\rangle, \langle \Omega_2\rangle,$ this implies that $J$ contains both $\Omega_1$ and $\Omega_2.$ Moreover, since $ J\subseteq I,$ this implies that $\Omega_1, \Omega_2\in I.$  Given  $\langle \Omega_1-\alpha,\Omega_2-\beta\rangle\subset I$, we have that $\Omega_1-\alpha,\Omega_2-\beta\in I.$ It follows that $\alpha, \beta\in I,$ hence   $I=\mathcal{A},$ a contradiction!  This confirms that $\langle \Omega_1-\alpha,\Omega_2-\beta\rangle$  is a maximal Poisson ideal in $\mathcal{A}.$
\end{proof}
\subsection{Simple quotients of the semiclassical limit of $U_q^+(G_2)$}
\label{sec6.4}
Given that $\Omega_1-\alpha$ and $\Omega_2-\beta$ generate a maximal Poisson prime ideal of $\mathcal{A},$ the factor algebra
$$\mathcal{A}_{\alpha,\beta}:=\frac{\mathcal{A}}{\langle\Omega_1-\alpha,\Omega_2-\beta\rangle}$$ is a Poisson-simple noetherian domain for all $(\alpha,\beta)\in \mathbb{K}^2\setminus \{(0,0)\}.$ 
Denote the canonical image of $X_i$ by 
$x_i:=X_i+\langle\Omega_1-\alpha,\Omega_2-\beta\rangle$ for each $1\leq i\leq 6.$ The algebra $\mathcal{A}_{\alpha,\beta}$ is commutative, and  satisfies the following two relations:
\begin{align}
x_1x_3x_5&-\frac{3}{2}x_1x_4-\frac{1}{2}x_2x_5+\frac{1}{2}x_3^2=\alpha,\label{pe1e}\\
x_2x_4x_6&-\frac{2}{3}x_3^3x_6-\frac{2}{3}x_2x_5^3+2x_3^2x_5^2-3
x_3x_4x_5+\frac{3}{2}x_4^2=\beta.\label{pe2e}
\end{align}
We also have the following extra relations in $\mathcal{A}_{\alpha,\beta},$ which can be verified through direct computations. 
\begin{lem}
\label{pl2}
\begin{align*}
(1)\ x_3^2=&2\alpha+3x_1x_4+x_2x_5-2x_1x_3x_5.\\
 \\
(2)\ x_4^2=&\frac{2}{3}\beta+\frac{8}{9}\alpha x_3x_6+\frac{4}{3}x_1x_3x_4x_6+\frac{4}{9}x_2x_3x_5x_6
-\frac{16}{9}\alpha x_1x_5x_6-\frac{8}{3}x_1^2x_4x_5x_6\\&+
\frac{16}{9}x_1^2x_3x_5^2x_6-\frac{8}{9}x_2x_5^3-\frac{8}{3}\alpha x_5^2-4x_1x_4x_5^2
+\frac{8}{3}x_1x_3x_5^3+2x_3x_4x_5-\frac{2}{3}x_2x_4x_6\\
&-\frac{8}{9}x_1x_2x_5^2x_6.
\\
\\
(3)\ x_3^2x_4&=2\alpha x_4+x_2x_4x_5+2\beta x_1+\frac{8}{3}\alpha x_1x_3x_6+4x_1^2x_3x_4x_6+\frac{4}{3}x_1x_2x_3x_5x_6\\&
-8x_1^3x_4x_5x_6-\frac{8}{3}x_1^2x_2x_5^2x_6+\frac{16}{3}x_1^3x_3x_5^2x_6-\frac{8}{3}x_1x_2x_5^3
-8\alpha x_1x_5^2-12x_1^2x_4x_5^2\\&+8x_1^2x_3x_5^3+4x_1x_3x_4x_5-2x_1x_2x_4x_6-\frac{16}{3}\alpha x_1^2x_5x_6.
\end{align*}
\begin{align*}
(4) \ x_3x_4^2=\frac{2}{3}&\beta x_3+\frac{16}{9}\alpha^2 x_6+\frac{16}{3}\alpha x_1x_4x_6+
\frac{16}{9}\alpha  x_2x_5x_6+\frac{16}{9}\alpha x_1x_3x_5x_6
+\frac{4}{9}x_2^2x_5^2x_6\\+&\frac{8}{9}x_1x_2x_3x_5^2x_6
-\frac{64}{9}\alpha x_1^3x_5x_6^2-\frac{160}{9}\alpha
x_1^2x_5^2x_6-\frac{80}{3}x_1^3x_4x_5^2x_6-\frac{64}{9}x_1^2x_2x_5^3x_6
\\-&\frac{8}{9}x_2x_3x_5^3-\frac{8}{3}\alpha x_3x_5^2
+4x_1x_3x_4x_5^2+\frac{160}{9}x_1^3x_3x_5^3x_6-16x_1^2x_4x_5^3-\frac{8}{3}x_1x_2x_5^4\\-& \frac{4}{3}x_1x_2x_4x_5x_6+\frac{8}{3}\beta x_1^2x_6+\frac{32}{9}\alpha x_1^2x_3x_6^2+
\frac{16}{3}x_1^3x_3x_4x_6^2+\frac{16}{9}x_1^2x_2x_3x_5x_6^2 \\-&\frac{32}{3}x_1^4x_4x_5x_6^2
-\frac{8}{3}x_1^2x_2x_4x_6^2+4\alpha x_4x_5+2x_2x_4x_5^2+4\beta x_1x_5+\frac{64}{9}x_1^4x_3x_5^2x_6^2\\
-&\frac{2}{3}x_2x_3x_4x_6-\frac{32}{3}\alpha x_1x_5^3 +
\frac{32}{3}x_1^2x_3x_5^4+\frac{32}{3}x_1^2x_3x_4x_5x_6-\frac{32}{9}x_1^3x_2x_5^2x_6^2.
\end{align*} 
 \end{lem}

Now, the commutative algebra $\mathcal{A}_{\alpha,\beta}$ is a Poisson $\mathbb{K}$-algebra with the Poisson bracket defined as follows:
\begin{align*}
\{x_2, x_1\}&=-3x_1x_2& \{x_3,x_1\}&=-x_1x_3-x_2&\{x_3,x_2\}&=-3x_2x_3\\
\{x_4, x_1\}&=-2x_3^2 & \{x_4, x_2\}&=-3x_2x_4-{4}x_3^3& \{x_4, x_3\}&=-3x_3x_4\\
\{x_5, x_1\}&=x_1x_5-2x_3& \{x_5,x_2\}&=-6x_3^2&\{x_5,x_3\}&=-x_3x_5-3x_4\\
\{x_5,x_4\}&=-3x_4x_5& \{x_6,x_1\}&=3x_1x_6-3x_5&\{x_6, x_2\}&=3x_2x_6+9x_4-18x_3x_5\\
\{x_6, x_3\}&=-6x_5^2&\{x_6, x_4\}&=-3x_4x_6-4x_5^3&\{x_6, x_5\}&=-3x_5x_6.
\end{align*} 
\begin{rem}
The Poisson algebra $\mathcal{A}_{\alpha,\beta}$ is the semiclassical limit of the quantum second Weyl algebra $A_{\alpha,\beta}$ studied in \cite{lo}. 
\end{rem}
In the remainder of this section, we study a Poisson torus arising from a localization of $\mathcal{A}_{\alpha,\beta},$ and a linear basis of $\mathcal{A}_{\alpha,\beta},$ both of which will be useful in computing the Poisson derivations of $\mathcal{A}_{\alpha,\beta}$ in the next section.
\subsubsection{A Poisson torus}
\label{sub3}
Let $\alpha, \beta\neq 0.$ Recall from Subsection \ref{sub2} that $\Omega_1=T_1T_3T_5$ and $\Omega_2=T_2T_4T_6$ in $\overline{\mathcal{A}}.$ From \cite[Corollaries 3.3 \& 3.4]{sc}, there exists a multiplicative set $S_{\alpha,\beta}$ such that
$$\mathcal{A}_{\alpha,\beta}S_{\alpha,\beta}^{-1}\cong \mathscr{P}_{\alpha,\beta}:=\frac{\mathfrak{R}_1}{\langle T_1T_3T_5-\alpha,T_2T_4T_6-\beta\rangle},$$ where  $\mathfrak{R}_1=\mathbb{K}[T_1^{\pm 1},\ldots,T_6^{\pm 1}]$ is a Poisson torus associated to the Poisson affine space $\overline{\mathcal{A}}.$ Let $t_i:=T_i+\langle T_1T_3T_5-\alpha,T_2T_4T_6-\beta\rangle$ denote the canonical image of $T_i$ in $\mathscr{P}_{\alpha,\beta}$ for each $1\leq i\leq 6.$ The algebra 
 $\mathscr{P}_{\alpha,\beta}$ is a Poisson torus  generated by $t_1^{\pm 1},\ldots, t_6^{\pm 1}$ subject to the relations:
 $$t_1=\alpha t_3^{-1}t_5^{-1}  \ \ \text{and} \ \ t_2=\beta t_4^{-1}t_5^{-1}.$$
One can verify that
$\mathscr{P}_{\alpha,\beta}\cong \mathbb{K}[t_3^{\pm 1},t_4^{\pm 1},t_5^{\pm 1},t_6^{\pm 1}],$
and that the isomorphism  holds whether $\alpha$ or $\beta$ is zero.

\subsubsection{Linear basis for $\mathcal{A}_{\alpha,\beta}$}
Set $\mathcal{A}_{\beta}:=\mathcal{A}/\langle\Omega_2-\beta\rangle,  \ \beta\in \mathbb{K}.$ 
Denote the canonical image of $X_i$ in $\mathcal{A}_{\beta}$ by 
$\widehat{x_i}:=X_i+\langle\Omega_2-\beta\rangle$ for each $1\leq i\leq 6.$ It can be verified that $\mathcal{A}_{\alpha,\beta}\cong \mathcal{A}_{\beta}/\langle\widehat{\Omega}_1-\alpha\rangle.$ Note that $\mathcal{A}_{\beta}$ satisfies the relation:
\begin{align}
\label{pe3}
\widehat{x_4}^2=\frac{2}{3}\beta-\frac{2}{3}\widehat{x_2}\widehat{x_4}\widehat{x_6}+\frac{4}{9}\widehat{x_3}^3\widehat{x_6}+\frac{4}{9}\widehat{x_2}\widehat{x_5}^3-
\frac{4}{3}\widehat{x_3}^2\widehat{x_5}^2+2\widehat{x_3}\widehat{x_4}\widehat{x_5}.
\end{align}
\begin{pro}
\label{pp1}
The set $\mathfrak{F}=\{ \widehat{x_1}^i\widehat{x_2}^j\widehat{x_3}^k\widehat{x_4}^{\xi}\widehat{x_5}^l\widehat{x_6}^m\mid (\xi, i,j,k,l,m)\in\{0,1\}\times \mathbb{N}^5\}$
is a $\mathbb{K}$-basis of $\mathcal{A}_{\beta}.$ 
\end{pro}
\begin{proof}
Since $(\Pi_{s=1}^{6}X_s^{i_s})_{i_s\in\mathbb{N}}$ is a basis of $\mathcal{A}$ over $\mathbb{K},$ we have that  $(\Pi_{s=1}^{6}\widehat{x_s}^{i_s})_{i_s\in\mathbb{N}}$
 is a spanning set of $\mathcal{A}_{\beta}$ over $\mathbb{K}.$  We want to show that $\mathfrak{F}$ is a spanning set of $\mathcal{A}_{\beta}.$ It is sufficient to do that by showing that
$\widehat{x_1}^{i_1}\widehat{x_2}^{i_2}\widehat{x_3}^{i_3}\widehat{x_4}^{i_4}
\widehat{x_5}^{i_5}\widehat{x_6}^{i_6}$ can be written as a finite linear combination of 
the elements of $\mathfrak{F}$ over $\mathbb{K}$ for all $i_1,\ldots,i_6\in \mathbb{N}.$
We do this by an induction on $i_4.$ The result is clear when $i_4=0.$ For $i_4\geq 0,$ suppose that 
$$\widehat{x_1}^{i_1}\widehat{x_2}^{i_2}\widehat{x_3}^{i_3}\widehat{x_4}^{i_4}
\widehat{x_5}^{i_5}\widehat{x_6}^{i_6}=\sum_{(\xi,\underline{v})\in I}
a_{(\xi,\underline{v})}\widehat{x_1}^i\widehat{x_2}^j\widehat{x_3}^k\widehat{x_4}^{\xi}\widehat{x_5}^l
\widehat{x_6}^m,$$
where $\underline{v}:=(i,j,k,l,m)\in \mathbb{N}^5,$ $I$ is a finite subset of $\{0,1\}\times \mathbb{N}^5,$ and the $a_{(\xi,\underline{v})}$ are scalars. 
It follows that 
$$\widehat{x_1}^{i_1}\widehat{x_2}^{i_2}\widehat{x_3}^{i_3}\widehat{x_4}^{i_4+1}
\widehat{x_5}^{i_5}\widehat{x_6}^{i_6}=\sum_{(\xi,\underline{v})\in I}
a_{(\xi,\underline{v})}\widehat{x_1}^i\widehat{x_2}^j\widehat{x_3}^k\widehat{x_4}^{\xi+1}\widehat{x_5}^l
\widehat{x_6}^m.$$
We have to show that $\widehat{x_1}^i\widehat{x_2}^j\widehat{x_3}^k\widehat{x_4}^{\xi+1}\widehat{x_5}^l
\widehat{x_6}^m\in$ Span$(\mathfrak{F})$ for all $(\xi,\underline{v})\in I.$ The result is obvious when $\xi=0.$ 
For $\xi=1,$ then using \eqref{pe3}, one can verify that
$\widehat{x_1}^i\widehat{x_2}^j\widehat{x_3}^k\widehat{x_4}^{2}\widehat{x_5}^l
\widehat{x_6}^m\in$ Span$(\mathfrak{F}).$ 
Consequently, 
$\widehat{x_1}^{i_1}\widehat{x_2}^{i_2}\widehat{x_3}^{i_3}\widehat{x_4}^{i_4+1}
\widehat{x_5}^{i_5}\widehat{x_6}^{i_6}\in$ Span$(\mathfrak{F}).$ Therefore, $\mathfrak{F}$ spans $\mathcal{A}_{\beta}.$ 
Before we continue the proof, the following ordering $<_4$ needs to be noted.
\begin{itemize}
\item[$\clubsuit$ ] 
Let $(i',j',k',l',m',n'),$  $(i,j,k,l,m,n)\in$ $ \mathbb{N}^6.$ We say that $(i,j,k,l,m,n)<_4 (i',j',k',l',m',n')$ if $[l<l']$ or $[l=l'$ and $i<i']$ or $[l=l', \ i=i'$ and $j<j']$ or $[l=l', \ i=i', \ j=j'$ and $k<k']$ or $[l=l', \ i=i', \ j=j', \ k=k'$ and $m<m']$ or $[l=l', \ i=i', \ j=j', \ k=k', \ m=m'$ and $n\leq n'].$ 

Note that the square brackets $[ \ ]$ are there to differentiate the options.
\end{itemize}
We proceed to show that $\mathfrak{F}$ is a linearly independent set.
  Suppose that
$$\sum_{(\xi,\underline{v})\in I}
a_{(\xi,\underline{v})}\widehat{x_1}^i\widehat{x_2}^j\widehat{x_3}^k\widehat{x_4}^{\xi}\widehat{x_5}^l
\widehat{x_6}^m=0.$$ It follows that 
$$\sum_{(\xi,\underline{v})\in I}
a_{(\xi,\underline{v})}X_1^iX_2^jX_3^kX_4^{\xi}X_5^l
X_6^m=\nu(\Omega_2-\beta),$$ where $\nu\in\mathcal{A}.$ Write $\nu=\sum_{(i,\ldots,n)\in J}b_{(i,\ldots,n)}
X_1^iX_2^jX_3^kX_4^{l}X_5^mX_6^{n},$ where $J$ is a finite subset of $\mathbb{N}^6,$ and 
$b_{(i,\ldots,n)}$ are scalars. From Subsection \ref{sub2}, we have that $$\Omega_2=X_2X_4X_6-\frac{2}{3}X_3^3X_6-\frac{2}{3}X_2X_5^3+2X_3^2X_5^2-3
X_3X_4X_5+\frac{3}{2}X_4^2.$$ It follows that
\begin{align*}
\sum_{(\xi,\underline{v})\in I}
a_{(\xi,\underline{v})}
X_1^iX_2^jX_3^kX_4^{\xi}X_5^l
X_6^m&= \sum_{(i,\ldots,n)\in J}b_{(i,\ldots,n)}
X_1^iX_2^jX_3^kX_4^{l}X_5^mX_6^{n} (\Omega_2-\beta)\\
&=\sum_{(i,\ldots, n)\in J}\frac{3}{2} b_{(i,\ldots,n)}X_1^iX_2^jX_3^kX_4^{l+2}X_5^mX_6^{n}+\text{LT}_{<_4},
\end{align*}
where $\text{LT}_{<_4}$ contains lower order terms with respect to 
$<_4$ (see item $\clubsuit$).  Moreover, $\text{LT}_{<_4}$ vanishes when $b_{(i,\ldots,n)}=0$ for all $(i,\ldots,n)\in J.$ One can easily confirm this when the previous line of equality (right hand side) is fully expanded.

Suppose that there exists $(i,j,k,l,m,n)\in J$ such that $b_{(i,j,k,l,m,n)}\neq 0.$  
Let  $(i',j',k',l',m',n')$ be the greatest element of $J$ with respect to 
$<_4$  such that  $b_{(i',j',k',l',m',n')}\neq 0.$  Identifying the coefficients
of $X_1^{i'}X_2^{j'}X_3^{k'}X_4^{l'+2}X_5^{m'}X_6^{n'},$ we have $\frac{3}{2} b_{(i',j',k',l',m',n')}=0$ (note that the family  $(X_1^iX_2^jX_3^kX_4^lX_5^mX_6^n)_{i,\ldots,n\in\mathbb{N}}$ is a basis for $\mathcal{A}$ and $\text{LT}_{<_4}$ contains lower order terms). Therefore, $b_{(i',j',k',l',m',n')}=0,$ a contradiction! As a result, $b_{(i,j,k,l,m,n)}=0$ for all $(i,j,k,l,m,n)\in J,$ and
$$\sum_{(\xi,\underline{v})\in I}
a_{(\xi,\underline{v})}X_1^iX_2^jX_3^kX_4^{\xi}X_5^l
X_6^m=0.$$ Consequently, $a_{(\xi, i,j,k,l)}=0$ for all $(\xi, i,j,k,l)\in I.$ 
\end{proof} 

We are now ready to find a basis for $\mathcal{A}_{\alpha,\beta}.$ 
\begin{pro}
\label{pc3p5}
The set $\mathfrak{P}=\{ x_1^ix_2^jx_3^{\epsilon_1}x_4^{\epsilon_2}x_5^kx_6^l\mid (\epsilon_1,\epsilon_2, i,j,k,l)\in \{0,1\}^2\times \mathbb{N}^4\}$ is a $\mathbb{K}$-basis of $\mathcal{A}_{\alpha,\beta}.$
\end{pro}
\begin{proof}
Since the set  $\mathfrak{F}=\{ \widehat{x_1}^{i_1}\widehat{x_2}^{i_2}\widehat{x_3}^{i_3}\widehat{x_4}^{\xi}
\widehat{x_5}^{i_5}\widehat{x_6}^{i_6}\mid (i_1,i_2,i_3,\xi, i_5,i_6)\in \{0,1\}\times \mathbb{N}^5\}
$ is a $\mathbb{K}$-basis of $\mathcal{A}_\beta$ (Prop. \ref{pp1}) and $\mathcal{A}_{\alpha,\beta}$ is identified with ${\mathcal{A}_\beta}/{\langle \widehat{\Omega}_1-\alpha\rangle},$ it follows that\newline  $(x_1^{i_1}x_2^{i_2}x_3^{i_3}x_4^\xi x_5^{i_5}x_6^{i_6})_{(i_1,i_2, i_3,\xi,i_5,i_6)\in \{0,1\}\times \mathbb{N}^5}$ is a spanning set of $\mathcal{A}_{\alpha,\beta}$ over $\mathbb{K}.$  We want to show that $\mathfrak{P}$ spans $\mathcal{A}_{\alpha,\beta}$ by showing that $e_1^{i_1}e_2^{i_2}e_3^{i_3}e_4^\xi e_5^{i_5}e_6^{i_6}$ can be written as a finite linear combination of the elements of $\mathfrak{P}$ over $\mathbb{K}$ for all $(i_1,i_2,i_3,\xi, i_5,i_6)\in \{0,1\}\times \mathbb{N}^5$. By Proposition \ref{pp1}, it is sufficient to do this by an induction on $i_3.$ The result is obvious when  $i_3=0$ or $1.$ For $i_3\geq 1,$ suppose that
\begin{align*}
x_1^{i_1}x_2^{i_2}x_3^{i_3}x_4^\xi x_5^{i_5}x_6^{i_6}
&=\sum_{(\epsilon_1,\epsilon_2,\underline{v})\in I}a_{(\epsilon_1,\epsilon_2,\underline{v})}
x_1^ix_2^jx_3^{\epsilon_1}x_4^{\epsilon_2}x_5^kx_6^l,
\end{align*}
where $\underline{v}:=(i,j,k,l)\in \mathbb{N}^4,$ and the $a_{(\epsilon_1, \epsilon_2, \underline{v})}$ are all scalars. Moreover, $I$ is a  finite subset of $\{0,1\}^2\times \mathbb{N}^4$.
It follows from the inductive hypothesis that
\begin{align*}
x_1^{i_1}x_2^{i_2}x_3^{i_3+1}x_4^\xi x_5^{i_5}x_6^{i_6}
=\sum_{(\epsilon_1,\epsilon_2,\underline{v})\in I}a_{(\epsilon_1,\epsilon_2,\underline{v})}
x_1^ix_2^jx_3^{\epsilon_1+1}x_4^{\epsilon_2}x_5^kx_6^l.
\end{align*}
We need to show that $x_1^ix_2^jx_3^{\epsilon_1+1}x_4^{\epsilon_2}x_5^kx_6^l\in$ Span$(\mathfrak{P})$ for all $(\epsilon_1,\epsilon_2,\underline{v})\in I.$ The result is obvious when $(\epsilon_1,\epsilon_2)
=(0,0), (0,1).$ 
Using Lemma \ref{pl2}(1),(3), one can also show that $x_1^ix_2^jx_3^{\epsilon_1+1}x_4^{\epsilon_2}x_5^kx_6^l\in$ Span$(\mathfrak{P})$ for all
$(\epsilon_1,\epsilon_2)=(1,0), (1,1);$ and $i,j,k,l\in \mathbb{N}.$ Therefore, $x_1^{i_1}x_2^{i_2}x_3^{i_3+1}x_4^\xi x_5^{i_5}x_6^{i_6}\in$ Span$(\mathfrak{P}).$ As a result, $\mathfrak{P}$ spans $\mathcal{A}_{\alpha,\beta}.$ 

We proceed to show that $\mathfrak{F}$ is a linearly independent set. Suppose that
\begin{align*}
\sum_{(\epsilon_1,\epsilon_2,\underline{v})\in I}a_{(\epsilon_1,\epsilon_2,\underline{v})}
x_1^ix_2^jx_3^{\epsilon_1}x_4^{\epsilon_2}x_5^kx_6^l=0.
\end{align*}
Then, 
\begin{align*}
\sum_{(\epsilon_1,\epsilon_2,\underline{v})\in I}a_{(\epsilon_1,\epsilon_2,\underline{v})}
\widehat{x_1}^i\widehat{x_2}^j\widehat{x_3}^{\epsilon_1}\widehat{x_4}^{\epsilon_2}
\widehat{e_5}^k\widehat{e_6}^l
&= (\widehat{\Omega}_1-\alpha)\nu
\end{align*}
in $\mathcal{A}_\beta,$ where $\nu\in \mathcal{A}_{\beta}.$ Set $\underline{w}:=(i,j,k,l,m)\in \mathbb{N}^5$, and let $J_1, J_2$ be finite subsets of $\mathbb{N}^5.$ One can write $\nu$ in terms of the basis 
$\mathfrak{F}$ of $\mathcal{A}_{\beta}$ as: 
$$\nu=\displaystyle\sum_{\underline{w}\in J_1}b_{\underline{w}}\widehat{x_1}^i\widehat{x_2}^j\widehat{x_3}^k
\widehat{x_4}\widehat{x_5}^l
 \widehat{x_6}^m+\displaystyle\sum_{\underline{w}\in J_2}c_{\underline{w}}\widehat{x_1}^i\widehat{x_2}^j\widehat{x_3}^k
\widehat{x_5}^l
 \widehat{x_6}^m,$$
 where   $b_{\underline{w}}$ and $  c_{\underline{w}}$ are scalars.  Note that
 $\widehat{\Omega}_1=\widehat{x_1}\widehat{x_3}\widehat{x_5}-\frac{3}{2}\widehat{x_1}\widehat{x_4}-\frac{1}{2}\widehat{x_2}\widehat{x_5}+\frac{1}{2}\widehat{x_3}^2.$ Given this expression, and the relation \eqref{pe3}, one can express the above equality as follows:
\begin{align*}
\sum_{(\epsilon_1,\epsilon_2,\underline{v})\in I}a_{(\epsilon_1,\epsilon_2,\underline{v})}
\widehat{x_1}^i\widehat{x_2}^j\widehat{x_3}^{\epsilon_1}\widehat{x_4}^{\epsilon_2}
\widehat{e_5}^k\widehat{e_6}^l=&\displaystyle\sum_{\underline{w}\in J_1}b_{\underline{w}}\widehat{x_1}^i\widehat{x_2}^j\widehat{x_3}^k
\widehat{x_4}\widehat{x_5}^l
 \widehat{x_6}^m(\widehat{\Omega}_1-\alpha)\\&+\displaystyle\sum_{\underline{w}\in J_2}c_{\underline{w}}\widehat{x_1}^i\widehat{x_2}^j\widehat{x_3}^k
\widehat{x_5}^l
 \widehat{x_6}^m(\widehat{\Omega}_1-\alpha)\\
 =& \sum_{\underline{w}\in J_1}\frac{1}{2}b_{\underline{w}}\widehat{x_1}^i\widehat{x_2}^j\widehat{x_3}^{k+2}
\widehat{x_4}\widehat{x_5}^l
 \widehat{x_6}^m \\
&-\sum_{\underline{w}\in J_1}\frac{2}{3} b_{\underline{w}}\widehat{x_1}^{i+1}\widehat{x_2}^j
 \widehat{x_3}^{k+3}
\widehat{x_5}^{l}
 \widehat{x_6}^{m+1}\\
 &-\sum_{\underline{w}\in J_2}\frac{3}{2}c_{\underline{w}}\widehat{x_1}^{i+1}\widehat{x_2}^j\widehat{x_3}^k
 \widehat{x_4}\widehat{x_5}^l
 \widehat{x_6}^m\\
 &+\sum_{\underline{w}\in J_2}\frac{1}{2}c_{\underline{w}}\widehat{x_1}^i\widehat{x_2}^j\widehat{x_3}^{k+2}
\widehat{x_5}^l
 \widehat{x_6}^m+\Upsilon,
 \end{align*}
 where 
 \begin{align*}
 \Upsilon =&\sum_{\underline{w}\in J_1}r_1 b_{\underline{w}} \widehat{x_1}^{i+1}\widehat{x_2}^{j}\widehat{x_3}^{k}
 \widehat{x_5}^{l}\widehat{x_6}^m+
 \sum_{\underline{w}\in J_1}r_2 b_{\underline{w}} \widehat{x_1}^{i+1}\widehat{x_2}^{1+j}\widehat{x_3}^{k}\widehat{x_4}
 \widehat{x_5}^{l}\widehat{x_6}^{m+1}\\
&+ \sum_{\underline{w}\in J_1}r_3 b_{\underline{w}}\widehat{x_1}^{i+1}\widehat{x_2}^{j+1}\widehat{x_3}^{k}
 \widehat{x_5}^{l+3}\widehat{x_6}^m+
 \sum_{\underline{w}\in J_1}r_4 b_{\underline{w}}\widehat{x_1}^{i+1}\widehat{x_2}^{j}\widehat{x_3}^{k+2}
 \widehat{x_5}^{l+2}\widehat{x_6}^m\\
 &+\sum_{\underline{w}\in J_1}r_5 b_{\underline{w}}\widehat{x_1}^{i+1}\widehat{x_2}^{j}\widehat{x_3}^{k+1}
 \widehat{x_4}\widehat{x_5}^{l+1}\widehat{x_6}^m
 +\sum_{\underline{w}\in J_1}r_6 b_{\underline{w}} \widehat{x_1}^{i}\widehat{x_2}^{j+1}\widehat{x_3}^{k}
 \widehat{x_4}\widehat{x_5}^{l+1}\widehat{x_6}^m\\
 &+\sum_{\underline{w}\in J_1} r_7b_{\underline{w}}\beta \widehat{x_1}^{i}\widehat{x_2}^{j}\widehat{x_3}^{k}\widehat{x_4}
\widehat{x_5}^{l}\widehat{x_6}^m+ \sum_{\underline{w}\in J_1} r_8 b_{\underline{w}} \widehat{x_1}^{i+1}\widehat{x_2}^j\widehat{x_3}^{k+1}\widehat{x_4}
 \widehat{x_5}^{l+1}\widehat{x_6}^m \\
&+\sum_{\underline{w}\in J_2}r_9 c_{\underline{w}}\widehat{x_1}^{i+1}\widehat{x_2}^{j}\widehat{x_3}^{k+1}
\widehat{x_5}^{l+1}\widehat{x_6}^m
+\sum_{\underline{w}\in J_2}r_{10}c_{\underline{w}}\widehat{x_1}^{i}\widehat{x_2}^{j+1}\widehat{x_3}^{k}
\widehat{x_5}^{l+1}\widehat{x_6}^m\\&+\sum_{\underline{w}\in J_2}r_{11}c_{\underline{w}}\alpha\widehat{x_1}^{i}\widehat{x_2}^{j}\widehat{x_3}^{k}
\widehat{x_5}^{l}\widehat{x_6}^m.
  \end{align*} 
 Note that $r_1, \ldots, r_{11}$ are some non-zero rational numbers.
Before we continue the proof, the following ordering $<_3$ needs to be noted.
\begin{itemize}
\item[$\spadesuit$] Let $(\vartheta_1,\vartheta_2,\vartheta_3,\vartheta_5,\vartheta_6)$, $(\varsigma_1,\varsigma_2,\varsigma_3,\varsigma_5,\varsigma_6)\in\mathbb{N}^5.$ We say that $(\varsigma_1,\varsigma_2,\varsigma_3,\varsigma_5,\varsigma_6)<_3 (\vartheta_1,\vartheta_2,\vartheta_3,\vartheta_5,\vartheta_6)$ if $[\vartheta_3>\varsigma_3]$ or 
$[\vartheta_3=\varsigma_3$ and $\vartheta_1>\varsigma_1]$ or $[\vartheta_3=\varsigma_3, \ \vartheta_1=\varsigma_1$ and $\vartheta_2>\varsigma_2]$ 
 or $[\vartheta_3=\varsigma_3,\  \vartheta_1=\varsigma_1, \ \vartheta_2=\varsigma_2$ and $\vartheta_5>\varsigma_5]$ or $[\vartheta_3=\varsigma_3,\ \vartheta_1=\varsigma_1, \ \vartheta_2=\varsigma_2,\  \vartheta_5=\varsigma_5$ and $\vartheta_6\geq \varsigma_6].$
\end{itemize} 
Observe that $\Upsilon$ contains lower order terms with respect to  
$<_3$ in each monomial type (note that there are two different types of monomials in the basis of $\mathcal{A}_{\beta};$ one with $\widehat{x_4}$ and the other without $\widehat{x_4}$ ). Now, suppose that there exists $(i,j,k,l,m)\in J_1$ and $(i,j,k,l,m)\in J_2$ such that $b_{(i,j,k,l,m)}\neq 0$ and $c_{(i,j,k,l,m)}\neq 0.$ 
Let $(v_1,v_2,v_3,v_5,v_6)$ and $(w_1,w_2,w_3,w_5,w_6)$ be the greatest elements  of  $J_1$ and $J_2$  respectively with respect to $<_3$  such that $b_{(v_1,v_2,v_3, v_5,v_6)}$ and $ c_{(w_1,w_2,w_3,w_5,w_6)}$  are non-zero.
 Since  $\mathfrak{F}$ is a linear basis for $\mathcal{A}_\beta$ and $\Upsilon$ contains lower order terms with respect to $<_3$, we have the following:
 if $w_3-v_3<2,$ then identifying the coefficients of $\widehat{x_1}^{v_1}\widehat{x_2}^{v_2}\widehat{x_3}^{v_3+2}\widehat{x_4}
 \widehat{x_5}^{v_5}\widehat{x_6}^{v_6}$ implies that $\frac{1}{2}b_{(v_1,v_2,v_3, v_5,v_6)}=0,$   a contradiction!
 Finally, if $w_3-v_3\geq 2,$ then identifying the coefficients  of $\widehat{x_1}^{w_1}\widehat{x_2}^{w_2}\widehat{x_3}^{w_3+2}\widehat{x_5}^{w_5}\widehat{x
_6}^{w_6}$ implies that $\frac{1}{2}c_{(w_1,w_2,w_3,w_5,w_6)}=0$, another contradiction!  
Therefore, either all $b_{(i,j,k,m,n)}$ or all $ c_{(i,j,k,m,n)}$ are zero. 
  Without loss of generality, suppose that there exists $(i,j,k,m,n)\in J_2$ such that  $c_{(i,j,k,m,n)}$ is not zero. Then, $ b_{(i,j,k,m,n)}$ are all zero. Let $(w_1,w_2,w_3,w_5,w_6)$ be the greatest element of $J_2$ such that $ c_{(w_1,w_2,w_3,w_5,w_6)}\neq 0.$  Identifying the coefficients of 
  $\widehat{x_1}^{w_1}\widehat{x_2}^{w_2}
 \widehat{x_3}^{w_3+2} \widehat{x_5}^{w_5}\widehat{x_6}^{w_6}$ in the above equality implies that $\frac{1}{2} c_{(w_1,w_2,w_3,w_5,w_6)}=0,$ a contradiction! We can therefore conclude that  $b_{(i,j,k,m,n)}$ and $ c_{(i,j,k,m,n)}$ are all zero. 
 Consequently,
\begin{align*}
\sum_{(\epsilon_1,\epsilon_2,\underline{v})\in I}a_{(\epsilon_1,\epsilon_2,\underline{v})}
\widehat{x_1}^i\widehat{x_2}^j\widehat{x_3}^{\epsilon_1}\widehat{x_4}^{\epsilon_2}
\widehat{x_5}^k\widehat{x_6}^l
&= 0.
\end{align*}
Since $\mathfrak{F}$ is a basis for $\mathcal{A}_{\beta}$, this implies that
$a_{(\epsilon_1,\epsilon_2,\underline{v})}=0$ for all $(\epsilon_1,\epsilon_2,\underline{v})\in I.$ Therefore, $\mathfrak{P}$ is a linearly independent set.
\end{proof}

\begin{cor}
\label{pc4c2}
Let $\underline{v}=(i,j,k,l)\in \mathbb{N}^2\times \mathbb{Z}^2, \ I$ represent a finite subset of $\{0,1\}^2\times \mathbb{N}^2\times \mathbb{Z}^2,$   and
$(a_{(\epsilon_1,\epsilon_2,\underline{v})})_{(\epsilon_1,\epsilon_2, \underline{v})\in I}$ be a family of scalars.  If 
$$\sum_{(\epsilon_1,\epsilon_2, \underline{v})\in I}a_{(\epsilon_1,\epsilon_2, \underline{v})}x_{1}^{i}x_{2}^{j}x_3^{\epsilon_1}
x_4^{\epsilon_2}t_5^{k}t_6^{l}=0,$$ then $a_{(\epsilon_1,\epsilon_2, \underline{v})}=0$ for all $(\epsilon_1,\epsilon_2, \underline{v})\in I.$ 
\end{cor}
\begin{proof}
The result is obvious when $k, l\geq 0$ due to Proposition \ref{pc3p5}.
When $k$ (resp. $l$) is negative, then one can multiply the above equality enough times by $t_5$ (resp. $t_6$) to kill all the negative powers and then apply Proposition \ref{pc3p5} to complete the proof.
\end{proof}

%%%%%%%%%%%%%%%%%%%%%%%%%%%%%%%%%%%%%%%%%%%%%%%%%%%%%%%%%%%%
%%%%%%%%%%%%%%%%%%%%%%%%%%%%%%%%%%%%%%%%%%%%%%%%%%%%%%%%%%%%
%%%%%%%%%%%%%%%%%%%%%%%%%%%%%%%%%%%%%%%%%%%%%%%%%%%%%%%%%%%% 

\section{ Poisson derivations of the semiclassical limit of the quantum second Weyl algebra $A_{\alpha,\beta}$}
\label{chap6}
This section focuses on studying the Poisson derivations of the Poisson algebra  $\mathcal{A}_{\alpha,\beta}$.
\subsection{Preliminaries and strategy}
\label{pr2}
 Let $2\leq j\leq 7$ and  $(\alpha,\beta)\in \mathbb{K}^2\setminus \{(0,0)\}.$ Set 
$$\mathcal{A}_{\alpha,\beta}^{(j)}:=\frac{\mathcal{A}^{(j)}}{\langle \Omega_1-\alpha,\Omega_2-\beta\rangle},$$
where $\mathcal{A}^{(j)}$ is defined in Subsection \ref{pdda}, and $\Omega_1$ and $\Omega_2$ are the generators of the centre of $\mathcal{A}^{(j)}$ (see Subsection \ref{pdda}). Recall that $\mathcal{A}^{(7)}=\mathcal{A}=\mathbb{K}[X_1,\ldots,X_6].$ It follows that 
$\mathcal{A}_{\alpha,\beta}^{(7)}=\mathcal{A}_{\alpha,\beta}.$
For each $2\leq j\leq 7,$ denote the canonical images of the generators $X_{i,j}$ of $\mathcal{A}^{(j)}$ in $\mathcal{A}_{\alpha,\beta}^{(j)}$ by $x_{i,j}$ for all 
$1\leq i\leq 6.$  Furthermore, from \cite{sc}, one can deduce the following data 
 of $\mathcal{A}_{\alpha,\beta}$ from the PDDA  of  $\mathcal{A}$ (see Subsection \ref{pdda}) as follows:
\begin{align*}
x_{1,6}&=x_1-\frac{1}{2}x_5x_6^{-1}\\
x_{2,6}&=x_2+\frac{3}{2}x_4x_6^{-1}-3x_3x_5x_6^{-1}+x_5^3x_6^{-2}\\
x_{3,6}&=x_3-x_5^2x_6^{-1}\\
x_{4,6}&=x_4-\frac{2}{3}x_5^3x_6^{-1}\\
x_{1,5}&=x_{1,6}-x_{3,6}x_{5,6}^{-1}+\frac{3}{4}x_{4,6}x_{5,6}^{-2}\\
x_{2,5}&=x_{2,6}-3x_{3,6}^2x_{5,6}^{-1}+\frac{9}{2}x_{3,6}
x_{4,6}x_{5,6}^{-2}-\frac{9}{4}x_{4,6}^2x_{5,6}^{-3}\\
x_{3,5}&=x_{3,6}-\frac{3}{2}x_{4,6}x_{5,6}^{-1}\\
x_{1,4}&=x_{1,5}-\frac{1}{3}x_{3,5}^2x_{4,5}^{-1}\\
x_{2,4}&=x_{2,5}-\frac{2}{3}x_{3,5}^3x_{4,5}^{-1}\\
x_{1,3}&=x_{1,4}-\frac{1}{2}x_{2,4}x_{3,4}^{-1}\\
t_1&:=x_{1,2}=x_{1,3}\\
t_2&:=x_{2,2}=x_{2,3}=x_{2,4}\\
t_3&:=x_{3,2}=x_{3,3}=x_{3,4}=x_{3,5}\\
t_4&:=x_{4,2}=x_{4,3}=x_{4,4}=x_{4,5}=x_{4,6}\\
t_5&:=x_{5,2}=x_{5,3}=x_{5,4}=x_{5,5}=x_{5,6}=x_5\\
t_6&:=x_{6,2}=x_{6,3}=x_{6,4}=x_{6,5}=x_{6,6}=x_6.
\end{align*}
For simplicity, we will refer to the above data as the PDDA of $\mathcal{A}_{\alpha,\beta}.$

Note that the $t_i$ are  the canonical images of $T_i$ in $\mathcal{A}^{(2)}_{\alpha,\beta}$ for all $1\leq i\leq 6.$ 
For each $2\leq j<7,$  set $S_j:=\{\lambda t_j^{i_j}t_{j+1}^{i_{j+1}}\ldots t_6^{i_6}\mid i_j,\ldots,i_6\in\mathbb{N}, \ \lambda\in \mathbb{K}^*\}.$ One can observe that $S_j$ is a multiplicative system of non-zero divisors (or regular elements) of $\mathcal{A}_{\alpha,\beta}^{(j)}.$  As a result, one can localize $\mathcal{A}_{\alpha,\beta}^{(j)}$ at $S_j.$ Let us denote this localization by $\mathcal{R}_j.$ That is,
$$\mathcal{R}_j:=\mathcal{A}_{\alpha,\beta}^{(j)}S_j^{-1}.$$ Again, the set $\Sigma_j:=\{t_j^k\mid k\in \mathbb{N}\}$  is a multiplicative set in both $\mathcal{A}_{\alpha,\beta}^{(j)}$ and $\mathcal{A}_{\alpha,\beta}^{(j+1)}$ for each  $2\leq j\leq 6.$  Therefore, $$\mathcal{A}_{\alpha,\beta}^{(j)}\Sigma_j^{-1}=\mathcal{A}_{\alpha,\beta}^{(j+1)}\Sigma_j^{-1}.$$
It follows that
\begin{align}
\label{pe0}
\mathcal{R}_j=\mathcal{R}_{j+1}\Sigma_j^{-1},
\end{align}
for all $2\leq j\leq 6,$ with the convention that $\mathcal{R}_7:=\mathcal{A}_{\alpha,\beta}.$ 
Similarly to \eqref{embp}, we construct the following embeddings:
\begin{align}
\label{pe}
\mathcal{R}_7=\mathcal{A}_{\alpha,\beta}\subset \mathcal{R}_6=\mathcal{R}_7\Sigma_6^{-1}\subset \mathcal{R}_5=\mathcal{R}_6\Sigma_5^{-1}\subset \mathcal{R}_4=\mathcal{R}_5\Sigma_4^{-1}
 \subset \mathcal{R}_3.
\end{align} 
Observe that $\mathcal{R}_3=\mathcal{A}_{\alpha,\beta}^{(3)}S_3^{-1}=\mathcal{R}_4\Sigma_3^{-1}$  is the Poisson torus   $\mathscr{P}_{\alpha,\beta}=\mathbb{K}[t_3^{\pm 1},t_4^{\pm 1},t_5^{\pm 1},t_6^{\pm 1}]$ in Subsection \ref{sec6.4}. \\

\textbf{Strategy to compute the Poisson derivations of $\mathcal{A}_{\alpha,\beta}=\mathcal{R}_7$}. From Corollary \ref{pic2}, we know that the Poisson derivations of the Poisson torus $\mathcal{R}_3.$ Therefore, we will extend the Poisson derivations of $\mathcal{A}_{\alpha,\beta}$ to Poisson derivations of $\mathcal{R}_3.$  We will then `pull back' the Poisson derivations of $\mathcal{R}_3$ sequentially to the Poisson derivations of $\mathcal{A}_{\alpha,\beta},$ and this will give us a complete description of the Poisson derivations of $\mathcal{A}_{\alpha,\beta}.$  This process will be carried out in steps, and the linear basis for $\mathcal{R}_i$ will play crucial role. As a result, we will compute these bases in the subsequent subsection. 

The embeddings in \eqref{pe} present an opportunity to compute the centre of each of the algebras $\mathcal{R}_i,$ which will be helpful in studying the Poisson derivations of  $\mathcal{A}_{\alpha,\beta}.$
\begin{lem}
As usual, let $Z_P(\mathcal{R}_i)$ denote the Poisson centre of $\mathcal{R}_i.$ Then  $Z_P(\mathcal{R}_i)=\mathbb{K}$ for each $3\leq i\leq 7.$
\end{lem}
\begin{proof}
It is easy to verify that $Z_P(\mathcal{R}_3)=\mathbb{K}.$ Note that $\mathcal{R}_7=\mathcal{A}_{\alpha,\beta}.$ Since $\mathcal{R}_i$ is a localization of $\mathcal{R}_{i+1},$ we have that 
$\mathbb{K}\subseteq Z_P(\mathcal{R}_7)\subseteq Z_P(\mathcal{R}_6)\subseteq \cdots \subseteq Z_P(\mathcal{R}_3)=\mathbb{K}.$ Therefore,
$Z_P(\mathcal{R}_7)= Z_P(\mathcal{R}_6)= \cdots = Z_P(\mathcal{R}_3)=\mathbb{K}.$
\end{proof}

\subsection{Linear bases for $\mathcal{R}_3,  \mathcal{R}_4$ and $\mathcal{R}_5.$}
We aim to find a basis for $\mathcal{R}_j$ for each $j=3,4,5.$  
Since $\mathcal{R}_3$ is a Poisson torus generated by $t_3^{\pm 1},\ldots,t_6^{\pm 1}$ over $\mathbb{K},$ the set $\{t_3^{i}t_4^{j}t_5^{k}t_6^{l}\mid i,j,k,l \in \mathbb{Z}\}$ is a basis of $\mathcal{R}_3.$ 

For simplicity,  we set
\begin{align*}
f_1:&=x_{1,4} & F_1:&=X_{1,4}\\
z_1:&=x_{1,5} & \zbar_1:&=X_{1,5} \\
z_2:&=x_{2,5} & \zbar_2:&=X_{2,5}.
\end{align*}

\subsubsection{Basis for $\mathcal{R}_4$} Observe that $$\mathcal{A}_{\alpha,\beta}^{(4)}=\frac{\mathcal{A}^{(4)}}{\langle \Omega_{1}-\alpha, \Omega_{2}-\beta\rangle},$$
where $\Omega_{1}=F_1T_3T_5-\frac{1}{2}T_2T_5$ and $\Omega_{2}=T_2T_4T_6$ in $\mathcal{A}^{(4)}$ (Subsection \ref{pdda}).
Set
$$\mathcal{A}_\beta^{(4)}S_4^{-1}:=\frac{\mathcal{A}^{(4)}S_4^{-1}}{\langle \Omega_{2}-\beta \rangle},$$ 
where $\beta\in \mathbb{K}$.  We will denote  the canonical images of $X_{i,4}$ (resp. $T_i$) in $\mathcal{A}_{\beta}^{(4)}$ by $\widehat{x_{i,4}}$ (resp.  $\widehat{t_i}$)  for all $1\leq i\leq 6.$  Observe that $\widehat{t_2}=\beta \widehat{t_6}^{-1}\widehat{t_4}^{-1}$ in  $\mathcal{A}_\beta^{(4)}S_4^{-1}.$  As usual, one can identify 
$\mathcal{R}_4$ with $\mathcal{A}_\beta^{(4)}S_4^{-1}/\langle\Omega_1-\alpha\rangle.$
\begin{pro}
\label{ev28}
The set $\mathfrak{P}_4=\{f_{1}^{i_1}t_4^{i_4}t_5^{i_5}
t_6^{i_6}, t_3^{i_3}t_4^{i_4}t_5^{i_5}
t_6^{i_6}\mid (i_1,i_3,i_4,i_5,i_6)\in\mathbb{N}^2\times\mathbb{Z}^3 \}$ is a $\mathbb{K}$-basis of $\mathcal{R}_4.$
\end{pro}
\begin{proof}
One can easily verify that
$\left(\widehat{f_{1}}^{k_1}\widehat{t_3}^{k_3}
\widehat{t_4}^{k_4}\widehat{t_5}^{k_5}\widehat{t_6}^{k_6}\right)_{(k_1,k_3,\ldots,k_6)\in \mathbb{N}^2\times \mathbb{Z}^3}$ is a basis of $\mathcal{A}_{\beta}^{(4)}S_4^{-1}$. Since $\mathcal{A}_\beta^{(4)}S_4^{-1}={\mathcal{A}^{(4)}S_4^{-1}}/{\langle \Omega_{2}-\beta \rangle},$ the family
$(f_{1}^{k_1}t_3^{k_3}
t_4^{k_4}t_5^{k_5}t_6^{k_6})_{(k_1,k_3,\ldots,k_6)\in \mathbb{N}^2\times \mathbb{Z}^3}$ spans $\mathcal{R}_4.$ We show that $\mathfrak{P}_4$ is a spanning set of $\mathcal{R}_4$ by showing that $f_{1}^{k_1}t_3^{k_3}
t_4^{k_4}t_5^{k_5}t_6^{k_6}$ can be written as a finite linear combination of the elements of $\mathfrak{P}_4$ for all $(k_1,k_3,\ldots,k_6)\in \mathbb{N}^2\times \mathbb{Z}^3.$ 
  It is sufficient to do this by an induction on  $k_1.$ The result is clear when $k_1=0.$ Assume that the statement is true for $k_1\geq 0.$ That is,
$$f_{1}^{k_1}t_3^{k_3}t_4^{k_4}t_5^{k_5}t_6^{k_6}
=\sum_{\underline{i}\in I_1}{a}_{\underline{i}}f_{1}^{i_1}t_4^{i_4}t_5^{i_5}t_6^{i_6}+\sum_{\underline{j} \in I_2}b_{\underline{j}}
t_3^{i_3}t_4^{i_4}t_5^{i_5}t_6^{i_6},$$ 
where $\underline{i}=(i_{1},i_{4},i_{5},i_{6})\in I_1\subset  \mathbb{N}\times\mathbb{Z}^3$ and  $\underline{j}=(i_{3},i_{4},i_{5},i_{6})\in I_2\subset  \mathbb{N}\times\mathbb{Z}^3.$
Note that the $a_{\underline{i}}$ and $b_{\underline{j}}$ are all scalars.
\begin{align*}
f_{1}^{k_1+1}t_3^{k_3}t_4^{k_4}t_5^{k_5}t_6^{k_6}&=
f_{1}(f_{1}^{k_1}t_3^{k_3}t_4^{k_4}t_5^{k_5}t_6^{k_6})
=\sum_{\underline{i}\in I_1}{a}_{\underline{i}}f_{1}^{i_1+1}t_4^{i_4}t_5^{i_5}t_6^{i_6}+\sum_{\underline{j}\in I_2}b_{\underline{j}}
f_{1}t_3^{i_3}t_4^{i_4}t_5^{i_5}t_6^{i_6}.
\end{align*}
Clearly, the monomial $f_{1}^{i_1+1}t_4^{i_4}t_5^{i_5}t_6^{i_6}\in$ Span($\mathfrak{P}_4$). We have to also show 
that $f_{1}t_3^{i_3}t_4^{i_4}t_5^{i_5}t_6^{i_6}\in$ Span($\mathfrak{P}_4$) for all $i_3\in \mathbb{N}$ and 
$i_4,i_5,i_6\in\mathbb{Z}.$ This can easily be achieved by induction on $i_3,$ and the  use of the relation $f_1t_3=\alpha t_5^{-1}+\frac{1}{2}\beta t_4^{-1}t_6^{-1}.$ 
Therefore, by the principle of mathematical induction, $\mathfrak{P}_4$ is a spanning set of $\mathcal{R}_4$ over $\mathbb{K}.$

We prove that $\mathfrak{P}_4$ is a linearly independent set. Suppose that
$$\sum_{\underline{i}\in I_1}a_{\underline{i}}f_{1}^{i_1}t_4^{i_4}t_5^{i_5}t_6^{i_6}
+\sum_{\underline{j}\in I_2}b_{\underline{j}}t_3^{i_3}t_4^{i_4}t_5^{i_5}t_6^{i_6}=0.$$ 
 It follows that there exists $\nu\in \mathcal{A}_\beta^{(4)}S_4^{-1}$ such that
$$\displaystyle\sum_{\underline{i}\in I_1}a_{\underline{i}}\widehat{f_{1}}^{i_1}\widehat{t_4}^{i_4}
\widehat{t_5}^{i_5}\widehat{t_6}^{i_6}
+\displaystyle\sum_{\underline{j}\in I_2}b_{\underline{j}}\widehat{t_3}^{i_3}\widehat{t_4}^{i_4}
\widehat{t_5}^{i_5}\widehat{t_6}^{i_6}
=\left(\widehat{\Omega}_{1}-\alpha\right)\nu.$$
Write $\nu=\displaystyle\sum_{\underline{l}\in J}c_{\underline{l}}\widehat{f_{1}}^{i_1}\widehat{t_3}^{i_3}\widehat{t_4}^{i_4}
\widehat{t_5}^{i_5}\widehat{t_6}^{i_6},$
with $\underline{l}={(i_1,i_3,i_4,i_5,i_6)}\in J\subset\mathbb{N}^2\times\mathbb{Z}^3$ and $c_{\underline{l}}\in \mathbb{K}.$ One can easily deduce  that $\widehat{\Omega}_1=
 \widehat{f_{1}}\widehat{t_3}\widehat{t_5}-\frac{1}{2}\widehat{t_2}\widehat{t_5}=\widehat{f_{1}}\widehat{t_3}\widehat{t_5}-\frac{1}{2}\beta\widehat{t_6}^{-1}\widehat{t_4}^{-1}\widehat{t_5}$ (note that
 $\widehat{t_2}=\beta\widehat{t_6}^{-1}\widehat{t_4}^{-1}$).  It follows that
 \begin{align*}
\displaystyle\sum_{\underline{i}\in I_1}a_{\underline{i}}\widehat{f_{1}}^{i_1}\widehat{t_4}^{i_4}
\widehat{t_5}^{i_5}\widehat{t_6}^{i_6}
+\displaystyle\sum_{\underline{j}\in I_2}b_{\underline{j}}\widehat{t_3}^{i_3}\widehat{t_4}^{i_4}
\widehat{t_5}^{i_5}\widehat{t_6}^{i_6}=&\displaystyle\sum_{\underline{l}\in J}c_{\underline{l}}
\widehat{f_{1}}^{i_1+1}\widehat{t_3}^{i_3+1}
\widehat{t_4}^{i_4}
\widehat{t_5}^{i_5+1}\widehat{t_6}^{i_6}\\
&-\sum_{\underline{l}\in J}\frac{1}{2}\beta c_{\underline{l}}\widehat{f_{1}}^{i_1}
\widehat{t_3}^{i_3}\widehat{t_4}^{i_4-1}
\widehat{t_5}^{i_5+1}\widehat{t_6}^{i_6-1}\\&
-\displaystyle\sum_{\underline{l}\in J}\alpha c_{\underline{l}}\widehat{f_{1}}^{i_1}\widehat{t_3}^{i_3}
\widehat{t_4}^{i_4}
\widehat{t_5}^{i_5}\widehat{t_6}^{i_6}.
\end{align*}
 Suppose that there exists $(i_1,i_3,i_4,i_5,i_6)\in J$ such that $c_{(i_1,i_3,i_4,i_5,i_6)}\neq 0.$
  Let $(w_1,w_3,w_4,w_5,w_6)\in J$ be the greatest element (in the lexicographic order on 
$\mathbb{N}^2\times \mathbb{Z}^3$) of  $J$ such that $c_{(w_1,w_3,w_4,w_5,w_6)}\neq 0.$ Since
   $\left( \widehat{f_{1}}^{k_1}\widehat{t_3}^{k_3}
\widehat{t_4}^{k_4}\widehat{t_5}^{k_5}\widehat{t_6}^{k_6}\right) _{(k_1,k_3,\ldots,k_6)\in \mathbb{N}^2\times \mathbb{Z}^3}$ is a basis of  $\mathcal{A}^{(4)}S_4^{-1},$ this implies that the coefficients of  $\widehat{f_{1}}^{w_1+1}\widehat{t_3}^{w_3+1}\widehat{t_4}^{w_4}\widehat{t_5}^{w_5+1}
  \widehat{t_6}^{w_6}$ in the above equality can be identified as: $c_{(w_1,w_3,w_4,w_5,w_6)}=0,$ a contradiction! Therefore, $c_{(i_1,i_3,i_4,i_5,i_6)}= 0$ for all $(i_1,i_3,i_4,i_5,i_6)\in J.$  This further implies that
  $$\displaystyle\sum_{\underline{i}\in I_1}a_{ \underline{i}}\widehat{f_{1}}^{i_1}\widehat{t_4}^{i_4}
\widehat{t_5}^{i_5}\widehat{t_6}^{i_6}
+\displaystyle\sum_{\underline{j}\in I_2}b_{\underline{j}}\widehat{t_3}^{i_3}\widehat{t_4}^{i_4}
\widehat{t_5}^{i_5}\widehat{t_6}^{i_6}
=0.$$
Consequently,  $a_{\underline{i}}$ and $ b_{\underline{j}}$ are all zero.   In conclusion, $\mathfrak{P}_4$ is a linearly independent set.
\end{proof}

\subsubsection{Basis for $\mathcal{R}_5.$}
We will identify $\mathcal{R}_5$ with ${\mathcal{A}_{\alpha}^{(5)}S_5^{-1}}/{\langle \widehat{\Omega}_{2}-\beta\rangle},$ where $\mathcal{A}_\alpha^{(5)}S_5^{-1}=\mathcal{A}^{(5)}S_5^{-1}/\langle \Omega_{1}-\alpha \rangle.$ Note that the canonical images of $X_{i,5}$ (resp. $T_i$) in $\mathcal{A}_{\alpha}^{(5)}$ will be denoted by $\widehat{x_{i,5}}$ (resp. $\widehat{t_i}$) for all $1\leq i\leq 6$.
 We now find a basis for $\mathcal{A}_\alpha^{(5)}S_5^{-1}.$ Recall from Subsection \ref{pdda} that  $\Omega_{1}=\zbar_1T_3T_5-\frac{1}{2}\zbar_2T_5$ and $\Omega_2=\zbar_2T_4T_6-\frac{2}{3} T_3^3T_6$ in $\mathcal{A}^{(5)}$ (remember that $\zbar_1:=X_{1,5}$ and $\zbar_2:=X_{2,5}$ ). Since $z_2t_4t_6-\frac{2}{3} t_3^3t_6=\beta$  and $\widehat{z_1}\widehat{t_3}\widehat{t_5}-\frac{1}{2} \widehat{z_2}\widehat{t_5}=\alpha$ in $\mathcal{R}_5$ and  $\mathcal{A}_\alpha^{(5)}S_5^{-1}$ respectively, we have the relation $\widehat{z_{2}}=
 2\widehat{z_{1}}\widehat{t_3}-2\alpha\widehat{t_5}^{-1}$ in $A_\alpha^{(5)}S_5^{-1}$ and, in $R_5,$ we have the following two relations:     
 \begin{align}
 z_{2}&=2(z_{1}t_3-\alpha t_5^{-1}).\label{pde2}
 \\
 t_3^3&=\dfrac{3}{2}(z_{2}t_4-\beta t_6^{-1})=3z_{1}t_3t_4-\dfrac{3}{2}\beta t_6^{-1}-3\alpha t_4t_5^{-1}.\label{pde3}
 \end{align}
 \begin{pro} 
 \label{pc4p20}
The set $\mathfrak{P}_5=\left\lbrace z_{1}^{i_1}t_3^{\xi}t_4^{i_4}t_5^{i_5}t_6^{i_6}\mid (\xi, i_1,i_4,i_5,i_6)\in \{0,1,2\}\times \mathbb{N}^2\times\mathbb{Z}^2\right\rbrace $ is a $\mathbb{K}$-basis of $\mathcal{R}_5$.
\end{pro}
\begin{proof}
One can easily show that the family $\left( \widehat{z_1}^{k_1}\widehat{t_3}^{k_3}\widehat{t_4}^{k_4}\widehat{t_5}^{k_5}\widehat{t_6}^{k_6}\right)_{(k_1,k_3,k_4,k_5,k_6)\in\mathbb{N}^2\times\mathbb{Z}^2}$ is a basis of $\mathcal{A}_{\alpha}^{(5)}S_5^{-1}/\langle \widehat{\Omega}_{2}-\beta\rangle$.
Since $\mathcal{R}_5$ is identified with $\mathcal{A}_{\alpha}^{(5)}S_5^{-1}/{\langle \widehat{\Omega}_{2}-\beta\rangle},$ we show that 
$z_1^{k_1}t_3^{k_3}t_4^{k_4}t_5^{k_5}t_6^{k_6}$ can be written as a finite linear combination of the elements of $\mathfrak{P}_5$ for all $(k_1,k_3,k_4,k_5,k_6)\in \mathbb{N}^3\times \mathbb{Z}^2.$ It is sufficient to do this by an induction on $k_3.$ The result is obvious when $k_3=0, 1, 2.$ For $k_3\geq 2,$ suppose that 
$$z_1^{k_1}t_3^{k_3}t_4^{k_4}t_5^{k_5}t_6^{k_6}=\sum_{(\xi,\underline{i})\in I}a_{(\xi,\underline{i})} z_{1}^{i_1}t_3^\xi t_4^{i_4}t_5^{i_5}t_6^{i_6},$$
where $I$ is a finite subset of $\{0,1,2\}\times \mathbb{N}^2\times \mathbb{Z}^2,$ and the $a_{(\xi,\underline{i})}$ are all scalars.  It follows that
$$z_1^{k_1}t_3^{k_3+1}t_4^{k_4}t_5^{k_5}t_6^{k_6}= \left( z_1^{k_1}t_3^{k_3}t_4^{k_4}t_5^{k_5}t_6^{k_6}\right) t_3=\sum_{(\xi,\underline{i})\in I} a_{(\xi,\underline{i})} z_{1}^{i_1}t_3^{\xi+1} t_4^{i_4}t_5^{i_5}t_6^{i_6}.$$ 
Now, $z_{1}^{i_1}t_3^{\xi+1} t_4^{i_4}t_5^{i_5}t_6^{i_6}\in$ Span($\mathfrak{P}_5$) when $\xi=0,1.$ For
 $\xi=2,$ 
one can easily verify that $z_{1}^{i_1}t_3^{3} t_4^{i_4}t_5^{i_5}t_6^{i_6}\in $ Span($\mathfrak{P}_5$) by using the relation in \eqref{pde3}. Therefore, by the principle of mathematical induction, $\mathfrak{P}_5$ spans $\mathcal{R}_5.$

We now prove that $\mathfrak{P}_5$ is a linearly independent set. Suppose that
$$\sum_{(\xi,\underline{i})\in I}a_{(\xi,\underline{i})} z_{1}^{i_1}t_3^\xi t_4^{i_4}t_5^{i_5}t_6^{i_6}=0.$$
Since $\mathcal{R}_5$ is identified with ${\mathcal{A}_{\alpha}^{(5)}S_5^{-1}}/{\langle \widehat{\Omega}_{2}-\beta\rangle},$ we have that
$$\sum_{(\xi,\underline{i})\in I}a_{(\xi.\underline{i})}\widehat{z_1}^{i_1}\widehat{t_3}^\xi \widehat{t_4}^{i_4}\widehat{t_5}^{i_5}\widehat{t_6}^{i_6}=\langle \widehat{\Omega}_{2}-\beta\rangle\nu,$$ where 
$\nu\in \mathcal{A}_{\alpha}^{(5)}S_5^{-1}.$ Write $\nu=\sum_{\underline{j}\in J}b_{\underline{j}}\widehat{z_1}^{i_1}\widehat{t_3}^{i_3}\widehat{t_4}^{i_4}\widehat{t_5}^{i_5}
\widehat{t_6}^{i_6},$ with  $\underline{j}=(i_1,i_3,i_4,i_5,i_6)\in J\subset\mathbb{N}^3\times \mathbb{Z}^2$ and  $b_{\underline{j}}\in \mathbb{K}.$ Given that $\Omega_2=\zbar_2T_4T_6-\frac{2}{3}T_3^3T_6$ in $\mathcal{A}^{(5)}$  and the relation \eqref{pde2}, one can  deduce that
 $$\widehat{\Omega}_2=\widehat{z_2}\widehat{t_4}\widehat{t_6}-\frac{2}{3}\widehat{t_3}^3\widehat{t_6}=2 \widehat{z_1}\widehat{t_3}
\widehat{t_4}\widehat{t_6}-2\alpha \widehat{t_4}\widehat{t_5}^{-1}\widehat{t_6}-\frac{2}{3}\widehat{t_3}^3\widehat{t_6}.$$  Therefore,
\begin{align*}
\sum_{(\xi,\underline{i})\in I}a_{(\xi.\underline{i})}\widehat{z_1}^{i_1}\widehat{t_3}^\xi \widehat{t_4}^{i_4}\widehat{t_5}^{i_5}\widehat{t_6}^{i_6}=&\sum_{\underline{j}\in J}2b_{\underline{j}}\widehat{z_1}^{i_1+1}\widehat{t_3}^{i_3+1}\widehat{t_4}^{i_4+1}
\widehat{t_5}^{i_5}
\widehat{t_6}^{i_6+1}\\
&-\sum_{\underline{j}\in J}\frac{2}{3}b_{\underline{j}}\widehat{z_1}^{i_1}\widehat{t_3}^{i_3+3}\widehat{t_4}^{i_4}\widehat{t_5}^{i_5}
\widehat{t_6}^{i_6+1}
\\
&- \sum_{\underline{j}\in J} 2\alpha b_{\underline{j}} \widehat{z_1}^{i_1}\widehat{t_3}^{i_3}\widehat{t_4}^{i_4+1}\widehat{t_5}^{i_5-1}
\widehat{t_6}^{i_6+1}\\&
-\sum_{\underline{j}\in J}\beta b_{\underline{j}}\widehat{z_1}^{i_1}\widehat{t_3}^{i_3}\widehat{t_4}^{i_4}\widehat{t_5}^{i_5}
\widehat{t_6}^{i_6}.
\end{align*}
Suppose that there exists $(i_1,i_3,i_4,i_5,i_6)\in J$ such that $b_{(i_1,i_3,i_4,i_5,i_6)}\neq 0.$
 Let $(w_1,w_3,w_4,w_5,w_6)\in J$ be the greatest element (in the lexicographic order on $\mathbb{N}^3\times \mathbb{Z}^2$) of $J$ such that $b_{(w_1,w_3,w_4,w_5,w_6)}\neq 0.$ Given that $\left( \widehat{z_{1}}^{k_1}\widehat{t_3}^{k_3}
\widehat{t_4}^{k_4}\widehat{t_5}^{k_5}\widehat{t_6}^{k_6}\right)_{(k_1,k_3,\ldots,k_6)\in \mathbb{N}^3\times \mathbb{Z}^2}$ is a basis of 
$\mathcal{A}_\alpha^{(5)}S_5^{-1},$
 one can identify the coefficients of  $\widehat{z_{1}}^{w_1+1}\widehat{t_3}^{w_3+1}\widehat{t_4}^{w_4+1}\widehat{t_5}^{w_5}
  \widehat{t_6}^{w_6+1}$ in the above equality as: $2 b_{(w_1,w_3,w_4,w_5,w_6)}=0.$ Hence, $b_{(w_1,w_3,w_4,w_5,w_6)}=0,$ a contradiction! Therefore, $b_{(i_1,i_3,i_4,i_5,i_6)}= 0$ for all $(i_1,i_3,i_4,i_5,i_6)\in J.$  Consequently,
  $$\sum_{(\xi,\underline{i})\in I}a_{(\xi.\underline{i})}\widehat{z_1}^{i_1}\widehat{t_3}^\xi \widehat{t_4}^{i_4}\widehat{t_5}^{i_5}\widehat{t_6}^{i_6}=0.$$
It follows that $a_{(\xi,\underline{i})}=0$ for all $(\xi,\underline{i})\in I.$  As a result, $\mathfrak{P}_5$ is a linearly independent set. 
\end{proof}
\begin{cor}
\label{pc4c1}
Let $I$ be a finite subset of $\{0,1,2\}\times \mathbb{N}\times\mathbb{Z}^3$ and $(a_{(\xi,\underline{i})})_{\underline{i}\in I}$ be a family of scalars.  If $$\sum_{(\xi,\underline{i})\in I}a_{(\xi.\underline{i})} z_{1}^{i_1}t_3^\xi t_4^{i_4}
t_5^{i_5}t_6^{i_6}=0, 
$$ then $a_{(\xi,\underline{i})}=0$ for all $(\xi, \underline{i})\in I.$
\end{cor}
\begin{proof}
When $i_4\geq 0,$ then the result is obvious as a  result of Proposition \ref{pc4p20}.
For $i_4<0,$ multiply both sides of the equality enough times by $t_4$ to kill all the negative powers of $t_4,$ and then apply Proposition \ref{pc4p20} to complete the proof.
\end{proof}
\begin{rem}
We were not successful in finding a basis for $\mathcal{R}_6.$ However, this has no effect on our main results in this section. Since $\mathcal{R}_7=\mathcal{A}_{\alpha,\beta},$ we already have a basis for $\mathcal{R}_7$ (Prop. \ref{pc3p5}).
\end{rem}
\begin{rem}
\label{prrr}
Recall the notations:
\begin{align*}
f_1:&=x_{1,4} & F_1:&=X_{1,4}\\
z_1:&=x_{1,5} & \zbar_1:&=X_{1,5} \\
z_2:&=x_{2,5} & \zbar_2:&=X_{2,5}.
\end{align*}
Then,
\begin{align*}
f_1&=t_1+\frac{1}{2} t_2t_3^{-1} & x_{3,6}&=t_3+\frac{3}{2} t_4t_5^{-1}\\
z_1&=f_1+\frac{1}{3} t_3^2t_4^{-1} &x_1&=x_{1,6}+\frac{1}{2} t_5t_6^{-1}\\
z_2&=t_2+\frac{2}{3} t_3^3t_4^{-1}& x_3&=x_{3,6}+t_5^2t_6^{-1}\\
x_{1,6}&=z_1+x_{3,6}t_5^{-1}-\frac{3}{4} t_4t_5^{-2}& x_4&=t_4+\frac{2}{3} t_5^3t_6^{-1}.
\end{align*}
These selected data of the PDDA of $\mathcal{A}_{\alpha,\beta}$   will be useful in Subsections \ref{pev1}.
\end{rem}
\subsection{Poisson derivations of $\mathcal{A}_{\alpha,\beta}$}
We are now ready to study the Poisson derivations of $\mathcal{A}_{\alpha,\beta}.$ We will begin with the case where both $\alpha$ and $\beta$ are non-zero, and subsequently look at the case where either $\alpha$ or $\beta$ is zero. 
\subsubsection{Poisson derivations of $\mathcal{A}_{\alpha,\beta} \ (\alpha, \beta \neq 0)$}
\label{pev1}
Throughout this subsection, we assume that $\alpha$ and 
$\beta$ are non-zero. Let Der$_P(\mathcal{A})$ be the collection of all the Poisson $\mathbb{K}$-derivations of $\mathcal{A}_{\alpha,\beta}$, and  $\mathcal{D}\in$ Der$_P(\mathcal{A}).$ Now, $\mathcal{D}$ extends uniquely to a Poisson derivation of each of the algebras in \eqref{pe} via localization. Hence, $\mathcal{D}$ is a Poisson derivation of the Poisson torus $\mathcal{R}_3=\mathbb{K}[t_3^{\pm 1}, t_4^{\pm 1}, t_5^{\pm 1}, t_6^{\pm 1}].$ It follows from  Corollary \ref{pic2} that $\mathcal{D}$ can be written as
$$\mathcal{D}=\text{ham}_{x}+\rho,$$ where  $\rho$ is a scalar Poisson derivation of $\mathcal{R}_3$ defined as $\rho(t_i)=\lambda_it_i, \ i=3,4,5,6;$ with $\lambda_i\in Z_P(\mathcal{R}_3)=\mathbb{K}$, and ham$_{x}=\{x,-\}:\mathcal{R}_3\rightarrow \mathcal{R}_3$ with $ x\in\mathcal{R}_3$ (see Corollary \ref{pic2}).   

We aim to describe $\mathcal{D}$ as a Poisson derivation of $\mathcal{A}_{\alpha,\beta}.$ We do this in several steps. We first describe $\mathcal{D}$ as a Poisson derivation of $\mathcal{R}_4.$ 
\begin{lem}
\begin{itemize}\label{ppl3}
\item[1.] $x\in \mathcal{R}_4.$
\item[2.]  $\lambda_5=\lambda_4+\lambda_6,$ $\rho(f_{1})=-(\lambda_3+\lambda_5)f_1$  
and $\rho(t_2)=-\lambda_5t_2.$ 
\item[3.] Set $\lambda_1:=-(\lambda_3+\lambda_5)$ and $\lambda_2:=-\lambda_5.$ Then, $\mathcal{D}(x_{\kappa,4})=\mathrm{ham}_{x}(x_{\kappa,4})+\lambda_\kappa x_{\kappa,4}$ for  all $\kappa\in \{1,\ldots,6\}.$ 
\end{itemize}
\end{lem}
\begin{proof}
1. Observe that $\mathcal{Q}:=\mathbb{K}[t_4^{\pm1},t_5^{\pm 1},t_6^{\pm 1}]$ is a subalgebra of both $\mathcal{R}_3$ and $\mathcal{R}_4.$ Furthermore, $\mathcal{R}_3=\mathcal{R}_4[t_3^{-1}].$ One can easily verify that $z:=t_4t_5^{-1}t_6$ is a Poisson central element of $\mathcal{Q}.$ 
Since $\mathcal{R}_3$ is a Poisson torus, it can be presented as a free $\mathcal{Q}$-module with basis 
$(t_3^j)_{j\in \mathbb{Z}}.$
One can therefore write $x\in \mathcal{R}_3$ as
 $x=\sum_{j\in\mathbb{Z}}b_jt_3^j, \  \text{where} \ b_j\in \mathcal{Q}.$
  Decompose $x$ as $x=x_-+x_+,$ where
$x_{+}:=\sum_{j\geq 0}b_jt_3^j$ and $x_{-}:=\sum_{j<0}b_jt_3^j.$  
Clearly, $x_+\in \mathcal{R}_4.$ We now want to show that $x_-\in \mathcal{R}_4.$ 
Write $x_-=\sum_{j=-1}^{-m}b_jt_3^j$ for some $ m\in \mathbb{N}_{>0}.$ 

Now,
$\mathcal{D}(z)=$ ham$_{x}(z)+\rho(z)= \mathrm{ham}_{x_-}(z)+\text{ham}_{x_+}(z)+(\lambda_4-\lambda_5+\lambda_6)z\in \mathcal{R}_4.$ We have that $\text{ham}_{x_+}(z)+(\lambda_4-\lambda_5+\lambda_6)z\in\mathcal{R}_4,$ hence
$\text{ham}_{x_-}(z)\in \mathcal{R}_4.$ Note that $\{t_3,z\}=2zt_3,$ and 
$\{\gamma,z\}=0$ for all $\gamma\in \mathcal{Q}$ since $z$ is Poisson central in $\mathcal{Q}.$
One can therefore express $\text{ham}_{x_-}(z)$ as follows:
$$\text{ham}_{x_-}(z)=\{x_-,z\}=\sum_{j=-1}^{-m}b_j\{t_3^j,z\}=\sum_{j=-1}^{-m}2jb_jzt_3^j\in \mathcal{R}_4.$$
Let $n\in \mathbb{N}_{>0},$ and set 
$$\ell^{(n)}:=\underbrace{\{\{ \ldots \{}_{n-\text{times}}x_-,\underbrace{z\}, z\},\ldots , z\}}_{n-\text{times}}\in \mathcal{R}_4.$$   We claim that $$\ell^{(n)}=\sum_{j=-1}^{-m}(2j)^nz^nb_jt_3^j,$$ for all $n\in \mathbb{N}_{>0}.$  Observe that $$\ell^{(1)}=\text{ham}_{x_-}(z)=\sum_{j=-1}^{-m}2jb_jzt_3^j,$$ hence the result is true for $n=1.$ Suppose that 
the result is true for $n\geq 1$. Then,
$$\ell^{(n+1)}=\{\ell^{(n)},z\}=\sum_{j=-1}^{-m}(2j)^nz^nb_j\{t_3^j,z\}=\sum_{j=-1}^{-m}(2j)^{n+1}z^{n+1}b_jt_3^j$$ as expected. By the principle of mathematical induction, the claim is proved.

Given that  $\ell^{(n)}=\sum_{j=-1}^{-m}(2j)^nz^nb_jt_3^j,$ it follows that $$\mu_n=\sum_{j=-1}^{-m}(2j)^n b_jt_3^j, \  \ \text{where} \ \ \mu_n:=\ell^{(n)}z^{-n}\in \mathcal{R}_4.$$ The above equality can be written as a matrix equation:
$$
\begin{bmatrix}
-2&-4&-6&\cdots &-2m\\
(-2)^2&(-4)^2&(-6)^2&\cdots &(-2m)^2\\
(-2)^3&(-4)^3&(-6)^3&\cdots &(-2m)^3\\
\vdots&\vdots&\vdots&\ddots&\vdots\\
(-2)^m&(-4)^m&(-6)^m&\cdots &(-2m)^m\\
\end{bmatrix}
\begin{bmatrix}
b_{-1}t_3^{-1}\\
b_{-2}t_3^{-2}\\
b_{-3}t_3^{-3}\\
\vdots\\
b_{-m}t_3^{-m}\\
\end{bmatrix}=
\begin{bmatrix}
\mu_1\\
\mu_2\\
\mu_3\\
\vdots\\
\mu_m\\
\end{bmatrix}.
$$
One can observe that the coefficient matrix
$$\begin{bmatrix}
-2&-4&-6&\cdots &-2m\\
(-2)^2&(-4)^2&(-6)^2&\cdots &(-2m)^2\\
(-2)^3&(-4)^3&(-6)^3&\cdots &(-2m)^3\\
\vdots&\vdots&\vdots&\ddots&\vdots\\
(-2)^m&(-4)^m&(-6)^m&\cdots &(-2m)^m\\
\end{bmatrix}$$ is similar to a Vandermonde matrix (since the terms in each column form a geometric sequence) which is well known to be invertible. This therefore implies that each $b_jt_3^j$ is a linear combination of the 
$\mu_n\in \mathcal{R}_4.$ As a result, $b_jt_3^j\in \mathcal{R}_4$ for all 
$j\in \{-1,\ldots, -m\}.$ Consequently, $x_-=\sum_{j=-1}^{-m}b_jt_3^j \in \mathcal{R}_4$ as desired. 

2. Recall that $\rho (t_\kappa)=\lambda_\kappa t_\kappa$ for all $\kappa\in \{3,4,5,6\}$ and $\lambda_\kappa\in \mathbb{K}.$
From Remark \ref{prrr}, we have that 
$f_{1}=t_1+\frac{1}{2} t_2t_3^{-1}.$ Again, recall from Subsection \ref{sub3} that $t_1=\alpha t_3^{-1}t_5^{-1}$ and $t_2=\beta t_4^{-1}t_6^{-1}$ in $\mathcal{R}_3=\mathscr{P}_{\alpha,\beta}.$ As a result, $f_1=\alpha t_3^{-1}t_5^{-1}+\frac{1}{2}\beta t_3^{-1} t_4^{-1}t_6^{-1}.$  Therefore,
\begin{align}
\rho(f_{1})
=&-(\lambda_3+\lambda_5)\alpha t_3^{-1}t_5^{-1}-\frac{1}{2}(\lambda_3+\lambda_4+\lambda_6)\beta t_3^{-1}t_4^{-1}t_6^{-1}. \label{pe5}
\end{align}
Also, $\rho(f_{1})\in \mathcal{R}_4$ implies that  $\rho(f_{1})$ can be written in terms of the basis $\mathfrak{P}_4$ of $\mathcal{R}_4$ (Prop. \ref{ev28}) as
\begin{align}
\label{pe6}
\rho(f_{1})=&\displaystyle\sum_{r> 0}a_rf_{1}^r+\displaystyle\sum_{s\geq 0}b_st_3^s,
\end{align}
where $a_r$ and $b_s$ belong to $\mathcal{Q}=\mathbb{K}[t_4^{\pm 1},t_5^{\pm 1},t_6^{\pm 1}].$ Note that
\begin{align}
f_{1}^r&=\left( \alpha t_3^{-1}t_5^{-1}+\frac{1}{2}\beta t_3^{-1} t_4^{-1}t_6^{-1}\right)^r=\sum_{i=0}^r {r\choose i}(\alpha)^i(\beta/2)^{r-1}t_3^{-r}t_4^{i-r}t_5^{-i}t_6^{i-r}\nonumber\\
&=c_rt_3^{-r},\label{pe7}
\end{align}
where 
\begin{align}
c_r=\displaystyle\sum_{i=0}^r {r\choose i}(\alpha)^i(\beta/2)^{r-i}t_4^{i-r}t_5^{-i}t_6^{i-r} \in \mathcal{Q}\setminus \{0\}.\label{pe81}
\end{align}
Substitute \eqref{pe7} into \eqref{pe6} to obtain
\begin{align}
\rho(f_{1})=&\displaystyle\sum_{r> 0}a_rc_rt_3^{-r}+\displaystyle\sum_{s\geq 0}b_st_3^s.\label{pe8}
\end{align}
One can rewrite \eqref{pe5} as
\begin{equation}
\rho (f_{1})=dt_3^{-1}, \label{pe9}
\end{equation} 
where $d=-(\lambda_5+\lambda_3)\alpha t_5^{-1}-\frac{1}{2}(\lambda_6+\lambda_4+\lambda_3)\beta t_4^{-1}t_6^{-1}\in\mathcal{Q}.$
Comparing \eqref{pe8} to \eqref{pe9} shows that 
$b_s=0$ for all $s\geq 0,$ and $a_rc_r=0$ for all $r\neq 1$. Therefore, $\rho(f_{1})=a_1c_1t_3^{-1}.$ Moreover, from \eqref{pe81}, $c_1=\frac{1}{2}\beta t_4^{-1}t_6^{-1}+\alpha t_5^{-1}.$ Hence, 
\begin{align}
\rho(f_{1})=a_1c_1t_3^{-1}&=a_1\left( \frac{1}{2}\beta t_4^{-1} t_6^{-1}+\alpha t_5^{-1}\right) t_3^{-1}
=a_1\alpha t_3^{-1}t_5^{-1}+\frac{1}{2}a_1\beta t_3^{-1}t_4^{-1}t_6^{-1}.\label{pe10}
\end{align}
Comparing \eqref{pe10} to \eqref{pe5} reveals that
$a_1=-(\lambda_5+\lambda_3)=-(\lambda_6+\lambda_4+\lambda_3).$ Consequently, $\lambda_5=\lambda_6+\lambda_4.$ Hence,
$\rho(f_{1})=-(\lambda_5+\lambda_3)\alpha t_3^{-1}t_5^{-1}
-\frac{1}{2}(\lambda_5+\lambda_3)\beta t_3^{-1}t_4^{-1} t_6^{-1}
=-(\lambda_5+\lambda_3)f_{1}.$
Finally, since $t_2=\beta t_4^{-1}t_6^{-1}$ in $\mathcal{R}_4,$ it follows that $$\rho(t_2)=
-(\lambda_6+\lambda_4)\beta t_4^{-1}t_6^{-1}=-(\lambda_6+\lambda_4)
t_2=-\lambda_5t_2.$$

3. Set $\lambda_1:=-(\lambda_3+\lambda_5)$ and $\lambda_2:=-\lambda_5.$ Remember that $f_1=x_{1,4}$ and $t_i=x_{i,4}$ ($2\leq i\leq 6$). It follows from points (1) and (2) that $\mathcal{D}(x_{\kappa,4})=\text{ham}_x(x_{\kappa,4})+\rho(x_{\kappa,4})=\text{ham}_x(x_{\kappa,4})
+\lambda_\kappa x_{\kappa,4}$ for all $\kappa\in \{1,\ldots,6\}.$ 

In conclusion, $\mathcal{D}=\text{ham}_x+\rho,$ with $x\in \mathcal{R}_4.$ 
\end{proof}
We now proceed to describe $\mathcal{D}$ as a Poisson derivation of $\mathcal{R}_5.$
\begin{lem}
\begin{itemize}
\item[1.] $x\in \mathcal{R}_5.$
\item[2.] $\lambda_4=3\lambda_3+\lambda_5,$  $\lambda_6=-3\lambda_3,$  $\rho(z_{1})=-(\lambda_3+\lambda_5)z_{1}$ and $\rho(z_{2})=-\lambda_5z_{2}.$ 
\item[3.]Set $\lambda_1:=-(\lambda_3+\lambda_5)$ and $\lambda_2:=-\lambda_5,$  then $\mathcal{D}(x_{\kappa,5})=\mathrm{ham}_x(x_{\kappa,5})+\lambda_\kappa x_{\kappa,5}$ for  all $\kappa\in \{1,\ldots,6\}.$
\end{itemize}
\end{lem}
\begin{proof}
In this proof, we denote $\underline{\upsilon}:=(i,j,k,l)\in  \mathbb{N}\times \mathbb{Z}^3.$

1.   We already know that 
 $x\in\mathcal{R}_4=\mathcal{R}_5[t_4^{-1}].$ Given the basis $\mathfrak{P}_5$ of $\mathcal{R}_5$ (Prop. \ref{pc4p20}),  $x$ can be written as 
 $x=\displaystyle\sum_{(\xi, \underline{\upsilon})\in I}a_{(\xi,\underline{\upsilon})}z_{1}^{i}t_3^{\xi}
t_4^{j}t_5^{k}t_6^{l},$ 
where $I$ is a finite subset of $\{0,1,2\}\times \mathbb{N}\times \mathbb{Z}^3$ and the
 $a_{(\xi, \underline{\upsilon})}$ are scalars.   Write 
 $x=x_-+x_+,$ where 
$$x_+=\sum_{\substack{(\xi, \underline{\upsilon})\in I\\ j\geq 0}} a_{(\xi, {\underline{\upsilon})}}z_{1}^{i}t_3^{\xi}
 t_4^{j}t_5^{k}t_6^{l}\ \  \text{and} \ \ x_-=\displaystyle \sum_{\substack{(\xi, \underline{\upsilon})\in I \\ j<0}}a_{(\xi,\underline{\upsilon})}z_{1}^{i}t_3^{\xi}
t_4^{j}t_5^{k}t_6^{l}.$$    
 Suppose that $x_-\neq 0.$ Then, there exists a minimum  $j_0<0$ such that  $a_{(\xi, i,j_0,k,l)}\neq 0$ for some
$(\xi,i,j_0,k,l)\in I$ and $a_{(\xi,i,j,k,l)}=0$ for all $(\xi,i,j_0,k,l)\in I$ with $j<j_0.$  Given this assumption, write
\begin{align*}
 x_-=\displaystyle\sum_{ \substack{(\xi,\underline{\upsilon})\in I\\ j_0 \leq j\leq -1}}a_{(\xi,\underline{\upsilon})}z_{1}^{i}t_3^{\xi}
 t_4^{j}t_5^{k}t_6^{l}.
\end{align*}
Contrary to our assumption, we aim to show that $x_-=0.$
 Let $s=3,6.$ Then, $\mathcal{D}(t_s)=\text{ham}_{x_+}(t_s)+\text{ham}_{x_-}(t_s)+
 \rho(t_s)\in \mathcal{R}_5$ for each $s=3,6.$ This implies that $\text{ham}_{x_-}(t_s)\in \mathcal{R}_5,$ since $\text{ham}_{x_+}(t_s)+\rho(t_s)=\text{ham}_{x_+}(t_s)+
 \lambda_st_s\in \mathcal{R}_5.$ Set $\underline{w}:=(i,j,k,l)\in \mathbb{N}^2\times \mathbb{Z}^2.$
 One can therefore write
$\text{ham}_{x_-}(t_s)\in \mathcal{R}_5$ in terms of the basis $\mathfrak{P}_5$ of $\mathcal{R}_5$ as:
  \begin{align}
\text{ham}_{x_-}(t_s)&= \sum_{(\xi,\underline{w})\in J} b_{(\xi, {\underline{w}})}z_{1}^{i}
t_3^{\xi}t_4^{j}t_5^{k}t_6^{l}, \label{pe12}
 \end{align}
 where  $ J$ is a finite subset of $\{0,1,2\}\times \mathbb{N}^2\times \mathbb{Z}^2$ and the $b_{(\xi, \underline{w})}$ are scalars. 
 
 When $s=6,$ then using Remark \ref{ev29}(2), one can also express $\text{ham}_{x_-}(t_6)$ as:
 \begin{align*}
 \text{ham}_{x_-}(t_6)&= \sum_{ \substack{(\xi,\underline{\upsilon})\in I\\ j_0 \leq j\leq -1}} 3(k+j-i)a_{(\xi,\underline{\upsilon})}
z_{1}^{i}t_3^{\xi}
t_4^{j}t_5^{k}t_6^{l+1}.
\end{align*}  
Comparing this expression for $\text{ham}_{x_-}(t_6)$ to \eqref{pe12} (when $s=6$), we have that
$$
\sum_{ \substack{(\xi,\underline{\upsilon})\in I\\ j_0 \leq j\leq -1}} 3(k+j-i)a_{(\xi,\underline{\upsilon})} 
z_{1}^{i}t_3^{\xi}
t_4^{j}t_5^{k}t_6^{l+1}= \sum_{(\xi,\underline{w})\in J} b_{(\xi, {\underline{w}})}z_{1}^{i}
t_3^{\xi}t_4^{j}t_5^{k}t_6^{l}.$$ 

As $\mathfrak{P}_5$ is a basis for $\mathcal{R}_5$ (Prop. \ref{pc4p20}), we deduce from Corollary \ref{pc4c1} that\newline 
 $\left(z_{1}^{i}
t_3^{\xi}t_4^{j}t_5^{k}t_6^{l}\right)_{(i\in \mathbb{N}; j,k,l\in \mathbb{Z}; \xi\in \{0,1,2\})}$ is a basis for $\mathcal{R}_5[t_4^{-1}].$
  Now, at $j=j_0,$ denote $\underline{\upsilon}=(i,j,k,l)$ by $\underline{\upsilon}_0:=
 (i,j_0,k,l).$ Since $\underline{v}_0\in \mathbb{N}\times \mathbb{Z}^3$ (with $j_0<0$) and $\underline{w}=(i,j,k,l)\in \mathbb{N}^2\times \mathbb{Z}^2$ (with $j\geq 0$), it follows from the above equality that, at $\underline{\upsilon}_0,$  we must have
 \begin{align}
 k=i-j_0,\label{pe13}
 \end{align}
 for some $(\xi, \underline{v}_0)\in I.$ 
 
Similarly, when $s=3,$ then using Remark \ref{ev29}(2), one can also express $\text{ham}_{x_-}(t_3)$ as: 
 \begin{align*}
 \text{ham}_{x_-}(t_3)=&-\sum \left[\frac{3}{2}\beta(3i-k-3j_0)a_{2, i,j_0,k,l+1}+2(i+1)\alpha
 a_{(0, i+1,j_0,k+1,l)}\right]z_1^it_4^{j_0}t_5^{k}t_6^l\\
 &+\sum \left[(3i-k-3j_0)a_{(0,i,j_0,k,l)}-2(i+1)\alpha a_{(1, i+1,j_0,k+1,l)}\right]
 z_1^it_3t_4^{j_0}t_5^{k}t_6^l\\
 &+\sum \left[(3i-k-3j_0)a_{(1, i,j_0,k,l)}-2(i+1)\alpha a_{(2, i+1,j_0,k+1,l)}\right]
 z_1^it_3^2t_4^{j_0}t_5^{k}t_6^l+\mathcal{K},
\end{align*} 
where $\mathcal{K}\in \text{Span}\left( \mathfrak{P}_5\setminus \{z_{1}^{i}t_3^{\xi}
 t_4^{j_0}t_5^{k}t_6^{l}\mid (\xi, i,j_0,k,l)\in \{0,1,2\}\times \mathbb{N}\times \mathbb{Z}^3 \}\right)$ 
(note that one will need the following two expressions 
$z_{2}=2(z_{1}t_3-\alpha t_5^{-1})\ \ \text{and} \ \ t_3^3=3z_1t_3t_4-3\alpha t_4t_5^{-1}-\dfrac{3\beta}{2} t_6^{-1}$ from \eqref{pde2} and \eqref{pde3} to express some of the monomials in terms of the basis $\mathfrak{P}_5$ of $\mathcal{R}_5$). Comparing this expression for 
 $\text{ham}_{x_-}(t_3)$ to \eqref{pe12} (when $s=3$) reveals that 
\begin{align*}
\sum_{(\xi,\underline{w})\in J} & b_{(\xi, {\underline{w}})}z_{1}^{i}
t_3^{\xi}t_4^{j}t_5^{k}t_6^{l}=\\
 &-\sum \left[\frac{3}{2}\beta(3i-k-3j_0)a_{(2, i,j_0,k,l+1}+2(i+1)\alpha
 a_{(0, i+1,j_0,k+1,l)}\right]z_1^it_4^{j_0}t_5^{k}t_6^l\\
 &+\sum \left[(3i-k-3j_0) a_{(0,i,j_0,k,l)}-2(i+1)\alpha a_{(1, i+1,j_0,k+1,l)}\right]
 z_1^it_3t_4^{j_0}t_5^{k}t_6^l\\
 &+\sum \left[(3i-k-3j_0)a_{(1, i,j_0,k,l)}-2(i+1)\alpha a_{(2, i+1,j_0,k+1,l)}\right]
 z_1^it_3^2t_4^{j_0}t_5^{k}t_6^l+\mathcal{K}.
\end{align*}  
We have already established that 
 $\left(z_{1}^{i}
t_3^{\xi}t_4^{j}t_5^{k}t_6^{l}\right)_{(i\in \mathbb{N}; j,k,l\in \mathbb{Z}; \xi\in \{0,1,2\})}$ is a basis for $\mathcal{R}_5[t_4^{-1}].$ Since $\underline{v}_0=(i,j_0,k,l)\in \mathbb{N}\times \mathbb{Z}^3$ (with $j_0<0$) and $\underline{w}=(i,j,k,l)\in \mathbb{N}^2\times \mathbb{Z}^2$ (with $j\geq 0$), it follows from the above equality that, at $\underline{\upsilon}_0,$  we must have
 \begin{align}
\frac{3}{2}\beta(3i-k-3j_0)a_{(2, i,j_0,k,l+1)}+2(i+1)\alpha
 a_{(0, i+1,j_0,k+1,l)}&=0,\label{2pe}\\
(3i-k-3j_0)a_{(0,i,j_0,k,l)}-2(i+1)\alpha a_{(1, i+1,j_0,k+1,l)}&=0,\label{3pe}\\
 (3i-k-3j_0)a_{(1, i,j_0,k,l)}-2(i+1)\alpha a_{(2, i+1,j_0,k+1,l)}&=0.\label{4pe}
 \end{align}
Suppose that there exists $(\xi, i,j_0,k,l)\in I$ such that $3i-k-3j_0=0.$  Then, 
\begin{align}
k=3(i-j_0).\label{5pe}
\end{align}
Comparing \eqref{5pe} to \eqref{pe13} clearly shows that
 $i-j_0=0$ which implies that $ i=j_0<0,$ a contradiction (note that $i\geq 0$). 
Therefore, $3i-k-3j_0\neq 0$ for all $(\xi, i,j,k)\in I.$

Now, observe that if there exists $\xi\in \{0,1,2\}$ such that $a_{(\xi, i,j_0,k,l)}=0$ for all $(i,j_0,k,l)\in \mathbb{N}\times \mathbb{Z}^3$, then one can easily deduce from equations \eqref{2pe}, \eqref{3pe} and \eqref{4pe} that $a_{(\xi, i,j_0,k,l)}=0$ for all $(\xi, i,j_0,k,l)\in I.$  This  contradicts our initial assumption. Therefore, for each $\xi\in \{0,1,2\},$
there exists some $(i,j_0,k,l)\in \mathbb{N}\times \mathbb{Z}^3$ such that $a_{(\xi, i,j_0,k,l)}\neq 0.$  Without loss of generality, let
$(u,j_0,v,w)$ be the greatest element in the lexicographic order on $\mathbb{N}\times \mathbb{Z}^3$ such that $a_{(0,u,j_0,v,w)}\neq 0$ and $a_{(0,i,j_0,k,l)}=0$ for all $i>u.$

From \eqref{3pe}, at $(i,j_0,k,l)=(u,j_0,v,w)$, we have: 
\begin{align*}
 (3u-v-3j_0)a_{(0, u,j_0,v,w)}-2(u+1)\alpha a_{(1, u+1,j_0,v+1,w)}&=0.
\end{align*}
From \eqref{4pe}, at $(i,j_0,k,l)=(u+1,j_0,v+1,w)$, we have: 
\begin{align*}
(3u-v-3j_0)a_{(1, u+1,j_0,v+1,w)}-2(u+1)\alpha a_{(2, u+2,j_0,v+2,w)}&=0.
\end{align*}
Finally, from \eqref{2pe}, at $(i,j_0,k,l)=(u+2,j_0,v+2,w-1),$ we have: 
\begin{align*}
\frac{3}{2}\beta(3u-v-3j_0)a_{(2,u+2,j_0,v+2,w)}+2(u+1)\alpha
 a_{(0,u+3,j_0,v+3,w-1)}&=0.
\end{align*}
Since $3i-k-3j_0\neq 0$  for all $(i,j_0,k,l)\in I$ and $u+3>u,$ it follows from the above three displayed equations (beginning from the last one)  that
$$a_{(0, u+3,j_0,v+3,w-1)}=0\Rightarrow 
a_{(2, u+2,j_0,v+2,w)}=0\Rightarrow a_{(1, u+1,j_0,v+1,w)}=0\Rightarrow 
a_{(0, u,j_0,v,w)}=0,$$ a contradiction!
Hence, $a_{(0, i,j_0,k,l)}=0$ for all $(i,j_0,k,l)\in \mathbb{N}\times \mathbb{Z}^3.$
From \eqref{2pe}, \eqref{3pe} and \eqref{4pe}, one can then conclude that  $a_{(\xi, i,j_0,k,l)}=0$ for all $(\xi, i,j_0,k,l)\in I.$ This is a contradiction to our assumption, hence $x_-=0.$ Consequently, $x=x_+\in \mathcal{R}_5$ as desired.

2. It follows from Remark \ref{prrr} that 
$z_{2}=t_{2}+\frac{2}{3} t_3^3t_4^{-1}.$  Since $\rho(t_\kappa)=\lambda_\kappa t_\kappa, \ \kappa\in \{2,3,4,5,6\},$  with $\lambda_2=-\lambda_5$ (see Lemma \ref{ppl3}), it follows that
\begin{align*}
\rho(z_{2})=&-\lambda_5t_{2}+\frac{2}{3}(3\lambda_3-\lambda_4)
t_3^3t_4^{-1}
=-\lambda_5z_{2}+\frac{2}{3}(3\lambda_3-\lambda_4+\lambda_5)
t_3^3t_4^{-1}.
\end{align*}
Furthermore, 
\begin{align*}
\mathcal{D}(z_{2})&=\text{ham}_x(z_{2})+\rho(z_{2})
=\text{ham}_x(z_{2})
-\lambda_5z_{2}+\frac{2}{3}(3\lambda_3-\lambda_4+\lambda_5)
t_3^3t_4^{-1}\in \mathcal{R}_5.
\end{align*}
We have that $(3\lambda_3-\lambda_4+\lambda_5)
t_3^3t_4^{-1}\in \mathcal{R}_5,$ since $\text{ham}_x(z_{2})
-\lambda_5z_{2}\in \mathcal{R}_5.$
This implies that
$(3\lambda_3-\lambda_4+\lambda_5)
t_3^3\in \mathcal{R}_5{t}_4.$  Set 
$w:=3\lambda_3-\lambda_4+\lambda_5.$ Suppose that $w\neq 0.$ From \eqref{pde3}, we have:
$$t_3^3=3z_{1}t_3t_4-\dfrac{3}{2}\beta t_6^{-1}-3\alpha t_4t_5^{-1}.$$ It follows that
 $$wt_3^3=
3wz_1t_3t_4-3w\alpha t_4t_5^{-1}-\frac{3}{2}w\beta t_6^{-1} \in \mathcal{R}_5t_4.$$ 
Since $t_3^3, \ t_4t_5^{-1}$ and $z_1t_3t_4$ are all elements of $\mathcal{R}_5t_4,$ this implies that
$t_6^{-1}\in \mathcal{R}_5t_4.$ Hence, $1\in \mathcal{R}_5t_4t_6,$ a contradiction (see the basis $\mathfrak{P}_5$ of 
$\mathcal{R}_5$ (Prop. \ref{pc4p20})). Therefore, 
$w=3\lambda_3-\lambda_4+\lambda_5=0,$ and so $\lambda_4=
3\lambda_3+\lambda_5.$ This further implies that $\rho(z_{2})=-\lambda_5 z_{2}$ as desired.

Again, from Lemma \ref{ppl3},  we have that $\rho(f_1)=-(\lambda_3+\lambda_5)f_1.$ Recall from Remark \ref{prrr} that $z_{1}=f_{1}+\frac{1}{3}t_3^2t_4^{-1}.$ 
It follows that
\begin{align*}
\rho(z_{1})=&-(\lambda_3+\lambda_5)f_{1}+\frac{1}{3}(2\lambda_3-\lambda_4)t_3^2t_4^{-1}
=-(\lambda_3+\lambda_5)z_{1}+\frac{1}{3}(3\lambda_3-\lambda_4+\lambda_5)t_3^2t_4^{-1}\\
=&-(\lambda_3+\lambda_5)z_{1}+\frac{1}{3}(3\lambda_3-(3\lambda_3+\lambda_5)+\lambda_5)t_3^2t_4^{-1}
=-(\lambda_3+\lambda_5)z_{1}.
\end{align*}
Finally, we know that $\rho(t_6)=\lambda_6t_6.$  From the relation \eqref{pde3}, we have: $$t_3^3=3z_1t_3t_4-3\alpha t_4t_5^{-1}-\dfrac{3\beta}{2} t_6^{-1}.$$
This implies that $$t_6^{-1}=\frac{2}{3\beta}(3z_1t_3t_4-3\alpha t_4t_5^{-1}-t_3^3).$$ Apply $\rho$ to this relation to obtain
$$-\lambda_6t_6^{-1}=3\lambda_3\left( \frac{2}{3\beta}\left(3z_1t_3t_4-3\alpha t_4t_5^{-1}-t_3^3 \right)\right).$$
Clearly, $\lambda_6=-3\lambda_3$ as desired.

3. Set $\lambda_1:=-(\lambda_3+\lambda_5)$ and $\lambda_2:=-\lambda_5.$ Remember that $z_1=x_{1,5}$ and $z_2=x_{2,5}.$  It follows from points (1) and (2) that $\mathcal{D}(x_{\kappa,5})=\mathrm{ham}_x(x_{\kappa,5})+\rho (x_{\kappa,5})=\mathrm{ham}_x(x_{\kappa,5})+\lambda_\kappa x_{\kappa,5}$ for  all $\kappa\in \{1,\ldots,6\}.$ 

In conclusion, $\mathcal{D}=\text{ham}_x+\rho$ with $x\in \mathcal{R}_5.$ 
\end{proof}

We are now ready to describe $\mathcal{D}$ as a Poisson derivation of $\mathcal{A}_{\alpha,\beta}.$
\begin{lem}
\label{ev30}
\begin{itemize}
\item[1.] $x\in \mathcal{A}_{\alpha,\beta}.$
\item[2.] $\rho(x_\kappa)=0$ for all
$\kappa\in\{1,\ldots,6\}.$
\item[3.] $\mathcal{D}=\mathrm{ham}_x.$ 
\end{itemize}
\end{lem}
\begin{proof}
In this proof, we denote $\underline{\upsilon}:=(i,j,k,l)\in \mathbb{N}^2\times\mathbb{Z}^2.$  Also, recall from the PDDA of $\mathcal{A}_{\alpha,\beta}$ at the beginning of this section that $t_5=x_5$ and $t_6=x_6.$

1. Given the basis $\mathfrak{P}$ of $\mathcal{A}_{\alpha,\beta}$ (Prop. \ref{pc3p5}), one can write
 $x\in\mathcal{R}_5=\mathcal{A}_{\alpha,\beta}[t_5^{-1},t_6^{-1}]$ as $$x=\displaystyle\sum_{(\epsilon_1, \epsilon_2,\underline{\upsilon})\in I}a_{(\epsilon_1,\epsilon_2, \underline{\upsilon})}x_{1}^{i}x_{2}^{j}x_3^{\epsilon_1}
x_4^{\epsilon_2}t_5^{k}t_6^{l},$$
where $I$ is a finite subset of $\{0,1\}^2\times \mathbb{N}^2\times\mathbb{Z}^2$ and $a_{(\epsilon_1,\epsilon_2,\underline{\upsilon})}$ are scalars. Write 
 $x=x_-+x_+,$ where 
 $$x_+=\sum_{\substack{(\epsilon_1, \epsilon_2,\underline{\upsilon})\in I\\ k, \ l\geq 0 }}a_{(\epsilon_1,\epsilon_2, \underline{\upsilon})}x_{1}^{i}x_{2}^{j}x_3^{\epsilon_1}
 x_4^{\epsilon_2}t_5^{k}t_6^{l},$$ and $$
x_-=\displaystyle\sum_{\substack{(\epsilon_1, \epsilon_2,\underline{\upsilon})\in I\\ k<0 \ \text{or} \ l<0}}a_{(\epsilon_1,\epsilon_2, \underline{\upsilon})}x_{1}^{i}x_{2}^{j}x_3^{\epsilon_1}
 x_4^{\epsilon_2}t_5^{k}t_6^{l}.$$

Suppose that $x_-\neq 0.$ Then, there exists a minimum negative integer $k_0$ or $l_0$ such that 
$a_{(\epsilon_1,\epsilon_2, i,j,k_0,l)}\neq 0$ or $a_{(\epsilon_1,\epsilon_2, i,j,k,l_0)}\neq 0$  for some  $(\epsilon_1,\epsilon_2, i,j,k_0,l), (\epsilon_1,\epsilon_2, i,j,k,l_0)\in I,$  
and  $a_{(\epsilon_1,\epsilon_2, i,j,k,l)}=0$ whenever $k<k_0$ or $l<l_0.$  Write
 $$x_-=\displaystyle\sum_{\substack{(\epsilon_1, \epsilon_2,\underline{\upsilon})\in I\\
 k_0\leq k \leq -1 \ \text{or} \ l_0\leq l\leq -1}}a_{(\epsilon_1,\epsilon_2, \underline{\upsilon})}x_{1}^{i}x_{2}^{j}x_3^{\epsilon_1}
 x_4^{\epsilon_2}t_5^{k}t_6^{l}.$$

Now $\mathcal{D}(x_3)=\text{ham}_{x_+}(x_3)+\text{ham}_{x_-}(x_3)+
 \rho(x_3)\in \mathcal{A}_{\alpha,\beta}.$  From Remark \ref{prrr}, we have that
 $x_3=x_{3,6}+t_5^2t_6^{-1}$ and $x_{3,6}=t_3+\frac{3}{2}t_4t_5^{-1}.$ Putting these two together gives
\begin{align*}
x_3
=t_3+\frac{3}{2}t_4t_5^{-1}+t_5^2t_6^{-1}.
\end{align*} 
Again, from Remark \ref{prrr}, we also have that $t_4=x_4-\frac{2}{3}
t_5^3t_6^{-1}.$  Note that $\rho(t_\kappa)
=\lambda_\kappa t_\kappa, \ \kappa = 3,4,5,6.$ 

Now,
 \begin{align}
 \rho(x_3)&=\lambda_3t_3+\frac{3}{2}(\lambda_4-\lambda_5)t_4t_5^{-1}+(2\lambda_5-\lambda_6)t_5^2t_6^{-1}\nonumber\\
 &=\lambda_3\left( x_{3,6}-\frac{3}{2}t_4t_5^{-1}\right) +\frac{3}{2}(\lambda_4-\lambda_5)t_4t_5^{-1}+(2\lambda_5-\lambda_6)
 t_5^2t_6^{-1}\nonumber\\
 &=\lambda_3x_{3,6}-\frac{3}{2}(\lambda_3-\lambda_4+\lambda_5)t_4
 t_5^{-1}+(2\lambda_5-\lambda_6)
 t_5^2t_6^{-1}\nonumber\\
 &=\lambda_3(x_3-t_5^2t_6^{-1})-\frac{3}{2}(\lambda_3-\lambda_4+\lambda_5)\left( x_{4}-\frac{2}{3}t_5^3t_6^{-1}\right) 
t_5^{-1}+(2\lambda_5-\lambda_6)
 t_5^2t_6^{-1}\nonumber\\
 &=\lambda_3x_3+\alpha_1x_4t_5^{-1}
 +\alpha_2t_5^2t_6^{-1},\label{pfe1}
\end{align}  
 where $\alpha_1=\frac{3}{2}(\lambda_4-\lambda_3-\lambda_5)$ and $\alpha_2=(3\lambda_5-\lambda_4-\lambda_6).$
Therefore,
$\mathcal{D}(x_3)=\text{ham}_{x_+}(x_3)+\text{ham}_{x_-}(x_3)+
 \lambda_3x_3+\alpha_1x_4t_5^{-1}
 +\alpha_2t_5^2t_6^{-1}\in \mathcal{A}_{\alpha,\beta}.$  It follows that
 $\mathcal{D}(x_3)t_5t_6=\text{ham}_{x_+}(x_3)t_5t_6+\text{ham}_{x_-}(x_3)t_5t_6+
 \lambda_3x_3t_5t_6+
 \alpha_1x_4t_6
 +\alpha_2t_5^3\in \mathcal{A}_{\alpha,\beta}.$ Hence, $\text{ham}_{x_-}(x_3)t_5t_6\in \mathcal{A}_{\alpha,\beta},$ since 
 $\text{ham}_{x_+}(x_3)t_5t_6+
 \lambda_3x_3t_5t_6+
 \alpha_1x_4t_6
 +\alpha_2t_5^3\in \mathcal{A}_{\alpha,\beta}.$
 
 One can also verify that
\begin{align}\text{ham}_{x_-}(x_3)t_5t_6
=\sum_{(\epsilon_1,\epsilon_2, \underline{v})\in I}&a_{(\epsilon_1,\epsilon_2, \underline{v})}\left(  (i+3j-3\epsilon_2-k) x_1^ix_2^jx_3^{\epsilon_1+1}x_4^{\epsilon_2}t_5^{k+1}t_6^{l+1}\right.\nonumber  \\ &-3kx_1^ix_2^jx_3^{\epsilon_1}x_4^{\epsilon_2+1}t_5^{k}t_6^{l+1} +ix_1^{i-1}x_2^{j+1}x_3^{\epsilon_1}x_4^{\epsilon_2}t_5^{k+1}t_6^{l+1}\nonumber \\
&\left.-6l x_1^ix_2^jx_3^{\epsilon_1}x_4^{\epsilon_2}t_5^{k+3}t_6^l \right).\label{pe15}
\end{align}  

Assume that there exists $l<0$ such that $a_{(\epsilon_1,\epsilon_2,i,j,k,l)}\neq 0.$ It follows from our initial assumption that
 $a_{(\epsilon_1,\epsilon_2,i,j,k,l_0)}\neq 0.$ Now, at $l=l_0,$ denote $\underline{\upsilon}=(i,j,k,l)$ by $\underline{\upsilon}_0:=(i,j,k,l_0).$   
  From \eqref{pe15}, we have that
 \begin{align*}
 \text{ham}_{x_-}(x_3)t_5t_6=&-\sum_{(\epsilon_1, \epsilon_2,\underline{\upsilon}_0)\in I}
6l_0
 a_{(\epsilon_1,\epsilon_2, \underline{\upsilon}_0)}x_{1}^{i}x_{2}^{j}x_3^{\epsilon_1}
 x_4^{\epsilon_2}t_5^{k+3}t_6^{l_0} +\mathcal{J}_1,
 \end{align*}
where $\mathcal{J}_1 \in \text{Span}\left( \mathfrak{P}\setminus 
 \{x_1^{i}x_2^{j}x_3^{\epsilon_1}x_4^{\epsilon_2}
 t_5^{k}t_6^{l_0}\mid  \epsilon_1,\epsilon_2\in \{0,1\}, \ k\in\mathbb{Z} \ \text{and} \  i,j\in \mathbb{N}  \}\right).$ 
 
Set $\underline{w}:=(i,j,k,l)\in \mathbb{N}^4.$ 
 One can also write $\text{ham}_{x_-}(x_3)t_5t_6\in \mathcal{A}_{\alpha,\beta}$ in terms of the basis $\mathfrak{P}$ of $\mathcal{A}_{\alpha,\beta}$ (Prop. \ref{pc3p5}) as:
\begin{align}
 \text{ham}_{x_-}(x_3)t_5t_6=&\sum_{(\epsilon_1,\epsilon_2, \underline{w})\in J}b_{(\epsilon_1,\epsilon_2,\underline{w})}x_1^{i}
x_2^{j}x_3^{\epsilon_1}x_4^{\epsilon_2}
t_5^{k}t_6^{l},\label{pc4e10}
 \end{align}
 where $J$ is a finite subset of  $\{0,1\}^2\times \mathbb{N}^4$ and $b_{(\epsilon_1,\epsilon_2,\underline{w})}\in \mathbb{K}.$  It follows that
 $$\sum_{(\epsilon_1,\epsilon_2, \underline{w})\in J}b_{(\epsilon_1,\epsilon_2,\underline{w})}x_1^{i}
x_2^{j}x_3^{\epsilon_1}x_4^{\epsilon_2}
t_5^{k}t_6^{l}=-\sum_{(\epsilon_1, \epsilon_2,\underline{\upsilon}_0)\in I}
6l_0
 a_{(\epsilon_1,\epsilon_2, \underline{\upsilon}_0)}x_{1}^{i}x_{2}^{j}x_3^{\epsilon_1}
 x_4^{\epsilon_2}t_5^{k+3}t_6^{l_0} +\mathcal{J}_1.$$
 As $\mathfrak{P}$ is a basis for $\mathcal{A}_{\alpha,\beta},$ we deduce from Corollary \ref{pc4c2} that \newline
 $\left(x_1^{i}
x_2^{j}x_3^{\epsilon_1}x_4^{\epsilon_2}
t_5^{k}t_6^{l}\right)_{((\epsilon_1,\epsilon_2, \underline{v})\in \{0,1\}^2\times \mathbb{N}^2 \times \mathbb{Z}^2)}$ is a basis for $\mathcal{A}_{\alpha,\beta}[t_5^{-1}, t_6^{-1}].$ Since $\underline{v}_0=(i,j,k,l_0)\in \mathbb{N}^2\times \mathbb{Z}^2$ (with $l_0<0$) and $\underline{w}=(i,j,k,l)\in \mathbb{N}^4$ (with $l\geq 0$)  in the above equality,  we must have
 $$6l_0 a_{(\epsilon_1,\epsilon_2,\underline{\upsilon}_0)}=0.$$
 Note that $l_0\neq 0,$ it follows that $a_{(\epsilon_1,\epsilon_2,\underline{\upsilon}_0)}=a_{(\epsilon_1,\epsilon_2,i,j,k,l_0)}$ are all zero, a contradiction! Therefore, $l\geq 0$ (i.e. there is no negative exponent of $t_6$).
 
Given that $l\geq 0,$ it follows from our initial assumption that there exists $k=k_0<0$ such that $a_{(\epsilon_1,\epsilon_2,i,j,k_0,l)}\neq 0.$  The rest of the proof will show that this assumption cannot also hold. 

Set $\underline{\upsilon}_0:=(i,j,k_0,l)\in \mathbb{N}^2\times \mathbb{Z}\times \mathbb{N}.$ From \eqref{pe15}, we have that
 \begin{align*}
 \text{ham}_{x_-}(x_3)t_5t_6=&- \displaystyle\sum_{(\epsilon_1,\epsilon_2, \underline{\upsilon}_0)\in I}3k a_{(\epsilon_1,\epsilon_2, \underline{\upsilon}_0)}x_{1}^{i}x_{2}^{j}x_3^{\epsilon_1}
 x_4^{\epsilon_2+1}t_5^{k_0}t_6^{l+1}+V,
 \end{align*}
 where $V\in \mathcal{J}_2 := \text{Span}\left( \mathfrak{P}\setminus 
 \{x_1^{i}x_2^{j}x_3^{\epsilon_1}x_4^{\epsilon_2}
 t_5^{k_0}t_6^{l}\mid \epsilon_1,\epsilon_2\in \{0,1\} \ \text{and} \ i,j,l\in \mathbb{N} \}\right).$ 
 It follows that
 \begin{align}
  \text{ham}_{x_-}(x_3)t_5t_6=\nonumber\\
  -\sum_{(\epsilon_1,\epsilon_2,\underline{v})\in I}& 3k_0 a_{(0,0,\underline{v})}x_{1}^{i}x_{2}^{j}
x_4t_5^{k_0}t_6^{l+1}-
 \sum_{(\epsilon_1,\epsilon_2,\underline{v})\in I}3k_0a_{(1,0,\underline{v})}x_{1}^{i}x_{2}^{j}x_3
x_4t_5^{k_0}t_6^{l+1}\nonumber\\
 -\sum_{(\epsilon_1,\epsilon_2,\underline{v})\in I}&3k_0 a_{(0,1, \underline{v})}x_{1}^{i}x_{2}^{j}
x_4^2t_5^{k_0}t_6^{l+1}-
\sum_{(\epsilon_1,\epsilon_2,\underline{v})\in I}3k_0 a_{(1,1, \underline{v})}x_{1}^{i}x_{2}^{j}x_3
 x_4^2t_5^{k_0}t_6^{l+1} +V.\label{pe18}
 \end{align}
 Write the relations in Lemma \ref{pl2}(2),(4) as follows:
\begin{align}
x_4^2=& \frac{2}{3}\beta- \frac{2}{3}x_2x_4x_6
+\frac{8}{9}\alpha x_3x_6+ \frac{4}{3}x_1x_3x_4x_6
+L_1,\label{pe19}\\
\nonumber\\
x_3x_4^2=&
 \frac{2}{3}\beta x_3- \frac{2}{3}x_2x_3
x_4x_6
+\frac{16}{9}\alpha^2 x_6+ \frac{16}{3}\alpha x_1x_4
x_6
+ \frac{8}{3}\beta x_1^2x_6
\nonumber\\&- \frac{8}{3}x_1^2
x_2x_4x_6^2
+ \frac{32}{9}\alpha x_1^2x_3x_6^2+
 \frac{16}{3}x_1^3x_3x_4x_6^2+
L_2,\label{pe20}
\end{align}
where $L_1$ and $L_2$ are some elements of the ideal  $\mathcal{A}_{\alpha,\beta}t_5\subseteq \mathcal{J}_2.$ Substitute \eqref{pe19} and \eqref{pe20} into  \eqref{pe18} and simplify to obtain:
\begin{align}
\text{ham}_{x_-}(e_3)t_5t_6
= \sum&[\lambda_{1,1}\beta a_{(0,1,i,j,k_0,l-1)}+\lambda_{1,2}\alpha^2a_{(1,1, i,j,k_0,l-2)}\nonumber\\&
+\lambda_{1,3}\beta a_{(1,1, i-2,j,k_0,l-2)}]x_1^ix_2^jt_5^{k_0}t_6^l\nonumber \\
+&\sum[\lambda_{2,1}\alpha a_{(0,1, i,j,k_0,l-2)}  
+\lambda_{2,2}\beta a_{(1,1, i,j,k_0,l-1)}\nonumber\\&+\lambda_{2,3}\alpha a_{(1,1,i-2,j,k_0,l-3)}]x_1^ix_2^jx_3t_5^{k_0}t_6^l\nonumber\\
 +&\sum[\lambda_{3,1}a_{(0,1, i,j-1,k_0,l-2)}  
+\lambda_{3,2}\alpha a_{(1,1, i-1,j,k_0,l-2)}\nonumber\\
                        &+\lambda_{3,3}a_{(1,1, i-2,j-1,k_0,l-3)}+\lambda_{3,4}a_{(0,0,i,j,k_0,l-1)}]x_1^ix_2^jx_4t_5^{k_0}t_6^l\nonumber\\
+&\sum[\lambda_{4,1}a_{(0,1,i-1,j,k_0,l-2)}  
+\lambda_{4,2}a_{(1,1,i,j-1,k_0,l-2)}
\nonumber\\
&+\lambda_{4,3}a_{(1,1,i-3,j,k_0,l-3)}+\lambda_{4,4}a_{(1,0,i,j,k_0,l-1)}]x_1^ix_2^jx_3x_4t_5^{k_0}t_6^l+V',
\label{pe21}
 \end{align}
where $V'\in \mathcal{J}_2$.  Also, $\lambda_{s,t}:=\lambda_{s,t}(j,k_0,l)$ are some families of scalars which are non-zero for all $s,t\in\{1,2,3,4\}$ and $j,l\in\mathbb{N}$, except 
 $\lambda_{1,4}$ and $\lambda_{2,4}$ which are assumed to be zero since they do not exist in the above expression.  Note that although each  $\lambda_{s,t}$ depends on $j,k_0,l,$ we have not made this dependency explicit in the above expression since the minimum requirement we need to complete the proof is for all the $\lambda_{s,t}$ existing in the above expression to be non-zero, which we already have.
   
Observe that \eqref{pe21} and \eqref{pc4e10} are equal, hence
\begin{align*}
\sum_{(\epsilon_1,\epsilon_2, \underline{w})\in J}b_{(\epsilon_1,\epsilon_2,\underline{w})}x_1^{i}
x_2^{j}x_3^{\epsilon_1}x_4^{\epsilon_2}
t_5^{k}t_6^{l}=
&\sum[\lambda_{1,1}\beta a_{(0,1,i,j,k_0,l-1)}+\lambda_{1,2}\alpha^2a_{(1,1, i,j,k_0,l-2)}\nonumber\\&
+\lambda_{1,3}\beta a_{(1,1, i-2,j,k_0,l-2)}]x_1^ix_2^jt_5^{k_0}t_6^l\nonumber \\
+&\sum[\lambda_{2,1}\alpha a_{(0,1, i,j,k_0,l-2)}  
+\lambda_{2,2}\beta a_{(1,1, i,j,k_0,l-1)}\nonumber\\&+\lambda_{2,3}\alpha a_{(1,1,i-2,j,k_0,l-3)}]x_1^ix_2^jx_3t_5^{k_0}t_6^l\nonumber\\
 +&\sum[\lambda_{3,1}a_{(0,1, i,j-1,k_0,l-2)}  
+\lambda_{3,2}\alpha a_{(1,1, i-1,j,k_0,l-2)}\nonumber\\
                        &+\lambda_{3,3}a_{(1,1, i-2,j-1,k_0,l-3)}+\lambda_{3,4}a_{(0,0,i,j,k_0,l-1)}]x_1^ix_2^jx_4t_5^{k_0}t_6^l\nonumber\\
+\sum&[\lambda_{4,1}a_{(0,1,i-1,j,k_0,l-2)}  
+\lambda_{4,2}a_{(1,1,i,j-1,k_0,l-2)}
\nonumber\\
+\lambda_{4,3}&a_{(1,1,i-3,j,k_0,l-3)}+\lambda_{4,4}a_{(1,0,i,j,k_0,l-1)}]x_1^ix_2^jx_3x_4t_5^{k_0}t_6^l+V'.
\end{align*}
We have previously established that $\left(x_1^{i}
x_2^{j}x_3^{\epsilon_1}x_4^{\epsilon_2}
t_5^{k}t_6^{l}\right)_{((\epsilon_1,\epsilon_2, \underline{v})\in \{0,1\}^2\times \mathbb{N}^2 \times \mathbb{Z}^2)}$ is a basis of $\mathcal{A}_{\alpha,\beta}[t_5^{-1}, t_6^{-1}]$ (note that $l\geq 0$ in this part of the proof). 

Since $\underline{v}_0=(i,j,k_0,l)\in \mathbb{N}^2\times \mathbb{Z}\times \mathbb{N}$ (with $k_0<0$) and $\underline{w}=(i,j,k,l)\in                      \mathbb{N}^4$ (with $k\geq 0$)  in the above equality, it follows  that
\begin{align}
\lambda_{1,1}&\beta a_{(0,1, i,j,k_0,l-1)}+\lambda_{1,2}\alpha^2 a_{(1,1, i,j,k_0,l-2)}
+\lambda_{1,3}\beta a_{(1,1, i-2,j,k_0,l-2)}=0,\label{pc4e15}\\
\lambda_{2,1}&\alpha a_{(0,1, i,j,k_0,l-2)}  
+\lambda_{2,2}\beta a_{(1,1, i,j,k_0,l-1)}+\lambda_{2,3}\alpha a_{(1,1, i-2,j,k_0,l-3)}
=0,\label{pc4e16}\\
\lambda_{3,1}&a_{(0,1, i,j-1,k_0,l-2)}  
+\lambda_{3,2}\alpha a_{(1,1, i-1,j,k_0,l-2)}+\lambda_{3,3}a_{(1,1, i-2,j-1,k_0,l-3)}\nonumber\\&+\lambda_{3,4}a_{(0,0, i,j,k_0,l-1)}=0,\label{pc4e17}\\
\lambda_{4,1}&a_{(0,1, i-1,j,k_0,l-2)}  
+\lambda_{4,2}a_{(1,1, i,j-1,k_0,l-2)}+\lambda_{4,3}a_{(1,1,i-3,j,k_0,l-3)}\nonumber\\
&+\lambda_{4,4}a_{(1,0, i,j,k_0,l-1)}=0.\label{pc4e18}
\end{align}   
From \eqref{pc4e15} and \eqref{pc4e16}, one can easily deduce that
\begin{align}
&a_{(0,1, i,j,k_0,l)}=-\frac{\alpha^2\lambda_{1,2}}{\beta\lambda_{1,1}}a_{(1,1, i,j,k_0,l-1)}
-\frac{\lambda_{1,3}}{\lambda_{1,1}}a_{(1,1, i-2,j,k_0,l-1)},\label{pc4e19}\\
&  
a_{(1,1, i,j,k_0,l)}=-\frac{\alpha\lambda_{2,1}}{\beta\lambda_{2,2}}a_{(0,1, i,j,k_0,l-1)}-\frac{\alpha\lambda_{2,3}}{\beta\lambda_{2,2}}a_{(1,1, i-2,j,k_0,l-2)}.\label{pc4e20}
\end{align}
Note that  $a_{(\epsilon_1,\epsilon_2, i,j,k_0,l)}:=0$ whenever $i<0$ or $j<0$ or $l<0$ for all $\epsilon_1,\epsilon_2\in \{0,1\}.$ 

\textbf{Claim.} The coefficients
$a_{(0,1, i,j,k_0,l)}$ and $a_{(1,1, i,j,k_0,l)}$ are all zero for all $l\geq 0$. We now justify the claim by an induction on $l.$ From \eqref{pc4e19} and \eqref{pc4e20}, the result is obviously true when $l=0.$ For $l\geq 0,$ assume that $a_{(0,1, i,j,k_0,l)}=a_{(1,1, i,j,k_0,l)}=0.$ Then, it follows from  \eqref{pc4e19} and \eqref{pc4e20} that
$$a_{(0,1, i,j,k_0,l+1)}=-\frac{\alpha^2\lambda_{1,2}}{\beta\lambda_{1,1}}a_{(1,1, i,j,k_0,l)}
-\frac{\lambda_{1,3}}{\lambda_{1,1}}a_{(1,1, i-2,j,k_0,l)},$$
$$a_{(1,1, i,j,k_0,l+1)}=-\frac{\alpha\lambda_{2,1}}{\beta\lambda_{2,2}}a_{(0,1,i,j,k_0,l)}-\frac{\alpha\lambda_{2,3}}{\beta\lambda_{2,2}}a_{(1,1, i-2,j,k_0,l-1)}.$$
From the inductive hypothesis, $a_{(1,1,i,j,k_0,l)}=a_{(1,1,i-2,j,k_0,l)}=a_{(0,1,i,j,k_0,l)}=a_{(1,1, i-2,j,k_0,l-1)}$ \newline $=0.$ 
Hence, $a_{(1,1, i,j,k_0,l+1)}=a_{(0,1,i,j,k_0,l+1)}=0.$  By the principle of mathematical induction,  $a_{(0,1, i,j,k_0,l)}=a_{(1,1,i,j,k_0,l)}=0$ for all $l\geq 0$ as desired.
Given that the families $a_{(0,1, i,j,k_0,l)}$ and $a_{(1,1, i,j,k_0,l)}$ are all zero, it follows from \eqref{pc4e17} and \eqref{pc4e18} that $a_{(0,0, i,j,k_0,l)}$ and 
$a_{(1,0, i,j,k_0,l)}$ are also zero for all $(i,j,k_0,l)\in \mathbb{N}^2\times\mathbb{Z}\times \mathbb{N}.$  This contradicts
our assumption. Hence, $x_-=0.$  Consequently, $x=x_+\in\mathcal{A}_{\alpha,\beta}$ as desired.

2. 
From Remark \ref{prrr}, we have that
$x_4=x_{4,6}+\frac{2}{3}t_5^3t_6^{-1}
=t_4+\frac{2}{3}t_5^3t_6^{-1}$. Again, from Lemma \ref{ev30}, we have that $\lambda_4=3\lambda_3+\lambda_5$ and $\lambda_6=-3\lambda_3.$ Therefore,

\begin{align*}
\rho(x_4)
&=\lambda_4t_{4}+\frac{2}{3}(3\lambda_5-\lambda_6)t_5^3
t_6^{-1}\\&=(3\lambda_3+\lambda_5)x_{4,6}+2(\lambda_3+\lambda_5)t_5^3
t_6^{-1}
\\
&=(3\lambda_3+\lambda_5)\left( x_4-\frac{2}{3}t_5^3t_6^{-1}\right) +2(\lambda_3+\lambda_5)t_5^3
t_6^{-1}\\
&=(3\lambda_3+\lambda_5)x_4+\frac{4}{3}\lambda_5t_5^3t_6^{-1}.
\end{align*}
Hence, 
\begin{align*}
\mathcal{D}(x_4)&=\text{ham}_x(x_4)+\rho(x_4)
=\text{ham}_x(x_4)+(3\lambda_3+\lambda_5)x_4+\frac{4}{3}\lambda_5t_5^3t_6^{-1}\in \mathcal{A}_{\alpha,\beta}.
\end{align*}
It follows that
$\lambda_5t_5^3t_6^{-1}\in \mathcal{A}_{\alpha,\beta},$ since 
$\text{ham}_x(x_4)+(3\lambda_3+\lambda_5)x_4\in \mathcal{A}_{\alpha,\beta}.$ Consequently,  $
\lambda_5t_5^3\in \mathcal{A}_{\alpha,\beta}t_6.$ Using the basis of $\mathcal{A}_{\alpha,\beta}$ (Prop. \ref{pc3p5}), we easily have that $\lambda_5=0$.  Therefore, $\rho(x_4)
=3\lambda_3x_4$ and $\rho(t_5)
=0.$ We already know from Lemma \ref{ev30} that $\rho(t_6)
=-3\lambda_3t_6.$   
From \eqref{pfe1}, we have
$\rho(x_3)=\lambda_3x_3+\frac{3}{2}(\lambda_4-\lambda_3-\lambda_5)x_4t_5^{-1}
+(3\lambda_5-\lambda_4-\lambda_6)t_5^2t_6^{-1}.$  Given that $\lambda_4=3\lambda_3, \ \lambda_5=0$ and $\lambda_6=-3\lambda_3,$ we have that $\rho(x_3)=\lambda_3x_3+3\lambda_3x_4t_5^{-1}.$
Now, $\mathcal{D}(x_3)=\text{ham}_x(x_3)+\rho(x_3)=\text{ham}_x(x_3)+\lambda_3x_3+3\lambda_3x_4t_5^{-1}
\in \mathcal{A}_{\alpha,\beta}.$ Observe that $\text{ham}_x(x_3), \lambda_3x_3\in \mathcal{A}_{\alpha,\beta}.$ Hence, $\lambda_3x_4t_5^{-1}
\in \mathcal{A}_{\alpha,\beta}$ which implies that $ \lambda_3x_4
\in \mathcal{A}_{\alpha,\beta}t_5.$ Similarly, $\lambda_3=0.$  We now have that
$\rho(x_3)=\rho(x_4)=\rho(x_5)=\rho(x_6)=0.$ We finish the proof by showing that 
$\rho(x_1)=\rho(x_2)=0.$ Recall from \eqref{pe2e} that
$$x_2x_4x_6-\frac{2}{3}x_3^3x_6-\frac{2}{3}x_2x_5^3+2x_3^2x_5^2-3
x_3x_4x_5+\frac{3}{2}x_4^2=\beta.$$ Apply $\rho$ to this relation to obtain $\rho(x_2)x_4x_6-\frac{2}{3}\rho(x_2)x_5^3=0.$ This implies that\newline
$\rho(x_2)\left( x_4x_6-\frac{2}{3}x_5^3\right)=0.$ Since $x_4x_6-\frac{2}{3}x_5^3\neq 0,$ it follows that $\rho(x_2)=0.$ Similarly, from \eqref{pe1e}, we have that
$$x_1x_3x_5-\frac{3}{2}x_1x_4-\frac{1}{2}x_2x_5+\frac{1}{2}x_3^2=\alpha.$$ Apply $\rho$ to this relation to obtain
$\rho(x_1)\left(x_3x_5-\frac{3}{2}x_4\right)=0.$ Since $x_3x_5-\frac{3}{2}x_4\neq 0,$ we must have: $\rho(x_1)=0.$ 
In conclusion, $\rho(x_\kappa)=0$ for all $\kappa\in \{1,\ldots,6\}.$ 

3. As a result of (1) and (2), we have that $\mathcal{D}(x_\kappa)=\text{ham}_x(x_\kappa).$ Consequently, $\mathcal{D}=\text{ham}_x$ as desired.
\end{proof}
\subsubsection{Poisson derivations of $\mathcal{A}_{\alpha, 0}$ and $\mathcal{A}_{0, \beta}$}
Following procedures similarly to the previous case (i.e. $\mathcal{A}_{\alpha,\beta}$ with $\alpha\beta\neq 0$), one can also compute the Poisson derivations of $\mathcal{A}_{\alpha,0}$ and $\mathcal{A}_{0,\beta}.$ The computations have been done, however, for the avoidance of redundancy, we are not going to include them here. We only summarize the results. Before we do that, we compute explicitly the scalar Poisson  derivations of $\mathcal{A}_{\alpha, 0}$ and $\mathcal{A}_{0,\beta}.$
\begin{lem}
\label{pev26}
Let $(\alpha,\beta)\in \mathbb{K}^2\setminus \{(0,0)\}.$ Suppose that $\vartheta$ and $\tilde{\vartheta}$ are linear maps  on $\mathcal{A}_{\alpha,0}$ and $\mathcal{A}_{0,\beta}$ respectively, and are defined by:
$$\vartheta(x_1)=-x_1, \ \ \vartheta(x_2)=-x_2, \ \ \vartheta(x_3)=0, \ \ \vartheta(x_4)=x_4, \ \ \vartheta(x_5)=x_5, \ \  \ \vartheta(x_6)=2x_6,$$
{and}
$$\tilde{\vartheta}(x_1)=-2x_1, \ \ \tilde{\vartheta}(x_2)=-3x_2, \ \ \tilde{\vartheta}(x_3)=-x_3, \ \ \tilde{\vartheta}(x_4)=0, \ \ \tilde{\vartheta}(x_5)=x_5, \ \  \ \tilde{\vartheta}(x_6)=3x_6.$$
Then,
 $\vartheta$  and $\tilde{\vartheta}$ extended to $\mathcal{A}_{\alpha,0}$
and $\mathcal{A}_{0,\beta}$ respectively using the Leibniz rule  are  $\mathbb{K}$-Poisson derivations of $\mathcal{A}_{\alpha,0}$
and $\mathcal{A}_{0,\beta}$ respectively.  
\end{lem}
\begin{proof}
We need to show that $\vartheta$ satisfies the following two relations (see \eqref{pe1e} and \eqref{pe2e}):
\begin{align*}
x_1x_3x_5&-\frac{3}{2}x_1x_4-\frac{1}{2}x_2x_5+\frac{1}{2}x_3^2=\alpha,\\
x_2x_4x_6&-\frac{2}{3}x_3^3x_6-\frac{2}{3}x_2x_5^3+2x_3^2x_5^2-3
x_3x_4x_5+\frac{3}{2}x_4^2=\beta,
\end{align*}
and the Poisson bracket of $\mathcal{A}_{\alpha,\beta}$ (see Subsection \ref{sec6.4}) when $\alpha\neq 0$ and $\beta=0,$ and do the same thing for 
 $\tilde{\vartheta}$ when $\alpha=0$ and $\beta\neq 0.$ We will only do this for  the relation $x_1x_3x_5-\frac{3}{2}x_1x_4-\frac{1}{2}x_2x_5+\frac{1}{2}x_3^2=\alpha$ and the Poisson bracket
 $\{x_6, x_2\}=3x_2x_6+9x_4-18x_3x_5$ in  $\mathcal{A}_{\alpha, 0}$, and leave the remaining ones for the reader to verify.  
 We have: 
  
\begin{align*}
\vartheta(x_1)x_3x_5+x_1\vartheta(x_3)x_5
+x_1x_3\vartheta(x_5)&-\frac{3}{2}[\vartheta(x_1)x_4+x_1\vartheta(x_4)]\\-\frac{1}{2}[\vartheta(x_2)x_5&+x_2\vartheta(x_5)]+\vartheta(x_3)x_3=0, 
\end{align*}
and
\begin{align*}
\vartheta(\{x_6, x_2\})&=\vartheta(3x_2x_6+9x_4-18x_3x_5)\\
&=3[\vartheta(x_2)x_6+x_2\vartheta
(x_6)]+9\vartheta(x_4)-18[\vartheta(x_3)x_5+x_3\vartheta(x_5)]\\
&=3(-x_2x_6+2x_2x_6)+9x_4-18x_3x_5\\
&=3x_2x_6+9x_4-18x_3x_5\\
&=\{x_6,x_2\}\\
&=2\{x_6,x_2\}-\{x_6,x_2\}\\
&=\{2x_6, x_2\}+\{x_6, -x_2\}\\
&=\{\vartheta(x_6), x_2\}+\{x_6, \vartheta(x_2)\}.
\end{align*}
\end{proof}
 
We summarize our main results in this section in the theorem below. Recall that HP$^1(\Delta)$ denote the first Poisson cohomology group of the Poisson algebra $\Delta$. 

\begin{thm}
Given 
$\mathcal{A}_{\alpha,\beta}=\mathbb{K}[X_1,\ldots, X_6]/\langle \Omega_1-\alpha,\Omega_2-\beta\rangle,$ with $(\alpha,\beta)\in \mathbb{K}^2\setminus \{(0,0)\},$  we have the following results:
\begin{itemize}
\item[1.]  if $\alpha, \beta\neq 0;$ then every Poisson derivation $\mathcal{D}$ of $\mathcal{A}_{\alpha,\beta}$ can uniquely be written as $\mathcal{D}=\mathrm{ham}_x,$ where $x\in\mathcal{ A}_{\alpha,\beta}.$
\item[2.]if $\alpha\neq 0$ and $\beta=0,$ then every Poisson derivation $\mathcal{D}$ of $\mathcal{A}_{\alpha,0}$ can uniquely be written as $\mathcal{D}=\mathrm{ham}_x+\lambda\vartheta,$ where $\lambda\in \mathbb{K}$ and $x\in \mathcal{A}_{\alpha,0}.$
 \item[3.] if $\alpha=0$ and $\beta\neq 0,$ then every Poisson derivation $\mathcal{D}$ of $\mathcal{A}_{0,\beta}$ can uniquely be written as $\mathcal{D}=\mathrm{ham}_x+\lambda\tilde{\vartheta},$ where $\lambda\in \mathbb{K}$ and $x\in \mathcal{A}_{0,\beta}.$
\item[4.] $HP^1(\mathcal{A}_{\alpha,0})=\mathbb{K}[\vartheta]$ and $HP^1(\mathcal{A}_{0,\beta})=\mathbb{K}[\tilde{\vartheta}],$ where
 $[\vartheta]$ and $[\tilde{\vartheta}]$ respectively denote the classes of $\vartheta$ and $\tilde{\vartheta}$  modulo the space of inner Poisson derivations.
\item[5.] if $\alpha,\beta\neq 0;$ then $HP^1(\mathcal{A}_{\alpha,\beta})=\{[0]\},$ where $[0]$ denotes the class of $0$   modulo the space of inner Poisson derivations.
\end{itemize}
 \end{thm}

Let $(\alpha, \beta)\in \mathbb{K}^2\setminus \{(0,0)\}$.
One can easily conclude that the first Poisson cohomology group HP$^1( \mathcal{A}_{\alpha,\beta})$ is isomorphic to the first Hochschild cohomology group HH$^1(A_{\alpha,\beta})$ studied in  \cite[Theorem 5.12]{lo}.  

It is natural to ask whether HP$^i( \mathcal{A}_{\alpha,\beta})$ is isomorphic to HH$^i(A_{\alpha,\beta})$ for all $i$.

%\section*{Acknowledgements} 

%It is our pleasure to thank Samuel Lopes for the fruitful discussions on the topic of this manuscript. We would also like to thank the referees for their helpful comments/suggestions which have helped to improve on some of the arguments in this paper, especially those in Section \ref{rec}, and also pointing out Corollary \ref{R}.

%%%%%%%%%%%%%%%%% END THE BIBLIOGRAPHY %%%%%%%%%%%%%

%%%%%%%%% ADDRESSES%%%%%%%%%%%%

\begin{minipage}{\textwidth}
%----------Author 1
\noindent S Launois \\
School of Mathematics, Statistics and Actuarial Science,\\
University of Kent\\
Canterbury, Kent, CT2 7FS,\\ UK\\[0.5ex]
email: S.Launois@kent.ac.uk \\

%----------Author 2
\noindent I Oppong\\
School of Mathematics, Statistics and Actuarial Science,\\
University of Kent\\
Canterbury, Kent, CT2 7FS,\\ UK\\[0.5ex]
email: I.Oppong@kent.ac.uk \\
\end{minipage}

%%%%%%%%% END OF ADDRESSES%%%%%%%%%%%%  

\begin{thebibliography}{10}


%
\bibitem{bkk}
A. ~Belov-Kanel and M. ~Kontsevich.
\newblock Automorphisms of the Weyl algebra.
\newblock { Letter in Mathematical
Physics}, 74:181--199, 2005. 
%

%
\bibitem{bg}
K.~A. Brown and K.~R. Goodearl.
\newblock { Lectures on Algebraic Quantum Groups}.
\newblock Advanced Courses in Mathematics CRM Barcelona (Birkh$\ddot{a}$user,
  Basel, 2002).
% 
 
%    
\bibitem{ca}
G.~Cauchon.
\newblock Effacement des d{\'e}rivations et spectres premiers des alg\`ebres
  quantiques.
\newblock { Journal of Algebra}, 260:476--518, 2003.
%

%
\bibitem{cho}
E.-H ~Cho and S.-Q. Oh.
\newblock Semiclassical limits of Ore extensions and a Poisson generalized Weyl algebra.
\newblock {Letter in Mathematical Physics}, 106(7):997--1009, 2016. 
%



%
\bibitem{dumas}
F. ~Dumas.
\newblock Rational equivalence for Poisson polynomial algebras.
\newblock { Lecture notes, December}, 2011.
 \newblock Available at {\tt https://lmbp.uca.fr/~fdumas/recherche.html}
%

%
\bibitem{fryer}
S. ~Fryer. 
\newblock The prime spectrum of quantum $SL_3$ and the Poisson prime spectrum of its semiclassical limit.
\newblock { Trans. London Maths. Soc.}, 4(1):1--29, 2017. 
%


%
\bibitem{gscl}
K. R. ~Goodearl.
\newblock Semiclassical limits of quantized coordinate rings.
\newblock { Advances in ring theory, Springer}, 165--204, 2010. 
%



\bibitem{gd}
K. R. ~Goodearl.
\newblock A Dixmier-Moeglin equivalence for Poisson algebras with torus actions.
\newblock { Contemporary Mathematics}, 419:131--154, 2006. 
%


%
\bibitem{gkls}
K. R. ~Goodearl and S. ~Launois.
\newblock The Dixmier-Moeglin equivalence and a Gel'fand-Kirillov problem for Poisson polynomial algebras.
\newblock Bulletin de la Soci{\'e}t{\'e} Math{\'e}matique de France, 139:1--39, 2011. 
%


%
\bibitem{kenL}
K. R. ~Goodearl and S. ~Letzter.
\newblock Semiclassical limits of quantum affine spaces.
\newblock Proc. Edinb. Math. Soc., 52:387--407, 2009. 
%



%
\bibitem{ak}
A. P. ~Kitchin.
\newblock Derivations of Quantum and Involution Generalized Weyl Algebras
\newblock {\em arXiv preprint arXiv:2107.14189}, 2021. 
%

%
\bibitem{sc}
S. ~Launois and C. ~Lecoutre.
\newblock Poisson deleting derivations algorithm and {P}oisson spectrum.
\newblock { Communications in Algebra}, 45:1294--1313, 2017. 
%

%
\bibitem{sltl}
S. ~Launois and T. H. ~Lenagan.
\newblock The first Hochschild cohomology group of quantum matrices and the quantum special linear group.
\newblock { Journal of Noncommutative Geometry}, 1:281--309, 2007. 
%

%
\bibitem{ss}
S. ~Launois and S. A. Lopes.
\newblock Automorphisms and derivations of $U_q(sl^+_4)$. 
\newblock{Journal of Pure and Applied Algebra},  211(1):249--264, 2007. 
%


%
\bibitem{lo}
S. ~Launois and I. ~Oppong.
\newblock Derivations of a family of quantum second Weyl algebras.
\newblock { Bulletin des Sciences Math\'ematiques}, 184: 103257, 2023.
%


%
\bibitem{fl}
F. ~Loose.
\newblock Symplectic algebras and Poisson algebras.
\newblock Communications in Algebra, 21:2395--2416, 1993. 
%

\bibitem{oh}
S.-Q. Oh.
\newblock Poisson polynomial rings.
\newblock Communications in Algebra, 21:2395--2416, 1993.

%
\bibitem{io}
I.~Oppong.
\newblock A Quantum Deformation of the Second Weyl Algebra: Its derivations and Poisson derivations.
\newblock PhD Thesis, University of Kent, 2021.
\newblock Available at {\tt https://kar.kent.ac.uk/92766/}


%
\bibitem{op}
J.~M. Osborn and D.~Passman.
\newblock Derivations of skew polynomial rings.
\newblock { Journal of Algebra}, 34:1265--1277, 2006.
%

%
\bibitem{xt}
X. ~Tang.
\newblock Derivations of the Two-Parameter Quantized Enveloping Algebra.
\newblock { Communications in Algebra}, 41:4602--4621, 2013. 
%

%
\bibitem{zt}
Y. Y.  ~Zhong and X. M. ~Tang.
\newblock Derivations of the Positive Part of the Two-parameter Quantum Group of Type G2.
\newblock { Acta Mathematica Sinica, English Series}, 37:1471--1484, 2021. 
%

\end{thebibliography}
\end{document}